\documentclass[11pt]{amsart}
\usepackage[utf8]{inputenc}
\usepackage[T1]{fontenc}
\usepackage{lmodern}
\usepackage{microtype}
\usepackage[leqno]{amsmath}
\usepackage{amssymb}
\usepackage{mathtools}
\usepackage{enumerate}
\usepackage{graphicx}
\usepackage[usenames,dvipsnames]{xcolor}
\usepackage[
colorlinks=true,linkcolor=NavyBlue,urlcolor=RoyalBlue,citecolor=PineGreen,bookmarks=true,bookmarksdepth=3,bookmarksopen=true,bookmarksopenlevel=2,unicode=true,linktocpage]{hyperref}

\numberwithin{equation}{section}

\newtheorem{theorem}{Theorem}[section]

\newtheorem{lemma}[theorem]{Lemma}
\newtheorem{proposition}[theorem]{Proposition}
\newtheorem{corollary}[theorem]{Corollary}

\newtheorem{definition}[theorem]{Definition}

\let\C\relax

\newcommand{\C}{\mathbf{C}}
\newcommand{\D}{\mathbf{D}}

\newcommand{\h}{\mathbf{H}}
\newcommand{\N}{\mathbf{N}}
\newcommand{\Z}{\mathbf{Z}}
\newcommand{\p}{\mathbf{P}}
\newcommand{\Q}{\mathbf{Q}}
\newcommand{\R}{\mathbf{R}}

\newcommand{\Fh}{\mathfrak {h}}
\newcommand{\Fg}{\mathfrak {g}}

\newcommand{\CC}{\mathcal {C}}
\newcommand{\CD}{\mathcal {D}}
\newcommand{\CE}{\mathcal {E}}
\newcommand{\CF}{\mathcal {F}}
\newcommand{\CI}{\mathcal {I}}

\newcommand{\CK}{\mathcal {K}}
\newcommand{\CL}{\mathcal {L}}

\newcommand{\CG}{\mathcal {G}}
\newcommand{\CH}{\mathcal {H}}

\newcommand{\SLE}{{\rm SLE}}
\newcommand{\CLE}{{\rm CLE}}

\newcommand{\dist}{\operatorname{dist}}

\newcommand{\im}{\operatorname{Im}}

\newcommand{\one}{{\bf 1}}

\newcommand{\wt}{\widetilde}
\newcommand{\wh}{\widehat}
\newcommand{\ol}{\overline}
\newcommand{\ul}{\underline}

\newcommand{\giv}{\,|\,}

\DeclarePairedDelimiter\abs{\lvert}{\rvert}

\newcommand*{\defeq}{\mathrel{\mathop:}=}

\newcommand*{\mmiddle}[1]{\mathrel{}\middle#1\mathrel{}}

\newcommand*{\pv}{\mathrm{p.v.}}

\newcommand*{\sle}[1]{$\SLE_{#1}$}
\newcommand*{\slek}{\sle{\kappa}}
\newcommand*{\slekp}{\sle{\kappa'}}
\newcommand*{\slekr}[1]{$\SLE_{\kappa}(#1)$}
\newcommand*{\slekpr}[1]{$\SLE_{\kappa'}(#1)$}
\newcommand*{\slekrb}[2]{$\SLE_{\kappa}^{#2}(#1)$}
\newcommand*{\slefmu}[1]{$\SLE_{4}^{0,#1}(-2)$}
\newcommand*{\cle}[1]{$\CLE_{#1}$}
\newcommand*{\clek}{\cle{\kappa}}
\newcommand*{\clekp}{\cle{\kappa'}}

\newcommand*{\bcle}[2]{$\mathrm{BCLE}_{#1}(#2)$}
\newcommand*{\bclek}[1]{\bcle{\kappa}{#1}}
\newcommand*{\bclekp}[1]{\bcle{\kappa'}{#1}}

\newcommand*{\rmin}{{\mathrm{in}}}
\newcommand*{\rmout}{{\mathrm{out}}}

\newcommand{\rev}{{\mathrm{rev}}}

\newcommand{\cleconfclass}[1]{{\mathfrak D}_{#1}^{\mathrm{ext},\CLE_{\kappa}}}
\newcommand{\extadjacent}{\beta_{\cap \cap}}
\newcommand{\extnested}{\beta_{\Cap}}
\newcommand{\intadjacent}{\alpha_{\cap \cap}}
\newcommand{\intnested}{\alpha_{\Cap}}

\newcommand{\outside}{{\mathrm{out}}}
\newcommand{\inside}{{\mathrm{in}}}

\newcommand{\separated}{{\mathrm{sep}}}
\newcommand{\resampled}{{\mathrm{res}}}

\newcommand{\markeddomain}[1]{{\mathfrak D}_{#1}}
\newcommand{\eldomain}[1]{{\mathfrak D}_{#1}^{\mathrm{ext}}}
\newcommand{\ildomain}[1]{{\mathfrak D}_{#1}^{\mathrm{int}}}

\newcommand{\mcclelaw}[1]{\p_{(#1)}^{\CLE_{\kappa}}}

\newcommand{\mcslelaw}[1]{\p_{(#1)}^{\SLE_{\kappa}}}

\newcommand{\domainpair}[1]{{\mathfrak {P}}_{#1}}

\newcommand*{\linkpatterns}[1]{\mathrm{LP}_{#1}}
\newcommand*{\Ncurves}[1]{\mathfrak{K}_{(#1)}}

\title{Multiple SLE$_\kappa$ from CLE$_\kappa$}

\usepackage[foot]{amsaddr}
\author{Valeria Ambrosio}
\author{Jason Miller}
\author{Yizheng Yuan}
\address{Department of Pure Mathematics and Mathematical Statistics, University of Cambridge}

\begin{document}

\begin{abstract}
We introduce multichordal CLE$_\kappa$ which is a random collection of non-crossing loops together with additional chords connecting a set of marked boundary points. The chords have a random link pattern, and their law conditionally on the link pattern is a (global) multichordal SLE$_\kappa$. We show that multichordal CLE$_\kappa$ arises as the conditional law of the remainder of a partially explored CLE$_\kappa$. The multichordal CLE$_\kappa$ are the conjectural scaling limits of FK and loop $O(n)$ models with some wiring patterns of the boundary arcs.

We further explain how CLE$_\kappa$ configurations can be locally resampled, and show that the partially explored strands can be relinked in any possible way with positive probability. Our results also establish a useful local independence property of CLE$_\kappa$. Altogether, the results and estimates in this paper serve to provide a toolbox for studying CLE$_\kappa$ and global multiple SLE$_\kappa$.
\end{abstract}

\date{\today}
\maketitle

\setcounter{tocdepth}{1}
\tableofcontents

\section{Introduction}
\label{sec:intro}

\subsection{Overview}
\label{subsec:overview}

The Schramm-Loewner evolution (SLE) \cite{s2000sle} is the canonical model for a random conformally invariant planar curve which lives in a simply connected domain.  SLE was introduced by Schramm in 1999 as a candidate to describe the scaling limit of loop-erased random walk in two dimensions and since has been conjectured to describe the scaling limit of the interfaces which arise in models from statistical mechanics on planar lattices and such conjectures have now been proved in a number of cases \cite{s2001percolation,lsw2004lerw,ss2009dgff,s2010ising}.  The SLE curves are indexed by a parameter $\kappa \geq 0$ (and denoted by \slek{}) which determines the roughness of the curve.  An \sle{0} is a smooth curve and for $\kappa > 0$ is fractal.  It is a simple curve for $\kappa \in (0,4]$, self-intersecting but not space-filling for $\kappa \in (4,8)$, and space-filling for $\kappa \geq 8$ \cite{rs2005basic}.  Furthermore, the dimension of the range of an \slek{} curve is $1+\kappa/8$ for $\kappa \leq 8$ and $2$ for $\kappa \geq 8$ \cite{rs2005basic,b2008dimension}.  Schramm defined \slek{} in terms of the Loewner equation and it is not immediate from its definition that it in fact corresponds to a continuous curve.  This was proved by Rohde-Schramm for $\kappa \neq 8$ in \cite{rs2005basic} and by Lawler-Schramm-Werner for $\kappa = 8$ \cite{lsw2004lerw} as a consequence of the convergence of the uniform spanning tree Peano curve to \sle{8} (see also \cite{am2022sle8} for a proof of this result using only continuum methods).

We emphasize that an \slek{} curve describes the scaling limit of a single interface from a discrete lattice model.  The object which arises as the scaling limit of all of the interfaces simultaneously is the so-called \emph{conformal loop ensemble} (\clek{}) \cite{s2009cle,sw2012cle}.  The \clek{}'s are defined for $\kappa \in (8/3,8)$ and consist of a countable collection of loops, each of which locally looks like an \slek{} curve.  At the extremes $\kappa=8/3$, (resp.\ $\kappa=8$) a \clek{} consists of the empty collection of loops (resp.\ a single space-filling loop).  The \clek{}'s have the same phases as \slek{}.  In particular, the loops are simple, do not intersect each other, or the domain boundary for $\kappa \in (8/3,4]$, and are self-intersecting, intersect each other, and the domain boundary for $\kappa \in (4,8)$.  On planar lattices, convergence results towards \clek{} have now been proved in a number of cases \cite{ks2019fkising,cn2006cle,bh2019ising,lsw2004lerw}.

We remark that a number of other models have been shown to converge to SLE and CLE on random planar maps \cite{s2016hc,kmsw2019bipolar,lsw2017schnyder,gm2021saw,gm2021percolation} using the framework developed in \cite{s2016zipper,dms2021mating}.

So-called \emph{multiple \slek{}} arise when one considers the scaling limit of a discrete model with alternating boundary conditions and then jointly explores the corresponding interfaces. It has been studied from the point of view of commutation relations and partition functions by a number of authors including \cite{dub-commutation,kl-msle,pw-msle-simple,bpw-msle-uniqueness} in the regime $\kappa\in(0,4]$, and more recently \cite{wu-hypergeometric-sle,zhan-msle,ahsy-welding-msle,flpw-msle} in the regime $\kappa\in(4,8)$. We refer to these works for more details and references regarding this perspective.

Multiple \slek{} arises in the scaling limit in the setting where one starts out with constant boundary conditions and then creates alternating boundary conditions by partially exploring some of the associated loops. From this perspective, it is natural that multiple \slek{} arises when one considers the continuum version of this setting, namely by starting out with a \clek{} and then partially exploring its loops. This perspective is explored, for instance, in \cite{msw2020nonsimple,mw2018connection,gmq2021sphere}. The purpose of the present article is to build on these works and formulate these results in a systematic manner which can be used as a convenient toolbox for other works.

\subsection{Main results}
\label{subsec:main_results}

We now turn to state the main results of this article and begin with a few definitions.  We will first review the definition of a marked domain and with the additional structure of an interior or exterior link pattern.

\begin{definition}
\label{def:marked_domain}
Let $D \subsetneq \C$ be a simply connected domain.  Fix $N \in \N_0$ and let $\ul{x} = (x_1,\ldots,x_{2N})$ be a collection of distinct prime ends which are ordered counterclockwise.
\begin{enumerate}[(i)]
\item We call $(D;\ul{x})$ a \emph{marked domain}. Let $\markeddomain{2N}$ be the collection of marked domains with $2N$ marked points.
\item For a marked domain $(D;\ul{x}) \in \markeddomain{2N}$, we can consider planar link patterns $\beta$ on the exterior of $D$, i.e.\ partitions of the marked points into pairs $\beta = \{\{a_1,b_1\},\ldots,\{a_N,b_N\}\}$ where $\{a_1,b_1,\ldots,a_N,b_N\} = \{1,\ldots,2N\}$, $a_r<b_r$, and such that the configuration $a_r<a_s<b_r<b_s$ does not occur. That is, each pair can be linked by a path on the exterior of $D$ without any two paths intersecting. We will refer to $(D;\ul{x};\beta)$ as a link pattern decorated marked domain. We let $\eldomain{2N}$ be the collection of exterior link pattern decorated marked domains.
\item For a marked domain $(D;\ul{x}) \in \markeddomain{2N}$, we also consider planar link patterns $\alpha$ in the interior of~$D$, defined in the same way as above. We let $\ildomain{2N}$ be the collection of interior link pattern decorated marked domains.
\end{enumerate}
\end{definition}

We let $\linkpatterns{N}$ denote the set of planar link patterns for $2N$ points (consisting of $N$ links). With a slight abuse of notation, we sometimes write $b_r = \alpha(a_r)$ if $\{a_r,b_r\} \in \alpha$.

The exterior link patterns can equivalently be viewed as wiring patterns of the boundary arcs. Let us say that for each $1 \leq i \leq N$ the arcs of $\partial D$ between $x_{2i-1}$ and $x_{2i}$ are \emph{wired}; the other arcs of $\partial D$ are \emph{free}. In case $N=0$, we say the boundary is wired. Then the exterior link patterns $\beta$ can be equivalently described by a planar partition of the wired arcs; see \cite{lpw-ust}. We say that the wired arcs belonging to the same partition are wired together.

\subsubsection{Existence and uniqueness of multiple \slek{} and \clek{}}  
\label{se:main_msle}

We will now state our results regarding the existence and uniqueness of multiple \slek{}.  We will consider first the case in which we have a fixed interior link pattern which agrees with the notion of global multiple SLE in the literature \cite{pw-msle-simple,bpw-msle-uniqueness,zhan-msle,ahsy-welding-msle,flpw-msle}. The present article provides an independent proof of its existence. Then we will consider multiple \slek{} with a random link pattern which will lead to our definition of multichordal \clek{}.

In both cases, the respective multiple \slek{} is described as an invariant measure under a resampling operation.  In the case that the link pattern is fixed, the resampling operation changes a single chord at a time, whereas when the resampling operation is not fixed, pairs of chords are resampled at a time.  Let us now give the formal definition of the former.

For $(D;\ul{x};\alpha) \in \ildomain{2N}$, let $\Ncurves{D;\ul{x};\alpha}$ be the set of $N$-tuples of curves $\ul{\eta}=(\gamma_1,\dots,\gamma_N)$ that connect the points in $\ul{x}$ according to $\alpha$. More precisely, for each $i=1,\ldots,N$, the points $x_{a_i},x_{b_i}$ are on the boundary of a connected component $D_i$ of $D \setminus \bigcup_{j\neq i}\gamma_j$, and the curve $\gamma_i$ connects $x_{a_i},x_{b_i}$ and is contained in $\ol{D_i}$. We let $\Ncurves{D;\ul{x}} = \bigcup_{\alpha\in\linkpatterns{N}}\Ncurves{D;\ul{x};\alpha}$. For $\ul{\eta} \in \Ncurves{D;\ul{x}}$, we denote by $\eta_j$ the chord of  $\ul{\eta}$ emanating from $x_j$ for each $j=1,\dots,2N$. That is, if $\gamma_i$ connects $x_{a_i},x_{b_i}$, then $\eta_{a_i}=\gamma_i$, $\eta_{b_i}=\gamma_i^\rev$ where $\gamma^\rev$ is the time reversal of $\gamma$. We denote by $\ul\eta_{\wh{a_i}} = \ul\eta_{\wh{b_i}}$ the tuple formed by $(\gamma_j)_{j\neq i}$. 

\begin{definition}
\label{def:interior_link_pattern_multiple_sle}
Suppose that $(D;\ul{x};\alpha) \in \ildomain{2N}$. We call a probability measure on $\Ncurves{D;\ul{x};\alpha}$ a multichordal \slek{} with interior link pattern $\alpha$ if for each $\{a_i,b_i\} \in \alpha$ its law is invariant under the operation of resampling the curve connecting $a_i,b_i$ from the law of an \slek{} in $D_i$.
\end{definition}

The following theorem shows the existence and uniqueness of the multichordal \slek{} with given interior link patterns. This result has been known: The uniqueness part is rather classical and follows from a Markov chain mixing argument, see e.g.\ \cite{ms2016ig2,msw2020nonsimple,bpw-msle-uniqueness,zhan-msle}. The existence has also been proved (in various regimes, respectively) in \cite{kl-msle,pw-msle-simple,zhan-msle,ahsy-welding-msle,flpw-msle} using different methods. Our proof is independent of the latter works on the existence.

\begin{theorem}
\label{thm:multichordal_link}
The following hold for each $\kappa \in (8/3,8)$.
\begin{enumerate}[(i)]
\item For each $(D;\ul{x};\alpha) \in \ildomain{2N}$ there exists a unique probability  measure $\mcslelaw{D;\ul{x};\alpha}$ satisfying Definition~\ref{def:interior_link_pattern_multiple_sle}. 
\item The probability measures $(\mcslelaw{D;\ul{x};\alpha})$ depend only on the conformal class of $(D;\ul{x};\alpha)$.  That is, if $(\wt{D};\ul{\wt{x}};\alpha) \in \ildomain{2N}$ and $\varphi \colon D \to \wt{D}$ is a conformal transformation taking $x_i$ to $\wt{x}_i$ for each $1 \leq i \leq 2N$ and $\ul{\eta} \sim \mcslelaw{D;\ul{x};\alpha}$, then $\varphi(\ul{\eta}) \sim \mcslelaw{\wt{D};\ul{\wt{x}};\alpha}$.
\item The family of probability measures $(\mcslelaw{D;\ul{x};\alpha})$ is Markovian.  That is, suppose that $(D;\ul{x};\alpha) \in \ildomain{2N}$ and $\ul{\eta} \sim \mcslelaw{D;\ul{x};\alpha}$. Let $i \in \{1,\ldots,2N\}$, let $\tau$ be a stopping time for $\CF_t = \sigma(\eta_i(s) : s \leq t)$, and let $(D_\tau; \ul{x}_\tau; \alpha_\tau)$ be the triple consisting of the component of $D \setminus \eta_i([0,\tau])$ with $\eta_i(\tau)$ on its boundary, the elements of $\ul{x} \setminus \{x_i\}$ together with $\eta_i(\tau)$ that are on $\partial D_\tau$, and the induced interior link pattern $\alpha_\tau$.  Then given $\CF_\tau$, the conditional law of $\eta_i|_{[\tau,\infty)}$ and the chords $(\eta_j)_{j \neq i}$ lying in the component $D_\tau$ is that of $\mcslelaw{D_\tau;\ul{x}_\tau;\alpha_\tau}$.
\end{enumerate}
\end{theorem}

We now give the definition of a multichordal \slek{} on an exterior link pattern decorated marked domain. This is a probability measure on $N$ curves with random interior link pattern $\alpha$, and the law of $\alpha$ depends on the exterior link pattern $\beta$.

We will see that these laws arise naturally from partially exploring strands of a \clek{}. (Throughout the article, we will refer to segments of the \clek{} loops as well as the chords in a multichordal \slek{} as strands.) Therefore it is natural to further sample \clek{}'s in the complementary components of the strands. We will call the resulting collection of random strands and loops a multichordal \clek{}. \emph{All CLE's in this paper are assumed to be nested.}

We first give the definitions of multichordal \slek{} with exterior link pattern in the cases $N=0,1,2$. In the case $N=0$, the multichordal \slek{} with $0$ curves is just the empty set. In the case $N=1$, the exterior link pattern is trivial, and the multichordal \slek{} with $1$ curve is the chordal \slek{}. In the case $N=2$, we define the multichordal \slek{} as the law of the pair of curves given the partial exploration of a \clek{} considered in \cite{msw2020nonsimple,mw2018connection}. More precisely, for $L>0$ let
\begin{equation}\label{eq:bichordal_link_probability}
 H_\kappa(L) = \frac{Y_\kappa(L)}{Y_\kappa(L)-2\cos(4\pi/\kappa)Y_\kappa(1/L)}
\end{equation}
where $Y_\kappa(L) = r(L)^{2/\kappa}(1-r(L))^{1-6/\kappa}{_2F_1}(4/\kappa,1-4/\kappa,8/\kappa;r(L))$ and $r(L) \in (0,1)$ is such that the conformal modulus of $(\h;0,1-r(L),1,\infty)$ is $L$, i.e.\ it is conformally equivalent to $((0,L)\times(0,1);0,L,L+i,i)$; see \cite{mw2018connection}. 
The bichordal \slek{} on the exterior link pattern $\extnested=\{ \{1,4\},\{2,3\}\}$ (resp.\ $\extadjacent=\{ \{1,2\},\{3,4\}\}$) is given by sampling the interior link pattern $\alpha$ so that the probability of $\intadjacent=\{ \{1,2\},\{3,4\} \}$ (resp.\ $\intnested=\{ \{1,4\},\{2,3\} \}$) is $H_\kappa(L)$ (resp.\ $H_\kappa(1/L)$) where $L$ is the conformal modulus of $(D;\ul{x})$, and then sampling a bichordal \slek{} with interior link pattern $\alpha$ (its existence is recalled in Section~\ref{subsec:bichordal_cle}).

\begin{figure}[ht]
\includegraphics[width=0.45\textwidth]{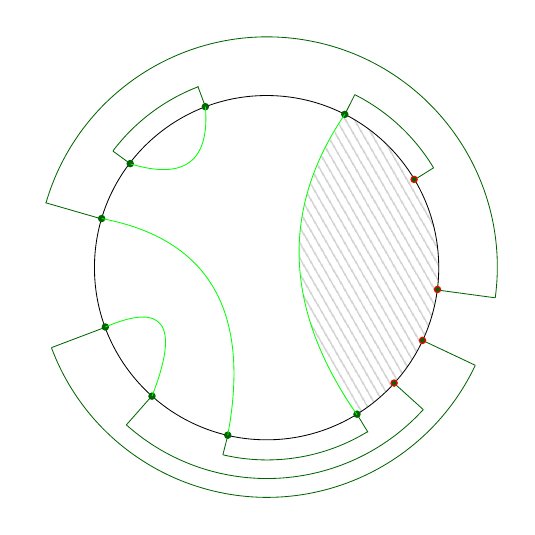}
\caption{\label{fig:extlink_induced}
Illustration of the exterior link pattern induced by a collection of curves. The initial configuration $(\D; \ul{x}; \beta) \in \eldomain{12}$ consists of marked points $\ul{x}$  (shown as dots) and a link pattern $\beta$, represented by dark green chords on the exterior of $\D$. The collection of curves $\{\gamma_i\}_{i \in \mathcal{I}'}$ is drawn in light green, the component $U$ of interest is dashed in grey, and the points in $\ul{x} \cap \partial U$ that are not endpoints of the curves $\{\gamma_i\}_{i\in \CI'}$ are circled in red. The exterior link pattern $\beta_U$ for $U$ that is induced by $\beta$ and the curves $\{\gamma_i\}_{i \in \mathcal{I}'}$ describes the connectivity among the points in $\ul{x} \cap \partial U$ that are not endpoints of the curves $\{\gamma_i\}_{i \in \mathcal{I}'}$. This pattern is obtained by tracing the connections in $\beta$ and following along the curves $\{\gamma_i\}_{i \in \mathcal{I}'}$ as necessary. In the configuration shown, the resulting pattern is $\beta_U = \{\{1,2\}, \{3,4\}\}$ indexing from the first counterclockwise marked point after $-i$.}
\end{figure}

We will shortly give the definition of the multichordal resampling kernel which selects four of the marked boundary points and resamples the pair of curves according to the bichordal \slek{} law. In order to do this, we need to consider the exterior link pattern that is \emph{induced by an initial link pattern~$\beta$ and a collection of (already sampled) curves}. Let $(D;\ul{x};\beta) \in \eldomain{2N}$ and $(\gamma_1,\dots,\gamma_N) \in \Ncurves{D;\ul{x}}$. Consider a subset $\CI'\subseteq \{1,\ldots,N\}$ and a component $U$ of $D\setminus \bigcup_{i\in \CI'}\gamma_i$ such that $\CI'$ contains every index $j$ for which $\gamma_j$ is disconnected from $U$ by $\bigcup_{i\in \CI'}\gamma_i$. The concatenation of $\beta$ with the curves $\{\gamma_i\}_{i \in \CI'}$ yields an exterior link pattern $\beta_U$ of the set of marked points of $\ul{x} \cap \partial U$ that are not endpoints of the curves $\{\gamma_i\}_{i \in \mathcal{I}'}$.\footnote{The induced exterior link pattern $\beta_U$ can be formally defined by the following iterative procedure.
\begin{itemize}
\item Initialise with $\beta_0 \coloneqq \beta$.
\item For each $j\in\CI'$, we update according  to the endpoints  $x_{b}, x_{c}$ of $\gamma_j$:
  \begin{itemize}
  \item If $\{b,c\}\in\beta_{n-1}$,  we remove this link and set $\beta_n\coloneqq\beta_{n-1}\setminus\{\{b,c\}\}$.
  \item Otherwise for $a,d$   such that $\{a,b\},\{c,d\}\in \beta_{n-1}$, we set $\beta_n\coloneqq(\beta_{n-1}\cup\{\{a,d\}\})\setminus \{\{a,b\},\{c,d\}\}$.
  \end{itemize}
\item After processing all $j\in\CI'$, reindex the remaining points of $\ul{x}\cap \partial U$ (that are not the endpoints of any of the curves $\{\gamma_i\}_{i\in \CI'}$) as $\{1,\ldots,2(N-\abs{\CI'})\}$.
\end{itemize}}
See Figure~\ref{fig:extlink_induced} for an illustration.

\begin{definition}[multichordal resampling kernel]\label{def:resampling_kernel}
Suppose that $(D;\ul{x};\beta) \in \eldomain{2N}$ where $N\geq 2$. For $1 \le j_1 < j_2 < j_3 < j_4 \le 2N$, consider the following resampling operation: Let $\ul{\eta}=(\gamma_1,\dots,\gamma_N) \in \Ncurves{D;\ul{x}}$. Let $E_{\ul{j}}$ be the event that there are exactly two curves $\gamma_{i},\gamma_{i'}$ whose endpoints are the four points in $x_{\ul{j}} = (x_{j_1},x_{j_2},x_{j_3},x_{j_4})$ and they lie on the boundary of the same connected component $D_{\ul{j}}$ of $D \setminus \bigcup_{\ell\neq i,i'} \gamma_\ell$. Let $\beta_{\ul{j}}$ be the link pattern induced by $\beta$ and $(\gamma_\ell)_{\ell\neq i,i'}$. On the event $E_{\ul{j}}$, we leave the other curves $(\gamma_\ell)_{\ell\neq i,i'}$ unchanged and resample the pair of curves connecting $x_{\ul{j}}$ according to the bichordal \slek{} law on $(D_{\ul{j}};x_{\ul{j}};\beta_{\ul{j}})$ described above. On the event $E_{\ul{j}}^c$, we leave $\ul{\eta}$ unchanged.
\end{definition}

\begin{definition}[multichordal \slek{} and \clek{}]
\label{def:exterior_link_pattern_multiple_sle}
Suppose that $(D;\ul{x};\beta) \in \eldomain{2N}$ where $N\geq 2$.  We call a probability measure on $\Ncurves{D;\ul{x}}$ a multichordal \slek{} with exterior link pattern $\beta$ if its law is invariant under the resampling operation from Definition~\ref{def:resampling_kernel} for any choice of $1 \le j_1 < j_2 < j_3 < j_4 \le 2N$. 
We call  $(\ul\eta,\Gamma)$ a multichordal \clek{}  if $\ul\eta$ is a multichordal \slek{}  in $(D;\ul{x};\beta)$ and $\Gamma$ is given by independent \clek{}'s in each connected component of  $D\setminus \ul\eta$.
\end{definition}

We will refer to multichordal \clek{} in the cases $N=0,1,2$ as \clek{}, monochordal \clek{}, and bichordal \clek{}, respectively. We will denote a multichordal \clek{} either by $(\ul{\eta},\Gamma) \sim \mcclelaw{D;\ul{x};\beta}$ or just by $\Gamma \sim \mcclelaw{D;\ul{x};\beta}$ (which then implicitly contains the chords $\ul{\eta}$).

From now on, we will refer to the multichordal \slek{} given an exterior link pattern also as a multichordal \clek{}. This is in order to avoid confusion with the multichordal \slek{} given an interior link pattern which is usually referred to in the literature.

\begin{theorem}
\label{thm:multichordal}
The following hold for each $\kappa \in (8/3,8)$ and $N \ge 2$.
\begin{enumerate}[(i)]
\item For each $(D;\ul{x};\beta) \in \eldomain{2N}$ there exists a unique probability measure $\mcclelaw{D;\ul{x};\beta}$ satisfying Definition~\ref{def:exterior_link_pattern_multiple_sle}. 
\item The probability measures $(\mcclelaw{D;\ul{x};\beta})$ depend only on the conformal class of $(D;\ul{x};\beta)$.  That is, if $(\wt{D};\ul{\wt{x}};\beta) \in \eldomain{2N}$ and $\varphi \colon D \to \wt{D}$ is a conformal transformation taking $x_i$ to $\wt{x}_i$ for each $1 \leq i \leq 2N$ and $(\ul{\eta},\Gamma) \sim \mcclelaw{D;\ul{x};\beta}$, then $(\varphi(\ul{\eta}),\varphi(\Gamma)) \sim \mcclelaw{\wt{D};\ul{\wt{x}};\beta}$.
\item The family of probability measures $(\mcclelaw{D;\ul{x};\beta})$ is Markovian.  That is, suppose that $(D;\ul{x};\beta) \in \eldomain{2N}$ and $(\ul{\eta},\Gamma) \sim \mcclelaw{D;\ul{x};\beta}$. Let $1 \leq i \leq 2N$, let $\tau$ be a stopping time for $\CF_t = (\eta_i(s) : s \leq t)$, and let $(D_\tau; \ul{x}_\tau; \beta_\tau)$ be the triple consisting of the component of $D \setminus \eta_i([0,\tau])$ with $\eta_i(\tau)$ on its boundary, the elements of $\ul{x} \setminus \{x_i\}$ together with $\eta_i(\tau)$ that are on $\partial D_\tau$, and the induced exterior link pattern $\beta_\tau$.  Then given $\CF_\tau$ and the elements of $(\eta_j)_{j \neq i}$ and $\Gamma$ lying outside the component~$D_\tau$, the conditional law of $\eta_i|_{[\tau,\infty)}$ and the elements of $(\eta_j)_{j \neq i}$ and $\Gamma$ lying in the component~$D_\tau$ is that of $\mcclelaw{D_\tau;\ul{x}_\tau;\beta_\tau}$.
\item\label{it:multsle_rel} Suppose that $\alpha$ is any interior link pattern.  Then the $\mcclelaw{D;\ul{x};\beta}$-probability that the interior link pattern induced by $\ul{\eta} \sim \mcclelaw{D;\ul{x};\beta}$ is equal to $\alpha$ is positive and is continuous in the marked point configuration $\ul{x}$.  The conditional law of $\ul{\eta}$ given this is equal to $\mcslelaw{D;\ul{x};\alpha}$ from Theorem~\ref{thm:multichordal_link}.
\end{enumerate}
\end{theorem}

By Theorem~\ref{thm:multichordal}\eqref{it:multsle_rel}, the statements of Theorem~\ref{thm:multichordal_link} (except for the uniqueness) follow immediately from Theorem~\ref{thm:multichordal}. The uniqueness in Theorem~\ref{thm:multichordal_link} follows from the same argument in Section~\ref{sec:multichordal_ex_uniq} used to prove the uniqueness in Theorem~\ref{thm:multichordal}.

We note that the multichordal \clek{} are the conjectural scaling limits of FK models and loop $O(n)$ models where the boundary arcs are wired according to the wiring pattern associated with $\beta$. The conjectural asymptotic probabilities of the linking patterns in these models are stated in \cite{fpw-connection-prob}. We conjecture that the linking probabilities for multichordal \clek{} in Theorem~\ref{thm:multichordal}\eqref{it:multsle_rel} are given by the same formula as in \cite{fpw-connection-prob}.

Our next result is the continuity of the multichordal \clek{} law with respect to the marked point configuration $\ul{x}$. In order to state the result, we first discuss the topology of the space of collections of loops and chords. Due to conformal invariance, it suffices to consider the unit disc $\D$.

Suppose $\gamma_1,\gamma_2$ are curves in $\ol\D$, viewed modulo reparameterization and orientation. Then their distance is defined as
\[ d(\gamma_1,\gamma_2) = \inf_{\varphi_1,\varphi_2} \sup_{t\in[0,1]}\abs{\gamma_1(\varphi_1(t))-\gamma_2(\varphi_2(t))} \]
where the infimum is taken with respect to all parameterizations $\varphi_i$ of $\gamma_i$ on $[0,1]$, $i=1,2$. The distance between two unrooted loops, viewed modulo reparameterization and orientation, is defined in the same way (where the point $\varphi_i(0) = \varphi_i(1) \in \gamma_i$ is not prescribed as the loop is unrooted). For two countable collections $\CK_1,\CK_2$ of loops and curves, we define their distance
\begin{equation}\label{eq:topology_loop_ensemble}
 d(\CK_1,\CK_2) = \inf_{\sim} \sup_{\gamma_1 \sim \gamma_2} d(\gamma_1,\gamma_2)
\end{equation}
where the infimum is over all bijections $\sim$ between $\CK_1$ and $\CK_2$ sending chords to chords and loops to loops. (The probability measures discussed in this article live on the subset of collections that contain only finitely many loops of diameter at least $\varepsilon$ for each $\varepsilon > 0$.)

\begin{theorem}\label{th:continuity_mcle}
Let $N \in \N$ and $\beta \in \linkpatterns{N}$. Suppose that $(\D;\ul{x};\beta) \in \eldomain{2N}$ and $(\ul{x}_n)$ is a sequence of marked point configurations on $\partial\D$ such that $\ul{x}_n \to \ul{x}$. Then the law of the multichordal \clek{} in $(\D;\ul{x}_n;\beta)$ converges weakly to the law of the multichordal \clek{} in $(\D;\ul{x};\beta)$ in the topology \eqref{eq:topology_loop_ensemble}.
\end{theorem}

In Section~\ref{sec:tv_convergence} we will prove a continuity result of the multichordal \clek{} law in a stronger topology, namely the total variation distance. The stronger continuity result holds for sequences of marked domains converging in the Carathéodory topology, and the result is that the law of the multichordal \clek{} restricted to regions away from the non-constant segments of the boundary converges in total variation. We refer to Section~\ref{sec:tv_convergence} for the precise statements.

\begin{figure}[ht]
\begin{center}
\includegraphics[width=0.45\textwidth]{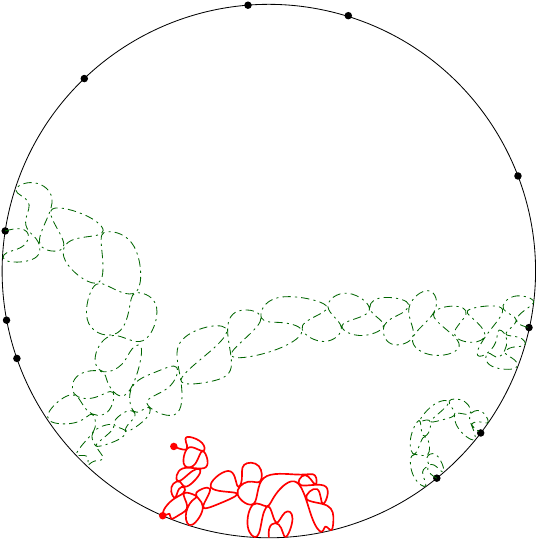}\hspace{0.05\textwidth}\includegraphics[width=0.45\textwidth]{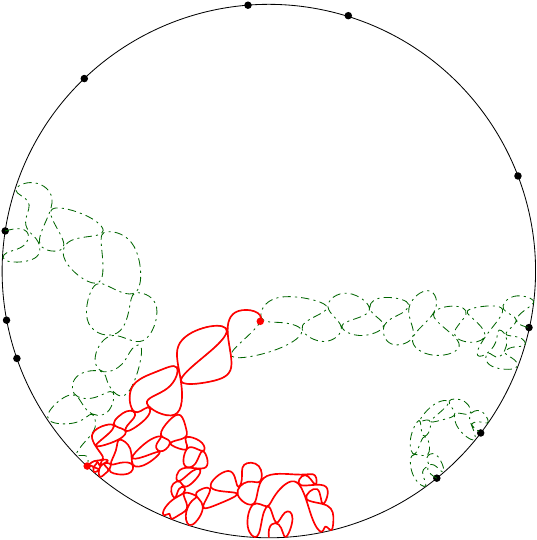}
\end{center}
\caption{Illustration of a clockwise exploration path of a multichordal \clek{} in~$\D$ starting from~$-i$ targeting $i$ (say) where $\kappa \in (4,8)$. The exploration path (drawn in red) is stopped at two different stopping times in respectively the left and right figure. The black dots represent the points in $\ul{x}$, the red dots represent the two additional marked points induced by the (stopped) exploration path as in the statement of Proposition~\ref{pr:mcle_exploration_path}.}\label{fig:clockwise_exploration_path}
\end{figure}

Next, we generalize the definition of an exploration path of a non-simple \clek{} (see Section~\ref{se:nonsimple_cle}) to the case of a multichordal \clek{}.  See Figure~\ref{fig:clockwise_exploration_path} for an illustration.

\begin{definition}[Exploration path of a multichordal \clek{}, $\kappa \in (4,8)$]
\label{def:mcle_exploration_path}
 Suppose $(\ul{\eta},\Gamma)$ is a multichordal \clek{} in $(\D;\ul{x};\beta) \in \eldomain{2N}$. The clockwise (resp.\ counterclockwise) exploration path $\eta_{z,w}$ of $(\ul{\eta},\Gamma)$ starting from $z \in \partial\D$ targeting $w \in \ol{\D}$ is defined as follows. Let $U$ be the connected component of $\D \setminus \bigcup\ul{\eta}$ with $z$ on its boundary, let $C$ be the clockwise (resp.\ counterclockwise) arc of $\partial\D$ from $z$ until the first point $u$ that intersects one of the chords of $\ul{\eta}$. Let $\Gamma_U \subseteq \Gamma$ be the collection of loops that are contained in $\ol{U}$. Let the first part of $\eta_{z,w}$ be the exploration path of $\Gamma_U$ starting from $z$ targeting $u$ that traces (portions of) loops of $\Gamma_U$ intersecting $C$ counterclockwise (resp.\ clockwise) until it separates $u$ from $w$ or reaches $u$. In the former case, $\eta_{z,w}$ continues as the exploration path of $\Gamma_U$ targeting $w$. In the latter case, let $\eta_j$ be the chord in $\ul{\eta}$ intersecting $u$. The point $u$ splits $\eta_j$ into two parts; we let $\eta_{z,w}$ continue tracing the part of $\eta_j$ that goes away from $\partial\D$ (i.e.\ it goes to the right (resp.\ left) in the case of a clockwise (resp.\ counterclockwise) exploration). The exploration stops when reaching an endpoint of $\eta_j$.
\end{definition}

The reason we stop the exploration when $\eta_{z,w}$ has reached the end of a chord in $\ul{\eta}$ is that there is not always a natural choice from which point we continue the exploration toward $w$. But one can keep exploring the connected component $D$ containing $w$ by conditioning on $(\ul{\eta},\Gamma)$ outside $D$, picking any point on $\partial D$, and further explore $(\ul{\eta},\Gamma)$ from there.

\begin{figure}[ht]
\begin{center}
\includegraphics[width=0.45\textwidth]{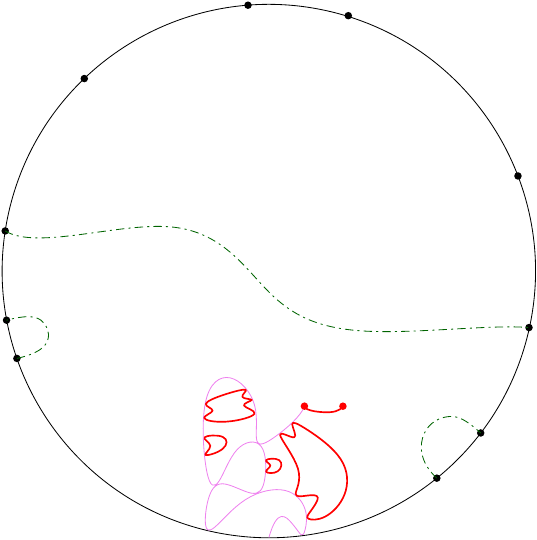}\hspace{0.05\textwidth}\includegraphics[width=0.45\textwidth]{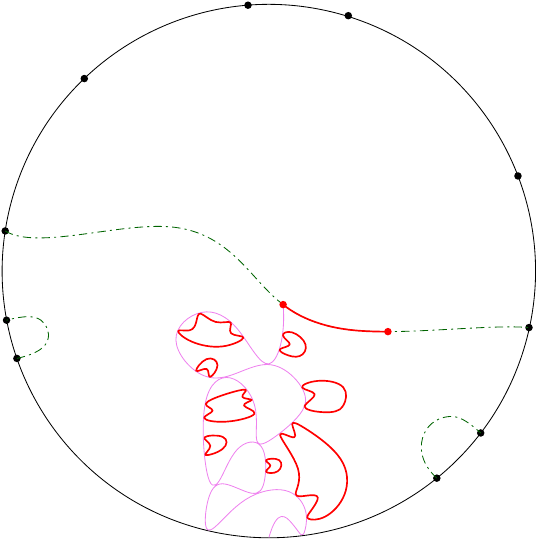}
\end{center}
\caption{Illustration of an \slekrb{\kappa-6}{1} exploration of a multichordal \clek{} in~$\D$ starting from~$-i$ targeting $i$ (say) where $\kappa \in (8/3,4)$. The exploration path consists of its trunk (drawn in pink) and a collection of (parts of) \clek{} loops (or chords) it explores (drawn in red). The exploration path is stopped at two different stopping times in respectively the left and right figure. The black dots represent the points in $\ul{x}$, the red dots represent the two additional marked points induced by the (stopped) exploration path as in the statement of Proposition~\ref{pr:mcle_exploration_path}.}\label{fig:cpi_exploration}
\end{figure}

Similarly, we generalize the definition of the \slekrb{\kappa-6}{1} and \slekrb{\kappa-6}{-1} (resp.\ \slefmu{\mu}) explorations of a simple \clek{} (see Section~\ref{se:simple_cle}) to the case of a multichordal \clek{}. See Figure~\ref{fig:cpi_exploration} for an illustration.

\begin{definition}[\slekrb{\kappa-6}{1} and \slekrb{\kappa-6}{-1} exploration of a multichordal \clek{}, $\kappa \in (8/3,4)$, and \slefmu{\mu} exploration of a multichordal \cle{4}]
 Suppose $(\ul{\eta},\Gamma)$ is a multichordal \clek{} in $(\D;\ul{x};\beta) \in \eldomain{2N}$. The \slekrb{\kappa-6}{1} (resp.\ \slekrb{\kappa-6}{-1}) exploration $\eta_{z,w}$ of $(\ul{\eta},\Gamma)$ starting from $z \in \partial\D$ targeting $w \in \ol{\D}$ is defined as follows. Let $U$ be the connected component of $\D \setminus \bigcup\ul{\eta}$ with $z$ on its boundary. Let $\Gamma_U \subseteq \Gamma$ be the collection of loops that are contained in $U$. Let the first part of $\eta_{z,w}$ be given by a sample of the \slekrb{\kappa-6}{1} (resp.\ \slekrb{\kappa-6}{-1}) exploration of $\Gamma_U$ starting from $z$ targeting $w$ (which traces loops of $\Gamma_U$ counterclockwise (resp.\ clockwise)) until the first time~$\sigma$ it hits a chord of $\ul{\eta}$. If $\sigma < \infty$ (i.e.\ it does not separate $w$ from $\ul{\eta}$ before), let $\eta_j$ be the chord in $\ul{\eta}$ intersecting $\eta_{z,w}(\sigma)$. We let $\eta_{z,w}$ continue tracing the part of $\eta_j$ from $\eta_{z,w}(\sigma)$ towards the right (resp.\ left). The exploration stops when reaching an endpoint of $\eta_j$.

 For $\mu \in \R$, the \slefmu{\mu} exploration of a multichordal \cle{4} is defined analogously where when it hits a chord of $\ul{\eta}$, it chooses one of the two directions independently and uniformly at random.
\label{def:mcle_cpi_exploration}
\end{definition}

We emphasise that as in the case of a usual \clek{} ($N=0$), the \slekrb{\kappa-6}{1} and \slekrb{\kappa-6}{-1} explorations are not deterministic functions of $(\ul{\eta},\Gamma)$ but contain additional randomness \cite{msw2020nonsimple}.

In the case of \clek{} (that is, $N=0$), given a stopping time for the exploration path, the remainder is given by a monochordal \slek{} in the connected component of the exploration path, and an independent \clek{} in each of the other connected components. The following proposition generalizes this to all $N \in \N$, and will be proved in Section~\ref{sec:proof_existence_multichordal}.

\begin{proposition}\label{pr:mcle_exploration_path}
 Suppose $(\ul{\eta},\Gamma)$ is a multichordal \clek{} in $(\D;\ul{x};\beta) \in \eldomain{2N}$, $N \in \N_0$. Let $\eta_{z,w}$ be an exploration path of $(\ul{\eta},\Gamma)$ starting from $z \in \partial\D$ targeting $w \in \ol{\D}$ as in Definition~\ref{def:mcle_exploration_path} (resp.\ Definition~\ref{def:mcle_cpi_exploration}). Let $\tau$ be a stopping time for $\eta_{z,w}$, and let $D_\tau$ be the connected component of $\D\setminus\eta_{z,w}([0,\tau])$ whose closure contains $w$. Let $\ul{x}_\tau$ consist of the points in $\ul{x}$ that are on $\partial D_\tau$ together with $\eta_{z,w}(\tau)$ and $\eta_{z,w}(\sigma)$ where $\sigma$ is the last time before $\tau$ when $\eta_{z,w}$ has started tracing a new loop or a chord of $\ul{\eta}$. Let $\beta_\tau$ be the exterior link pattern induced by $\beta$, $\eta_{z,w}|_{[\sigma,\tau]}$, and $(\ul{\eta},\Gamma)|_{\D \setminus D_\tau}$. Then the conditional law of the remainder of $(\ul{\eta},\Gamma)$ given $\eta_{z,w}|_{[0,\tau]}$ and $(\ul{\eta},\Gamma)|_{\D \setminus D_\tau}$ is a multichordal \clek{} in $(D_\tau;\ul{x}_\tau;\beta_\tau)$.
\end{proposition}

Here, the induced exterior link pattern $\beta_\tau$ is defined as above Definition~\ref{def:resampling_kernel}, with an additional link between the points $\eta_{z,w}(\tau)$ and $\eta_{z,w}(\sigma)$.

\subsubsection{Multichordal \clek{} from partial explorations of \clek{}}

Our next result relates partial explorations of nested \clek{} to the multichordal \clek{} considered above. In order to state this result, we first need to give a precise definition of what we mean by a partial exploration of a \clek{}.  See Figures~\ref{fig:partially_explored} and~\ref{fig:link_pattern_partially_explored} for an illustration.

\begin{figure}[ht]
\centering
\includegraphics[width=0.45\textwidth]{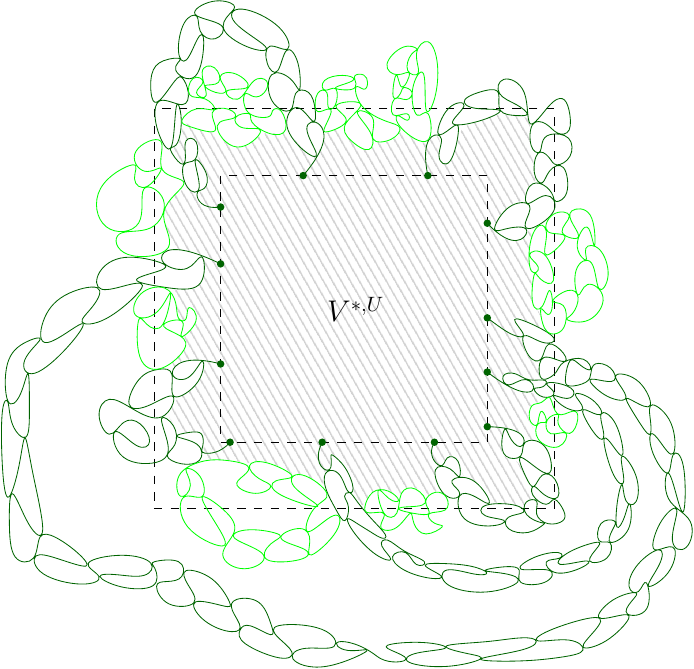}
\caption{\label{fig:partially_explored}The partially explored nested \clek{} $\Gamma_\outside^{*,V,U}$ for $U \subseteq V$ with $(U,V) \in \domainpair{D}$.  In the illustration, $V$ (resp.\ $U$) is shown as the outer (resp.\ inner) square.}
\end{figure}

\begin{figure}[ht]
\centering
\includegraphics[width=0.45\textwidth]{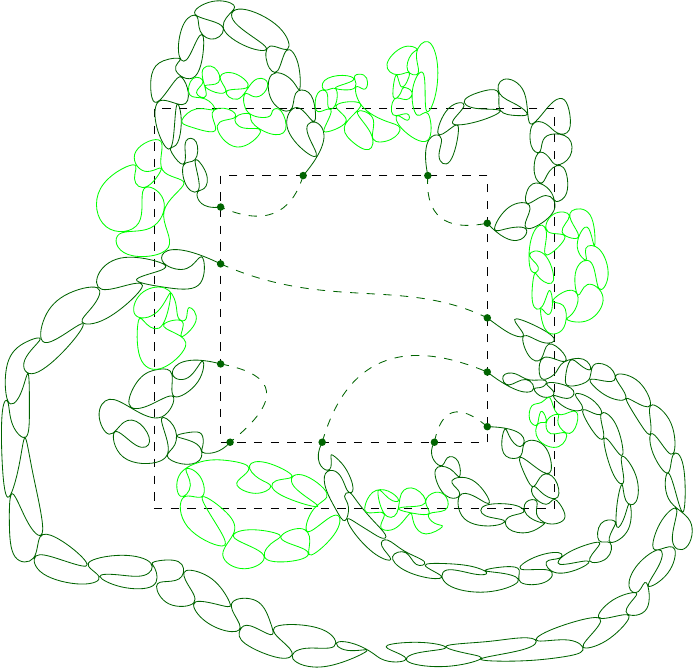}\hspace{0.05\textwidth}\includegraphics[width=0.45\textwidth]{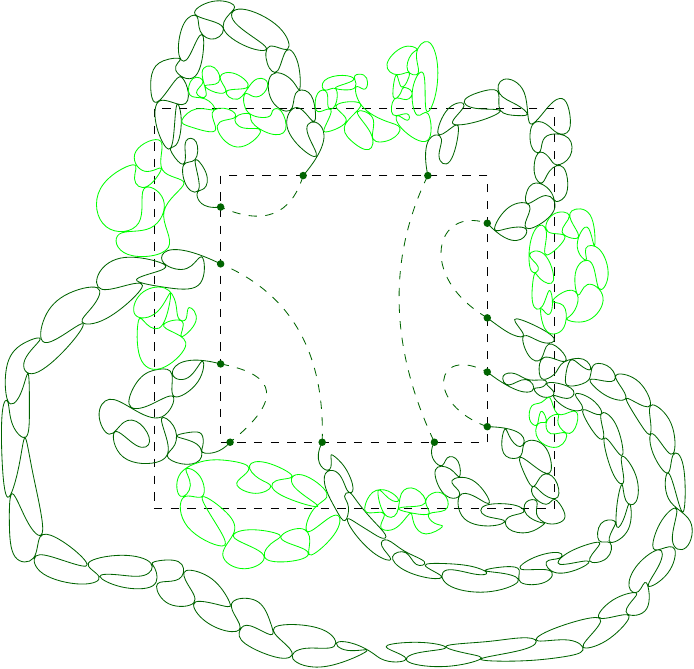}
\caption{\label{fig:link_pattern_partially_explored}Two possible interior link patterns with the given exterior link pattern coming from the partial exploration $\Gamma_\outside^{*,V,U}$ as illustrated in Figure~\ref{fig:partially_explored}.}
\end{figure}

\begin{definition}[Partially explored CLE]
\label{def:partially-explored-cle}
Let $D \subseteq \C$ be a simply connected domain and $\varphi \colon D \to \D$ be a conformal transformation.  Let $\domainpair{D}$ consist of all pairs $(U,V)$ of simply connected sub-domains $U \subseteq V \subseteq D$ with $\dist(\varphi(D\setminus V), \varphi(U)) > 0$.  (Note that $\domainpair{D}$ does not depend on the choice of $\varphi$.) 

Suppose $\Gamma$ is a nested multichordal \clek{} in $(D;\ul{x};\beta) \in \eldomain{2N}$, $N \in \N_0$, and $(U,V) \in \domainpair{D}$. We let $\Gamma_\outside^{*,V,U}$ be the collection of maximal segments of loops and strands in $\Gamma$ that intersect $D \setminus V$ and are disjoint from $U$.  We call $\Gamma_\outside^{*,V,U}$ the \emph{partial exploration} of $\Gamma$ in $D\setminus V$ until hitting $\ol{U}$. We let $V^{*,U}$ be the connected component containing $U$ after removing from $D$ all loops and strands of $\Gamma_\outside^{*,V,U}$. This defines a marked domain $(V^{*,U};\ul{x}^*;\beta^*)$ where the marked points $\ul{x}^*$ correspond to the ends of the incomplete strands in $\Gamma_\outside^{*,V,U}$ and the points in $\ul{x} \cap \partial V^{*,U}$. The strands in $\Gamma_\outside^{*,V,U}$ together with $\beta$ induce a planar link pattern $\beta^*$ on the exterior of $V^{*,U}$.  We let $\Gamma_\inside^{*,V,U}$ be the collection of loops and strands of $\Gamma$ in $V^{*,U}$, which we call the \emph{unexplored part} of $\Gamma$.
\end{definition}

\begin{theorem}
\label{thm:cle_partially_explored}
Suppose that $(D;\ul{x};\beta) \in \eldomain{2N}$, $N \in \N_0$, and $\Gamma$ is a nested multichordal \clek{} in $(D;\ul{x};\beta)$. For $(U,V) \in \domainpair{D}$, let $\Gamma_\outside^{*,V,U}$ and $(V^{*,U};\ul{x}^*;\beta^*)$ be as in Definition~\ref{def:partially-explored-cle}.  Then the conditional law of the remainder $\Gamma_\inside^{*,V,U}$ of $\Gamma_\outside^{*,V,U}$ is given by $\mcclelaw{V^{*,U};\ul{x}^*;\beta^*}$.
\end{theorem}

\subsubsection{Independence across scales and resampling}

Theorem~\ref{thm:cle_partially_explored} is very useful for proving things about \clek{} because it can be used to locally relink loops. We will explore this systematically in Section~\ref{sec:resampling_tools}. In particular, we will show that for any sufficiently long sequence of nested annuli, with probability tending to $1$ very quickly, a large fraction satisfy a quantitative regularity condition which guarantees that if we resample the CLE strands in randomly chosen locations inside the annulus, we can relink the loops in any prescribed way with uniformly positive probability.

Interesting cases are the creation of single loops that disconnect the inside of the annulus from the outside, and the breaking up of all crossings across the annulus. When $\kappa \in (4,8)$, this can be done without changing the shape of the complementary connected components (as in \cite{msw2020nonsimple}). We will use this extensively in later work \cite{amy2025tightness}.  We refer to Section~\ref{sec:resampling_tools} for the precise statements.

Another useful application of the partially explored \clek{} is an independence across scales result. It roughly says that there is ``sufficient independence'' between the \clek{} configurations in disjoint annuli so that any event for the \clek{} configuration in an annulus that has positive probability is extremely likely to occur in a positive fraction of nested annuli. We refer to Section~\ref{subsec:ind_across_scales} for the precise statements.

\subsection*{Acknowledgements}

The authors thank Ellen Powell and Wendelin Werner for their helpful comments on an earlier version of the paper. V.A., J.M., and Y.Y.\ were supported by ERC starting grant SPRS (804116). J.M.\ and Y.Y.\ also received support from ERC consolidator grant ARPF (Horizon Europe UKRI G120614), and Y.Y.\ in addition received support by the Royal Society.

\subsection*{Outline}

The remainder of this article is structured as follows. In Section~\ref{sec:prelim}, we will collect a number of preliminaries. We will also recall the bichordal \clek{} studied in \cite{msw2020nonsimple}. In Section~\ref{sec:multichordal_ex_uniq} we will show the existence and uniqueness of multichordal \clek{} and prove their main properties stated in the introduction. In Section~\ref{sec:strands} we will prove the independence across scales results for \clek{}. A crucial technical tool will be the \emph{separation of strands}. In Section~\ref{sec:resampling_tools} we prove our results about the resampling operation, including the creation of disconnecting loops and the breaking of loop crossings. In Section~\ref{sec:tv_convergence} we will prove the continuity of \clek{} in total variation.

\subsection*{Notation}

We denote annuli by $A(z,r_1,r_2) \defeq B(z,r_2)\setminus \ol{B(z,r_1)}$ where $r_2>r_1>0$. We write $U_1 \Subset U_2$ to denote that $U_1$ is compactly contained in $U_2$, i.e.\ $\ol{U_1}$ is compact and $\ol{U_1} \subseteq U_2$. We write $a \lesssim b$ meaning that $a \le cb$ for some constant $c$ whose value may change from line to line.

\section{Preliminaries}
\label{sec:prelim}

The purpose of this section is to collect some preliminaries.  We will start with the definition of SLE in Section~\ref{subsec:sle}.  We will then review the definition of CLE in Section~\ref{subsec:cle}.  Finally, we will collect a few results about monochordal and bichordal $\CLE$ in Section~\ref{subsec:bichordal_cle}.

\subsection{Schramm-Loewner evolution}
\label{subsec:sle}

The Schramm-Loewner evolution (\slek{}) was introduced by Schramm in \cite{s2000sle}.  The starting point for the definition of \slek{} is the chordal Loewner equation
\begin{equation}
\label{eqn:loewner_equation}
\partial_t g_t(z) = \frac{2}{g_t(z) - W_t},\quad g_0(z) = z.
\end{equation}
Here, $W \colon \R_+ \to \R$ is a continuous function and for each fixed $z \in \h$ the solution $(g_t(z))$ to~\eqref{eqn:loewner_equation} is defined up until $\tau(z) = \inf\{t \geq 0 : \im(g_t(z)) = 0\}$.  Let $K_t = \{z \in \h : \tau(z) \leq t\}$ and $\h_t = \h \setminus K_t$.  Then $g_t$ is the unique conformal transformation $\h_t \to \h$ with $|g_t(z) - z| \to 0$ as $z \to \infty$.

Suppose that $\kappa \geq 0$.  Then \slek{} in $\h$ from $0$ to $\infty$ is defined by taking $W_t = \sqrt{\kappa} B_t$ where $B$ is a standard Brownian motion.  It is not immediate from its definition that \slek{} corresponds to a continuous curve, meaning that there exists a continuous curve $\eta \colon \R_+ \to \ol{\h}$ so that for each $t \geq 0$ we have that $\h_t$ is the unbounded component of $\h \setminus \eta([0,t])$.  This was proved by Rohde-Schramm for $\kappa \neq 8$ in \cite{rs2005basic} and for $\kappa=8$ by Lawler-Schramm-Werner in \cite{lsw2004lerw} as a consequence of the convergence of the uniform spanning tree Peano curve to \sle{8} (see also \cite{am2022sle8} for a proof which makes use of only continuum methods).

The \slekr{\rho} processes are an important variant of \slek{} first introduced in \cite[Section~8.3]{lsw2003restriction} where one keeps track of extra marked points.  They are defined by solving~\eqref{eqn:loewner_equation} where the driving function $W$ is given by the solution to the SDE
\begin{equation}
\label{eqn:sle_kappa_rho}
dW_t = \sqrt{\kappa} dB_t + \sum_i \frac{\rho_i}{W_t - V_t^i} dt,\quad dV_t^i = \frac{2}{V_t^i - W_t},\quad V_0^i = z_i.
\end{equation}
The existence and uniqueness of solutions to~\eqref{eqn:sle_kappa_rho} up until the continuation threshold is reached was proved in \cite{ms2016ig1}.  The continuity of the corresponding curves up until this time was proved in \cite{ms2016ig1}.

\subsubsection{Generalised \slekr{\rho} processes}
\label{subsubsec:gen_sle}

It is also possible to consider \slekr{\rho} processes for some $\rho \le -2$ \cite{s2009cle}. These are referred to as \emph{generalised \slekr{\rho} processes}.

Suppose that $\rho \in (-2-\kappa/2,-2)$ and that we have a single force point at $z \in \R^-$ with weight $\rho$. Then in~\eqref{eqn:sle_kappa_rho}, the process $X_t = (W_t-V_t)/\sqrt{\kappa}$ evolves according to a Bessel process of dimension $1+2(\rho+2)/\kappa \in (0,1)$. The integral
\[
 \int \frac{2}{V_s-W_s} ds = -\frac{2}{\sqrt{\kappa}}\int \frac{1}{X_s} ds
\]
becomes infinite after times when $W_t = V_t$, so that~\eqref{eqn:sle_kappa_rho} is not defined after that time (i.e.\ the continuation threshold is reached). However, one can define a \emph{principle value integral}
\begin{equation}\label{eq:pvint}
 \pv\int \frac{1}{X_s} ds
\end{equation}
by subtracting a divergent compensation term, see \cite[Chapter~XI~(1.26)]{ry-book}. One can thus define the generalised \slekr{\rho} processes by considering a Bessel process $X$ of dimension $1+2(\rho+2)/\kappa$ and setting
\begin{equation}\label{eq:gen_sle}
 W_t = V_t+\sqrt{\kappa}X_t ,
 \quad V_t = -\frac{2}{\sqrt{\kappa}}\pv\int_0^t \frac{1}{X_s}ds
\end{equation}
for the driving function $W$ in~\eqref{eqn:loewner_equation}.

One can further define the \emph{side-swapping \slekrb{\rho}{\beta} processes} as follows. Let $\beta \in [-1,1]$. The \emph{skew Bessel process} $X$ with skew parameter $\beta$ is given by letting $\abs{X}$ be a Bessel process of the corresponding dimension and assigning a sign to each excursion of $X$ away from $0$ independently at random, with probability $(1+\beta)/2$ (resp.\ $(1-\beta)/2$) for a positive (resp.\ negative) sign. One can then define $\pv\int 1/X_s \,ds$ as the difference of the principle value integrals~\eqref{eq:pvint} for the positive and negative parts of $X$. The \slekrb{\rho}{\beta} process is now defined by~\eqref{eq:gen_sle} with this choice of $X$.

In the case $\rho = -2$, the corresponding Bessel process has dimension $1$. In that case, one can define $\pv\int 1/X_s \,ds$ only when $\beta=0$, but one can add a local time drift as follows. Let $X$ be a Bessel process of dimension $1$ with skew parameter $\beta=0$, and let $(l^0_t)$ be its local time process at $0$. For $\mu \in \R$, the \slekrb{-2}{0,\mu} process is defined by setting
\begin{equation*}
 W_t = V_t+\sqrt{\kappa}X_t ,
 \quad V_t = -\frac{2}{\sqrt{\kappa}}\left( \pv\int_0^t \frac{1}{X_s}ds + \mu l^0_t \right)
\end{equation*}
for the driving function $W$ in~\eqref{eqn:loewner_equation}.

In the case that $\rho \in (\kappa/2-4,-2)$, the continuity of the generated curves (for $\beta = \pm 1$) was proved in \cite{ms2019lightcone} while for $\rho \in (-2-\kappa/2,\kappa/2-4]$ the continuity was proved in \cite{msw2017cleperc}.

\subsection{Conformal loop ensembles}
\label{subsec:cle}

The \emph{conformal loop ensembles} (\clek{}) were introduced in \cite{s2009cle,sw2012cle}.  They consist of a countable collection of loops in a simply connected domain, each of which locally looks like an \slek{}. The \clek{} are defined for $\kappa \in (8/3,8)$. We review the definition of \clek{} separately for $\kappa \in (4,8)$ and for $\kappa \in (8/3,4]$.

\subsubsection{Non-simple CLE}
\label{se:nonsimple_cle}

We describe the construction of \clek{} for $\kappa \in (4,8)$.  Suppose that $D \subseteq \C$ is a simply connected domain.  Suppose that we have fixed $x \in \partial D$ and a countable dense set $(y_n)$ in $\partial D$.  For each $n \in \N$, we let $\eta_n$ be an \slekr{\kappa-6} in $D$ from $x$ to $y_n$ where the force point is infinitesimally to the right of $x$.  We note for each $n,m$ that $\eta_m$ viewed as a process targeted at $y_n$ has the same law as $\eta_n$ up until the first time that $\eta_m$ disconnects $y_n$ from $y_m$ \cite{sw2005coordinate}.  This means that we can assume that the $(\eta_n)$ are coupled together onto a common probability space so that any finite subcollection agrees up until the first time that their target points are separated and afterwards evolve independently.  This tree of \slekr{\kappa-6}'s is the so-called \emph{exploration tree}.

Let us now describe how the loops of a \clek{} which intersect $\partial D$ are constructed out of the branches of the exploration tree.  Fix $n \in \N$, $t > 0$, let $\sigma$ (resp.\ $\tau$) be the last time before $t$ (resp.\ next time after $t$) where the right boundary of $\eta_n$ intersects $\partial D$. 
Then $\eta_n|_{[\sigma,\tau]}$ is part of a loop $\CL$.  To obtain the rest of $\CL$, suppose that $(y_{n_k})$ is a subsequence of $(y_n)$ in the counterclockwise arc of $\partial D$ from $\eta_n(\sigma)$ to $\eta_n(\tau)$ which converges to $\eta_n(\sigma)$.  Then we know that each $\eta_{n_k}$ has $\eta_n|_{[\sigma,\tau]}$ as a subarc and all of the $\eta_{n_k}$ agree up until the time they disconnect $\eta_n(\sigma)$ from their target point.  We thus take the remainder of $\CL$ to the the limit of the $\eta_{n_k}$'s up until this disconnection time.  Varying $n$ and $t$ yields a family of loops which all intersect $\partial D$ and yields the boundary intersecting loops of the \clek{}.  To generate the remaining loops, we sample independently from the same law in each of the holes whose winding number induced by the loops is $0$.  This gives us a non-nested \clek{}.  To get a nested \clek{}, we generate the loops in each of the remaining holes.

It was explained in \cite{s2009cle} that the continuity of the loops of a \clek{} follows from the continuity of the \slekr{\kappa-6} processes, which was proved in \cite{ms2016ig1}.  It was also shown in~\cite{s2009cle} that the law of the loops of a \clek{} does not depend on the choice of the root of the exploration tree as a consequence of the reversibility of the \slekr{\kappa-6} processes, which was proved in \cite{ms2016ig3}.  Finally, the local finiteness of \clek{} for $\kappa \in (4,8)$ was proved in \cite{ms2017ig4} as a consequence of the continuity of space-filling SLE.  We recall that this means that if we have a \clek{} on a bounded Jordan domain, then it is a.s.\ the case that for each $\varepsilon > 0$ the number of loops which have diameter at least $\varepsilon > 0$ is finite.  The \emph{gasket} of a CLE is the set of points in $\ol{D}$ that do not have non-trivial winding number by any loop of the CLE. The dimension of the CLE gasket was computed in \cite{ssw2009radii,msw2014dimension}.

We will call each of the branches of the \clek{} exploration tree described above the exploration path from the root to its target point. More precisely, we call it the \emph{counterclockwise exploration path} as it discovers the loop touching $\partial D$ in counterclockwise order. The counterclockwise exploration path $\eta_{x,y}$ from $x$ targeting $y$ is determined by the \clek{} as follows. For each loop $\CL \in \Gamma$ intersecting $\partial D$, let $z_\CL$ be first point of $\CL \cap \partial D$ encountered when travelling counterclockwise along $\partial D$ from $x$. Orient each such loop clockwise and root it at $z_\CL$. Let $\eta'$ be the path obtained by concatenating all boundary intersecting loops in the order of their roots. Finally, we let $\eta_{x,y}$ be the subpath of $\eta'$ consisting of those segments that are evolving in the complementary connected component adjacent to $y$.

We define the \emph{clockwise exploration path} analogously by interchanging ``clockwise'' and ``counterclockwise'' everywhere in the definition above. Finally, we define the exploration path from $x$ targeting an interior point $y \in D$ as follows. Whenever $\eta'$ disconnects $y$ from $\partial D$ by surrounding it clockwise (resp.\ counterclockwise), we continue with the clockwise (resp.\ counterclockwise) exploration path of the \clek{} in the connected component containing $y$.

The following property of non-simple CLE is proved in \cite{gmq2021sphere}.

\begin{lemma}[{see the proof of \cite[Lemma~3.2]{gmq2021sphere}}]\label{le:cle_finite_chain}
 Let $\kappa \in (4,8)$. Let $O \subseteq \C$ be an open, connected set that intersects $\partial\D$. Let $\Gamma$ be a (nested or non-nested) \clek{} in $\D$, and $\Gamma(O) \subseteq \Gamma$ the collection of loops that intersect $O$. Then a.s.\ for every loop $\CL$ in $\Gamma(O)$ there is a finite sequence of loops $\CL_1,\ldots,\CL_m$ in $\Gamma(O)$ for some $m \in \N$ such that $\CL_m = \CL$, the loop $\CL_1$ intersects $\partial\D \cap O$, and $\CL_i$ intersects $\CL_{i+1}$ in $O$ for each $i=1,\ldots,m-1$.
\end{lemma}

\subsubsection{Simple CLE}
\label{se:simple_cle}

The construction of \clek{} for $\kappa \in (8/3,4]$ can be done similarly using a generalised \slekr{\kappa-6} (note that $\kappa-6 \le -2$). In contrast to the case of non-simple CLE, the exploration tree is not determined by the \clek{} for $\kappa \in (8/3,4)$. It is also harder to prove that the construction gives rise to a collection of simple (continuous) loops, the collection is locally finite, and that the law of the loop ensemble does not depend on the choice of the root. These facts are proved in \cite{sw2012cle} (or alternatively in \cite{msw2017cleperc}).

We now give the precise construction of the exploration tree and the collection of \clek{} loops it induces. Let $x \in \partial D$ and $\beta \in [-1,1]$ (resp.\ $\mu \in \R$ in the case $\kappa = 4$). By the target-invariance of the \slekrb{\kappa-6}{\beta} (resp.\ \slefmu{\mu}) shown in \cite{sw2005coordinate,s2009cle} (or alternatively in \cite{msw2017cleperc}), we can define a branching \slekrb{\kappa-6}{\beta} (resp.\ \slefmu{\mu}) starting from $x$ targeting a countable dense subset of $\partial D$ similarly as in Section~\ref{se:nonsimple_cle}. For each branch of the \slekrb{\kappa-6}{\beta} (resp.\ \slefmu{\mu}) exploration tree and each time interval where the corresponding Bessel process makes an excursion away from $0$ (recall Section~\ref{subsubsec:gen_sle}), the \slekrb{\kappa-6}{\beta} (resp.\ \slefmu{\mu}) process makes a loop during that interval. The collection of loops obtained in this way form a part of the non-nested \clek{}. The remaining loops of the non-nested \clek{} is obtained by iterating the construction in each of the complementary connected components with winding number $0$ with respect to the loops. To get a nested \clek{}, we iterate the construction in each of the connected components inside the loops. It is shown in \cite{sw2012cle,ww-cle-explorations} that the law of the \clek{} does not depend on the choice of the root $x$ nor on the parameter $\beta$ (resp.\ $\mu$).

We can also define the \slekrb{\kappa-6}{\beta} (resp.\ \slefmu{\mu}) exploration targeted towards an interior point $y \in D$. Let $\eta$ be the branch of the exploration tree whose target lies in the same connected component as $y$. When $\eta$ disconnects $y$ from $\partial D$, we continue with an \slekrb{\kappa-6}{\beta} (resp.\ \slefmu{\mu}) in the connected component containing $y$ if the number of \clek{} loops surrounding it is even, and an \slekrb{\kappa-6}{-\beta} (resp.\ \slefmu{-\mu}) if the number of \clek{} loops surrounding it is odd.

The set of points on an \slekrb{\kappa-6}{\beta} (resp.\ \slefmu{\mu}) at the times when the corresponding Bessel process is $0$ is called the \emph{trunk} of the \slekrb{\kappa-6}{\beta} (resp.\ \slefmu{\mu}). An \slekrb{\kappa-6}{\beta} (resp.\ \slefmu{\mu}) thus consists of its trunk and the collection of \clek{} loops it intersects. It is shown in \cite{msw2017cleperc} that the trunk of an \slekrb{\kappa-6}{\beta} (resp.\ \slefmu{\mu}) is given by a particular $\SLE_{16/\kappa}(\rho)$ process (see also Section~\ref{subsubsec:simple_cle_gff}).

\subsection{Monochordal and bichordal \clek{}}\label{subsec:bichordal_cle}
 
In~\cite[Section~3]{msw2020nonsimple} it is shown that the law of the remainder of a \clek{} after partially revealing two loop segments is conformally invariant and is given by the bichordal \clek{} law given above Definition~\ref{def:resampling_kernel} (in the case $N=2$) in the marked domain remaining upon the exploration. We now rephrase some results from \cite[Section~3]{msw2020nonsimple} concerning monochordal and bichordal \clek{} in a way which will be convenient for generalizing to arbitrary $N \in \N$.

Recall the definition of multichordal \clek{} from Definition~\ref{def:exterior_link_pattern_multiple_sle}. When $N=1$, the  monochordal \clek{}   $(\eta,\Gamma)$ in $(D;x_1,x_2) \in \markeddomain{2}$  consists of an  \slek{} process~$\eta$ in~$D$ from $x_1$ to $x_2$ and a collection $\Gamma$ of conditionally independent \clek{}'s in each connected component of $\D\setminus\eta$. It is clear that the law of monochordal \clek{} is unique and exists for every $(D;\ul{x}) \in \markeddomain{2}$.  It also  arises when, for example given    a \clek{}  $\Gamma$ in $\D$, one discovers a branch of  the \clek{} exploration tree from a boundary point and  stops the exploration while it is tracing a loop. If one then  starts a similar   exploration from a second point on $\partial \D$, one obtains a marked domain with $4$ prime ends. This setup is studied in~\cite[Section~3]{msw2020nonsimple}, where it is shown that the conditional law of the two remaining chords, along with the loops of $\Gamma$, is conformally invariant and depends only on the conformal class of $(D; \ul{x})$, not on the details of the exploration. The joint law of the two chords is given by the bichordal \clek{} law defined above Definition~\ref{def:resampling_kernel} in the case $\beta = \extadjacent$.

We begin by reviewing the result for the regime $\kappa \in (4,8)$. Recall the definition of the exploration path for a monochordal \clek{} given in Definition~\ref{def:mcle_exploration_path}.

\begin{lemma}[{\cite[Lemma~3.2]{msw2020nonsimple}}]
\label{le:4strands}
Let $\kappa \in (4,8)$. Suppose $(\eta,\Gamma)$ is a monochordal \clek{} in $(\D;-i,i)$. Let $\eta_2$ be the clockwise exploration path of $(\eta,\Gamma)$ starting from $1$ targeting $i$, and let $\tau_2$ be a stopping time for $\eta_2$. Let $\ol{z}_2$ be the last point on the clockwise arc of $\partial\D$ from $1$ to $-i$ visited by $\eta_2|_{[0,\tau_2]}$. Let $E_1$ be the event that $\eta_2(\tau_2) \in \eta$. Let $\ol{\eta}_2$ be the path which:
\begin{itemize}
\item On $E_1$, traces $\eta$ from $\ol{z}_2$ to $-i$.
\item On $E_2 = E_1^c$, starts from $\ol{z}_2$ and traces the corresponding loop of $\Gamma$ from $\ol{z}_2$ in the clockwise direction until $\eta_2(\tau_2)$.
\end{itemize}
Let $\ol{\tau}_2$ be a stopping time for the filtration $\sigma(\eta_2|_{[0,\tau_2]}, \ol{\eta}_2(s) : s \leq t)$. Suppose that we are on the event that the points $i,-i,\ol{\eta}_2(\ol{\tau}_2),\eta_2(\tau_2)$ are four distinct boundary points of the same connected component $D$ of $\D\setminus(\eta_2([0,\tau_2])\cup\ol{\eta}_2([0,\ol{\tau}_2]))$. Then the conditional probability of $E_1$ given $\eta_2|_{[0,\tau_2]}$, $\ol{\eta}_2|_{[0,\ol{\tau}_2]}$ is a.s.\ a function $H_\kappa(\cdot)$ of the conformal modulus of $(D;-i,\ol{\eta}_2(\ol{\tau}_2),\eta_2(\tau_2),i)$ (not depending on the choice of $\ol{\tau}_2,\tau_2$).
\end{lemma}

Lemma~\ref{le:4strands} is a rephrasing of the setup of \cite[Section~3.1]{msw2020nonsimple} where $\eta$ is seen as the remainder of a loop to be completed where the explored part of the loop is mapped to the clockwise portion of $\partial\D$ from $-i$ to $i$.  As in \cite[Section~3.1]{msw2020nonsimple}.  If we denote by $\CL_1$ the loop that $\eta$ is part of and by $\CL_2$ the loop containing the segments of $\eta_2$, $\ol{\eta}_2$ from  $\ol{z}_2$ to time $\tau_2$,  $\ol{\tau}_2$ respectively,  then $E_1$ corresponds to the event where $\CL_1$ and $\CL_2$ are in fact the same loop while the event $E_2$ corresponds to the event where $\CL_1$ and $\CL_2$ are separate loops.

We can also view the setup of Lemma~\ref{le:4strands} after reflecting the entire configuration across the vertical axis, while maintaining the interpretation that the clockwise arc of $\partial\D$ from $-i$ to $i$ corresponds to the explored part of the loop that $\eta$ is part of. In this reflected setup, the loops of $\Gamma$ lying to the left of $\eta$ now play the role of loops in the next level of nesting. More precisely, let  $r(z) = -\overline{z}$ denote the reflection map with respect to the vertical axis. Since the law of the monochordal \clek{}  $(\eta,\Gamma)$ is invariant under $r$, applying $r$ to the configuration $(\eta, \eta_2, \ol{\eta}_2, \Gamma)$ yields a new configuration $(\eta, \wt{\eta}_2, \ol{\wt{\eta}}_2, \Gamma)$, where the law of $(\eta, \Gamma)$ is left unchanged, $\wt{\eta}_2 = r(\eta_2)$ is the counterclockwise exploration path starting from $-1$ targeting $i$, and the path $\ol{\wt{\eta}}_2 = r(\ol{\eta}_2)$ along with the events $\wt{E}_1$, $\wt{E}_2$ are defined analogously as in the original setting, for appropriate stopping times $\wt{\tau}_2$ and $\ol{\wt{\tau}}_2$.

Lemma~\ref{le:4strands} states that in this setting, the conditional probability of $\wt{E}_1$ given $\wt{\eta}_2|_{[0,\wt{\tau}_2]}$, $\ol{\wt{\eta}}_2|_{[0,\ol{\wt{\tau}}_2]}$ is a.s.\ given by $H_\kappa(L)$ where $L$ is the conformal modulus of $(D;i,\wt{\eta}_2(\wt{\tau}_2),\ol{\wt{\eta}}_2(\ol{\wt{\tau}}_2),-i)$,  with $D$ the connected component of $\D \setminus (\wt{\eta}_2([0,\wt{\tau}_2]) \cup \ol{\wt{\eta}}_2([0,\ol{\wt{\tau}}_2]))$  that contains these four boundary points. In this case, the corresponding exterior link pattern is $\beta = \extnested$.

In particular, we can phrase the statement of Lemma~\ref{le:4strands} as follows.

\begin{lemma}[{\cite[Lemma~3.1]{msw2020nonsimple}}]
\label{le:4strands_cle}
Suppose that we are either in the setup of Lemma~\ref{le:4strands} or in the setup reflected with respect to the vertical axis. On the event that the points $i,-i,\ol{\eta}_2(\ol{\tau}_2),\eta_2(\tau_2)$ (resp.\ $i,-i,\ol{\wt{\eta}}_2(\ol{\wt{\tau}}_2),\wt{\eta}_2(\wt{\tau}_2)$) are distinct boundary points of the same connected component $D$ of $\D\setminus(\eta_2([0,\tau_2])\cup \ol{\eta}_2([0,\ol{\tau}_2]))$ (resp.\ $\D\setminus(\wt{\eta}_2([0,\wt{\tau}_2])\cup \ol{\wt{\eta}}_2([0,\ol{\wt{\tau}}_2]))$), the conditional law of the remainder of $(\eta,\eta_2,\Gamma\big|_D)$ (resp.\ $(\eta,\wt{\eta}_2,\Gamma\big|_D)$) given $\eta_2|_{[0,\tau_2]}$, $\ol{\eta}_2|_{[0,\ol{\tau}_2]}$ (resp.\ $\wt{\eta}_2|_{[0,\wt{\tau}_2]}$, $\ol{\wt{\eta}}_2|_{[0,\ol{\wt{\tau}}_2]}$) is a.s.\ given by the bichordal \clek{} law in $(D;i,-i, \ol{\eta}_2(\ol{\tau}_2),\eta_2(\tau_2);\extadjacent)$ (resp.\ $(D;i,\wt{\eta}_2(\wt{\tau}_2),\ol{\wt{\eta}}_2(\ol{\wt{\tau}}_2),-i;\extnested)$).
\end{lemma}

We now review the construction of the bichordal \clek{} in \cite{msw2020nonsimple} in the regime $\kappa \in (8/3,4)$ (resp.\ the bichordal \cle{4} in \cite{mw2018connection}). Recall the definition of the \slekrb{\kappa-6}{1} and \slekrb{\kappa-6}{-1} (resp.\ \slefmu{\mu}) explorations of a monochordal \clek{} given in Definition~\ref{def:mcle_cpi_exploration}.

\begin{lemma}[{\cite[Lemma~3.5]{msw2020nonsimple}, \cite[Section~3]{mw2018connection}}]
\label{le:4strands_simple}
Let $\kappa \in (8/3,4]$. Suppose that $(\eta,\Gamma)$ is a monochordal \clek{} in $(\D;-i,i)$. Let $\eta_2$ be an \slekrb{\kappa-6}{-1} (resp.\ \slefmu{\mu}) exploration of $(\eta,\Gamma)$ starting from $1$ targeting~$-i$. Let $\tau_2$ be a stopping time for~$\eta_2$. Suppose that we are on the event that $\eta_2(\tau_2)$ is not on the trunk of $\eta_2$. Let $\sigma_2$ be the last time before~$\tau_2$ where $\eta_2$ is on its trunk. Let $D$ be the connected component of $\D \setminus \eta_2[0,\tau_2]$ adjacent to $\pm i$. Let $E_1$ be the event that $\eta_2(\tau_2) \in \eta$. Then the conditional probability of $E_1$ given $\eta_2|_{[0,\tau_2]}$ is a.s.\ a function $H_\kappa(\cdot)$ of the conformal modulus of $(D;-i,\eta_2(\tau_2),\eta_2(\sigma_2),i)$ (not depending on the choice of~$\tau_2$ (or $\mu$)).
\end{lemma}

Similarly as before, if we view $\eta$ as the remainder of a loop $\CL_1$, and let $\CL_2$ be the loop that $\eta_2[\sigma_2,\tau_2]$ is part of, then $E_1$ is the event that $\CL_1 = \CL_2$. It was proved in \cite{mw2018connection} that given $E_1$ (resp.\ $E_1^c$), the conditional law of the remainder of $(\CL_1,\CL_2)$ is that of a bichordal \slek{} with interior link pattern $\intadjacent$ (resp.\ $\intnested$).

As before, we can also consider the setup reflected across the vertical axis, in which case we view the \slekrb{\kappa-6}{1} (resp.\ \slefmu{\mu}) as being inside the loop $\CL_1$ that $\eta$ is tracing. Concretely, let $\wt{\eta}_2$ be an \slekrb{\kappa-6}{1} (resp.\ \slefmu{\mu}) exploration of $(\eta,\Gamma)$ starting from $-1$ targeting~$-i$. Let $\wt{\tau}_2$ be a stopping time for~$\wt{\eta}_2$, and let $\wt{\sigma}_2$ be the last time before~$\wt{\tau}_2$ where $\wt{\eta}_2$ is on its trunk.

\begin{lemma}[{\cite{mw2018connection}}]
\label{le:4strands_cle_simple}
Suppose that we are either in the setup of Lemma~\ref{le:4strands_simple} or in the setup reflected with respect to the vertical axis. On the event that the points $i,-i,\eta_2(\tau_2),\eta_2(\sigma_2)$ (resp.\ $i,\wt{\eta}_2(\wt{\sigma}_2),\wt{\eta}_2(\wt{\tau}_2),-i$) are distinct boundary points of a connected component $D$ of $\D \setminus \eta_2[0,\tau_2]$ (resp.\ $\D \setminus \wt{\eta}_2[0,\wt{\tau}_2]$), the conditional law of the remainder of $(\eta,\eta_2,\Gamma\big|_D)$ (resp.\ $(\eta,\wt{\eta}_2,\Gamma\big|_D)$) given $\eta_2|_{[0,\tau_2]}$ (resp.\ $\wt{\eta}_2|_{[0,\wt{\tau}_2]}$) is a.s.\ given by the bichordal \clek{} law in $(D;i,-i,\eta_2(\tau_2),\eta_2(\sigma_2);\extadjacent)$ (resp.\ $(D;i,\wt{\eta}_2(\wt{\sigma}_2),\wt{\eta}_2(\wt{\tau}_2),-i;\extnested)$).
\end{lemma}

We remark that for $\kappa=4$ there is a particularly simple description of the law of the bichordal \cle{4} in $(D;x_1,x_2,x_3,x_4;\extadjacent)$ as follows. Sample $\epsilon \in \{-1,1\}$ with $\p[\epsilon = 1] = (1+H_4)/2$, then sample a GFF $h_\epsilon$ on $D$ with boundary values $2\lambda, 0, \epsilon 2\lambda, 0$, respectively, on the four arcs of $\partial D$ starting with the arc between $x_1,x_2$, and let $(\ul{\eta},\Gamma)$ be the collection of level lines of $h_\epsilon$.

By conformal invariance in Lemma~\ref{le:4strands_cle} and~\ref{le:4strands_cle_simple}, is then easy to see that   the law of the bichordal \clek{} in $(\D;\ul{x};\beta) \in \eldomain{4}$ actually exists
for \emph{any} choice of $\ul{x}$ in $\partial\D$ and $\beta\in\{\extadjacent,\extnested\}$. Indeed, suppose we are given $(\D;\ul{x}) \in \eldomain{4}$ and $\beta=\extadjacent$.  In the settings of Lemma~\ref{le:4strands_cle} and~\ref{le:4strands_cle_simple} there exists a unique conformal map $\varphi \colon D \to \D$ which takes $(D;i,-i,\ol{\eta}_2(\ol{\tau}_2),\eta_2(t);\extadjacent)$ to $(\D;x_1, x_2,x_3,\varphi(\eta_2(t));\extadjacent)$.  We have a positive probability that while exploring $\eta_2$ we realize the marked point configuration $(\D;\ul{x})$ before $\eta_2$ disconnects the other points (see~\cite{msw2020nonsimple}). The same applies to $\beta=\extnested$ as by definition the bichordal \clek{} in $(\D;x_1,x_2,x_3,x_4;\extnested)$ is the same as the bichordal \clek{} in $(\D;x_2,x_3,x_4,x_1;\extadjacent)$.
This also implies that the function $H_\kappa(\cdot)$ in Lemma~\ref{le:4strands} and~\ref{le:4strands_simple} is well-defined for all $(\D;\ul{x};\beta) \in \eldomain{4}$.

Using these arguments, one also gets the following statements.

\begin{lemma}[{\cite[Lemma 3.4]{msw2020nonsimple}}]
\label{lem:function_positive}
The function $H_\kappa(\cdot)$ is bounded away from $0$ and from $1$ on any compact subset of $(0, \infty)$. The function $H_\kappa(L)$ converges to either $0$ or $1$ as $L \to 0$ or $L\to \infty$.
\end{lemma}

\begin{lemma}[{\cite{mw2018connection}}]
\label{lem:bichordal_continuous}
The law of the bichordal \clek{} in the topology~\eqref{eq:topology_loop_ensemble} is a continuous function of $(\D;\ul{x};\beta) \in \eldomain{4}$. In particular, the function $H_\kappa$ is continuous.
\end{lemma}

The function $H_\kappa$ was proved in \cite{mw2018connection} to be given by the formula~\eqref{eq:bichordal_link_probability}. For the proofs in this paper, the exact formula for $H_\kappa$ will not be important.

Next, we review and extend some commutation properties of explorations of monochordal and bichordal \clek{} which will be extensively used in the present work. We begin with the case $\kappa \in (4,8)$.

\begin{lemma}[{\cite[Lemma~3.4]{gmq2021sphere}}]
\label{le:sle_given_cle}
Let $\kappa \in (4,8)$. Suppose $(\eta,\Gamma)$ is a monochordal \clek{} in $(\D;-i,i)$.
Let $I\subseteq\partial\D$ be a connected boundary arc, and $\Gamma(I)$ the collection of loops of $\Gamma$ that intersect $I$. Then, on the event $\{\eta\cap I = \varnothing\}$, the conditional law of $\eta$ given $\Gamma(I)$ is that of an \slek{} in the connected component of $\D\setminus \bigcup\Gamma(I)$ with $-i$, $i$ on its boundary.
\end{lemma}

The analogue for $\kappa \in (8/3,4]$ was essentially proved in \cite{sw2012cle}. It is also a consequence of the following restriction property for \slekrb{\kappa-6}{\beta} which is stated in \cite[Section~2.2]{mw2018connection} and which we restate in Lemma~\ref{le:cpi_restriction} just below. It can be proved, for example, by discovering the loops intersecting (a small neighbourhood of) $K$ iteratively with an \slekrb{\kappa-6}{-\beta} (resp.\ \slefmu{-\mu}) exploration, and using that one can couple an \slekrb{\kappa-6}{\beta} (resp.\ \slefmu{\mu}) from $x \in \partial\D$ to $y \in \partial\D$ with an \slekrb{\kappa-6}{-\beta} (resp.\ \slefmu{-\mu}) from $y$ to $x$ that has the same range and such that their initial segments commute (see \cite[Theorem~7.4]{msw2017cleperc} resp.\ \cite[Proposition~5.3]{msw2017cleperc}).

\begin{lemma}\label{le:cpi_restriction}
Let $\kappa \in (8/3,4]$. Let $\eta$ be an \slekrb{\kappa-6}{\beta} (resp.\ \slefmu{\mu}) from $-i$ to $i$ in~$\D$, and let $\Gamma$ be a \clek{} in~$\D$ coupled with $\eta$. Let $K\subseteq\ol{\D}$ be a compact set such that $-i \notin K$ and $\D \setminus K$ is simply connected. Let $\Gamma(K)$ be the collection of loops of $\Gamma$ that intersect $K$. Then the conditional law of $\eta$ given $\Gamma(K)$ until hitting $\bigcup\Gamma(K)$ is that of an \slekrb{\kappa-6}{\beta} (resp.\ \slefmu{\mu}) in the respective connected component of $\D \setminus \bigcup\Gamma(K)$.
\end{lemma}

We state the following variant of Lemma~\ref{le:sle_given_cle} which incorporates the entire range of $\kappa \in (8/3,8)$.

\begin{lemma}\label{le:sle_given_cle_gen}
Suppose $(\eta,\Gamma)$ is a monochordal \clek{} in $(\D;-i,i)$.
Let $K\subseteq\ol{\D}$ be a compact set such that $\D \setminus K$ is simply connected. Let $\Gamma(K)$ be the collection of loops of $\Gamma$ that intersect $K$. Then, on the event $\{\eta\cap K = \varnothing\}$, the conditional law of $\eta$ given $\Gamma(K)$ is that of an \slek{} in the connected component of $\D\setminus \bigcup\Gamma(K)$ with $-i$, $i$ on its boundary.
\end{lemma}

\begin{proof}
The statement for $\kappa \in (8/3,4]$ follows from the same argument as the proof of Lemma~\ref{le:sle_given_cle} with Lemma~\ref{le:cpi_restriction} as the input.

The statement for $\kappa \in (4,8)$ follows by a repeated application of Lemma~\ref{le:sle_given_cle}. To be precise, first fix $\varepsilon > 0$ and let $K^\varepsilon$ be the $\varepsilon$-neighbourhood of $K$. We first prove the statement with $K^\varepsilon$ in place of $K$. By Lemma~\ref{le:cle_finite_chain}, every loop in $\Gamma(K^\varepsilon)$ is connected to $\partial\D$ via a finite sequence of intersecting loops in $\Gamma(K^\varepsilon)$. Therefore we obtain the statement by iteratively applying Lemma~\ref{le:sle_given_cle} with a countable dense choice of boundary intervals. Finally, observe that $\Gamma(K) = \bigcap_\varepsilon \Gamma(K^\varepsilon)$, therefore $\D\setminus \bigcup\Gamma(K^\varepsilon) \nearrow \D\setminus \bigcup\Gamma(K)$ as $\varepsilon \searrow 0$ by the local finiteness of \clek{}. The statement with $\Gamma(K)$ then follows by the continuity of the \slek{} law in the domain.
\end{proof}

Building on these lemmas, we prove an analogue of Lemma~\ref{le:sle_given_cle_gen} for the bichordal \clek{}. This will be deduced by exploring parts of a \clek{} in different orders (which is also the argument used to prove \cite[Lemma~3.4]{gmq2021sphere}). See Figure~\ref{fi:bichordal_conditioning} for an illustration.

\begin{figure}[ht]
\centering
\includegraphics[width=0.4\textwidth]{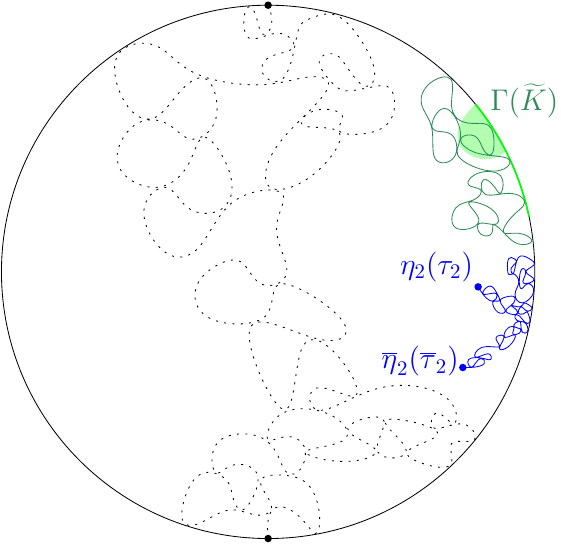}
\caption{The proof of Lemma~\ref{le:bichordal_given_loops}. The two exploration procedures produce the same result, hence the two conditional laws are the same.}
\label{fi:bichordal_conditioning}
\end{figure}

\begin{lemma}\label{le:bichordal_given_loops}
Suppose $(\eta_1,\eta_2,\Gamma)$ is a bichordal \clek{} in $(\D;\ul{x};\beta) \in \eldomain{4}$. Let $K\subseteq\ol{\D}$ be a compact set such that $\D \setminus K$ is simply connected. Let $\Gamma(K)$ be the collection of loops of $\Gamma$ that intersect $K$. Then, on the event that $\eta_1,\eta_2$ do not intersect $K$ and the points in $\ul{x}$ are boundary points of the same connected component $D$ of $\D \setminus \bigcup\Gamma(K)$,  the conditional law of $(\eta_1,\eta_2,\Gamma\big|_D)$ given $\Gamma(K)$ is that of a bichordal \clek{} in $(D;\ul{x};\beta)$.
\end{lemma}

\begin{proof}
Let us first consider $\kappa \in (4,8)$. By the reflection symmetry, we can assume that $\beta = \extadjacent$. By Lemma~\ref{le:4strands_cle}, the law of a bichordal \clek{} in $(\D;\ul{x};\extadjacent) \in \eldomain{4}$ can be obtained from a suitable exploration of a monochordal \clek{}.  Thus, a first approach to obtain the law from the statement is to consider the setup of Lemma~\ref{le:4strands} for a monochordal \clek{} $(\eta_1,\Gamma)$, and explore further $\Gamma(\wt{K})$ where $\wt{K}\subseteq\ol{\D}$ is a compact set (connected either to the boundary arc between $1,i$ or between $i,-i$ depending on whether $K$ is connected to a wired or free arc).

Alternatively, we can start with a monochordal \clek{} in $(\D; -i,i)$ and first explore $\Gamma(\wt{K})$. By Lemma~\ref{le:sle_given_cle_gen}, the remainder of $(\eta_1,\Gamma\big|_{\wt{D}})$ given $\Gamma(\wt{K})$ is again a monochordal \clek{} in the connected component $\wt{D}$ of $\D\setminus\bigcup\Gamma(\wt{K})$ adjacent to $-i$, $i$. We can then explore $\eta_2|_{[0,\tau_2]}$, $\ol{\eta}_2|_{[0,\ol{\tau}_2]}$ as in the setup of Lemma~\ref{le:4strands} applied to $\wt D$.  By the conformal invariance of bichordal \clek{} and monochordal \clek{}, respectively,  on the event where $\eta_1,\eta_2$ do not intersect $\wt{K}$ and $i,-i,\ol{\eta}_2(\ol{\tau}_2),\eta_2(\tau_2)$ are on the same component  $\wt D'$  of $\D\setminus(\bigcup\Gamma(\wt{K})\cup\eta_2([0,\tau_2])\cup\ol{\eta}_2([0,\ol{\tau}_2]))$, the first and second procedures yield the same result.
By  Lemma~\ref{le:4strands_cle} applied to $\wt{D}$,  the joint conditional law of the remainder of $(\eta_1, \eta_2,\Gamma\big|_{\wt D'})$ is a bichordal \clek{} in $(\wt D';i,-i, \ol{\eta}_2(\ol{\tau}_2),\eta_2(\tau_2));\extadjacent)$.

Since we can choose the stopping time $\tau_2$ so that with positive probability the marked domain is conformally equivalent to $(\D;\ul{x};\beta)$, we obtain the statement for a random $K$ given by the conformal image of $\wt{K}$. We can repeat this argument with a countable dense collection of $\wt{K}$. Finally, note that if we take a sequence $K_n \searrow K$, then $\bigcup\Gamma(K_n) \searrow \bigcup\Gamma(K)$ by the local finiteness of \clek{}. The result follows by the continuity of the bichordal \clek{} law, see Lemma~\ref{lem:bichordal_continuous}.

The argument for $\kappa \in (8/3,4]$ is similar where we use instead an \slekrb{\kappa-6}{-1} (resp.\ \slefmu{\mu}) and its restriction property in Lemma~\ref{le:cpi_restriction}.
\end{proof}

\begin{lemma}\label{le:bichordal_given_cpi}
Let $\kappa \in (8/3,4]$. Suppose $(\eta_1,\eta_2,\Gamma)$ is a bichordal \clek{} in $(\D;\ul{x};\beta) \in \eldomain{4}$. Let $\eta$ be an \slekrb{\kappa-6}{1} (resp.\ \slefmu{\mu}) exploration of $(\eta_1,\eta_2,\Gamma)$, and let $\tau$ be a stopping time for $\eta$ such that either $\eta(\tau)$ is on its trunk or $\eta(\tau) \in \ul{x}$. Then, on the event that $\eta(\tau)$ is on its trunk and the points in $\ul{x}$ are boundary points of the same connected component $D$ of $\D \setminus \bigcup\eta[0,\tau]$,  the conditional law of $(\eta_1,\eta_2,\Gamma\big|_D)$ given $\eta[0,\tau]$ is that of a bichordal \clek{} in $(D;\ul{x};\beta)$.
\end{lemma}

\begin{proof}
The proof is analogous to that of Lemma~\ref{le:bichordal_given_loops} where we use that for two boundary points $x,y$ one can couple an \slekrb{\kappa-6}{1} (resp.\ \slefmu{\mu}) from $x$ to $y$ with an \slekrb{\kappa-6}{-1} (resp.\ \slefmu{-\mu}) from $y$ to $x$ so that their initial segments commute \cite[Theorem~7.4]{msw2017cleperc} (resp.\ \cite[Proposition~5.3]{msw2017cleperc}).
\end{proof}

\section{Existence and uniqueness of multichordal \clek{}}
\label{sec:multichordal_ex_uniq}

In this section we prove Theorems~\ref{thm:multichordal},~\ref{th:continuity_mcle}, and~\ref{thm:cle_partially_explored}. We note that the existence part in Theorem~\ref{thm:multichordal_link} is an immediate consequence of Theorem~\ref{thm:multichordal}; the uniqueness part follows by the same proof. We will break the proof of Theorem~\ref{thm:multichordal} into several parts. We will show in Section~\ref{subsec:multichordal_uniqueness} that the law of multichordal \clek{} is unique and conformally invariant. Then we will argue in Section~\ref{subsec:multichordal_existence} that multichordal \clek{} exists for each $(\D;\ul{x};\beta) \in \eldomain{2N}$. We will do this by generalizing the construction of bichordal \clek{}  from~\cite{msw2020nonsimple}, which was based on the exploration of a pair of boundary touching \clek{} loops, to the explorations of several loops, proving also Proposition~\ref{pr:mcle_exploration_path}. As a part of the proof, we will develop a type of Markovian exploration of a multichordal \clek{} and prove some properties of them. In Section~\ref{subsec:continuity_mcle} we will prove Theorem~\ref{th:continuity_mcle} along with several continuity statements for the multichordal \clek{} law. Finally, in Section~\ref{subsec:partially_explored_cles} we prove Theorem~\ref{thm:cle_partially_explored}.

\subsection{Uniqueness of multichordal \clek{}}
\label{subsec:multichordal_uniqueness}

In this subsection we prove the uniqueness of multichordal \clek{}.

\begin{lemma}
\label{lem:unique_multichordal}
Given $(\D;\ul{x};\beta) \in \eldomain{2N}$, the law of the multichordal \clek{}  in $(\D;\ul{x};\beta)$, if it exists, is unique.
\end{lemma}

As a consequence of Lemma~\ref{lem:unique_multichordal}, we obtain the conformal invariance of the family of multichordal \clek{}.

\begin{lemma}\label{lem:conformal_invariant}
The probability measures $(\mcclelaw{D;\ul{x};\beta})$, if they exist, depend only on the conformal class of $(D;\ul{x};\beta) \in \eldomain{2N}$.  
\end{lemma}

\begin{proof}[Proof of Lemma~\ref{lem:conformal_invariant} given Lemma~\ref{lem:unique_multichordal}]
Assume $( D;\ul{ x};\beta) \in \eldomain{2N}$ and $\varphi \colon  D \to \D$ is a conformal transformation taking $x_i$ to ${y}_i$ for each $1 \leq i \leq 2N$. 
Assume $\mu = \mcclelaw{D;\ul{x};\beta}$ satisfies the resampling property from Definition~\ref{def:exterior_link_pattern_multiple_sle}. 
Then it is immediate from the conformal invariance of the bichordal \clek{} in Definition~\ref{def:resampling_kernel} that 
$\varphi_*\mu$ satisfies the resampling property from Definition~\ref{def:exterior_link_pattern_multiple_sle} in $(\D;\ul{y};\beta)$. Hence, by Lemma~\ref{lem:unique_multichordal}, it must agree with $\mcclelaw{\D;\ul{y};\beta}$ if the latter exists. 
\end{proof}

We turn towards proving Lemma~\ref{lem:unique_multichordal}. Consider the space $\Ncurves{D;\ul{x}}$ of $N$-tuples of curves $\ul{\eta} = (\gamma_1,\dots,\gamma_N)$ defined at the beginning of Section~\ref{se:main_msle}.

Consider first the case $\kappa \in (8/3,4]$. We define a Markov kernel $\Phi$ that transitions from a configuration $\ul{\eta}$ to a new configuration $\ul{\wt{\eta}}$ as follows. Choose $1 \le j_1 < j_2 < j_3 < j_4 \le 2N$ uniformly at random. Let $E_{\ul{j}}$ be the event in Definition~\ref{def:resampling_kernel}.
\begin{itemize}
 \item If $E_{\ul{j}}^c$ occurs, then we set $\ul{\wt{\eta}} = \ul{\eta}$.
 \item If $E_{\ul{j}}$ occurs, let $\gamma_i,\gamma_{i'}$ be the curves connecting the points in $x_{\ul{j}}$, and let $(D_{\ul{j}};x_{\ul{j}};\beta_{\ul{j}})$ be as in Definition~\ref{def:resampling_kernel}. We then set $\wt{\gamma}_\ell = \gamma_\ell$ for $\ell \neq i,i'$, and we let $\wt{\gamma}_i$, $\wt{\gamma}_{i'}$ be sampled from the law of a bichordal \clek{} from Definition~\ref{def:resampling_kernel} in $(D_{\ul{j}};x_{\ul{j}};\beta_{\ul{j}})$.
\end{itemize}

When $\kappa \in (4,8)$, we need to be a slightly bit careful with the measurability of the kernel $\Phi$. For this reason, we consider the following variant of the kernel.

For $\ul{\eta} = (\gamma_1,\dots,\gamma_N) \in \Ncurves{D;\ul{x}}$, we consider the curves $\ul{\gamma} = (\gamma_1^L,\gamma_1^R,\ldots,\gamma_N^L,\gamma_N^R)$ where $\gamma_i^L$ (resp.\ $\gamma_i^R$) represents the left (resp.\ right) outer boundary of $\gamma_i$, $i=1,\ldots,N$. It suffices to show that the law of $\ul{\gamma}$ is uniquely determined by the resampling property in Definition~\ref{def:exterior_link_pattern_multiple_sle} since the law of $\gamma_i$ given $(\gamma_i^L,\gamma_i^R)$ is necessarily that of an \slekr{\kappa/2-4;\kappa/2-4} in each of the connected components bounded between $\gamma_i^L,\gamma_i^R$ \cite{ms2016ig1}.

We define a Markov kernel $\Phi$ that transitions from a configuration $\ul{\gamma}$ to a new configuration $\ul{\wt{\gamma}}$ as follows. Choose $1 \le j_1 < j_2 < j_3 < j_4 \le 2N$ uniformly at random. Let $E_{\ul{j}}$ be the event in Definition~\ref{def:resampling_kernel}.
\begin{itemize}
 \item If $E_{\ul{j}}^c$ occurs, then we set $\ul{\wt{\gamma}} = \ul{\gamma}$.
 \item If $E_{\ul{j}}$ occurs, let $\gamma_i^L,\gamma_i^R,\gamma_{i'}^L,\gamma_{i'}^R$ be the curves connecting the points in $x_{\ul{j}}$, and let $(D_{\ul{j}};x_{\ul{j}};\beta_{\ul{j}})$ be as in Definition~\ref{def:resampling_kernel}. We then set $\wt{\gamma}_\ell^q = \gamma_\ell^q$ for $\ell \neq i,i'$ and $q \in \{L,R\}$, and we let $\wt{\gamma}_i^q$, $\wt{\gamma}_{i'}^q$ for $q \in \{L,R\}$ to be the left and right boundaries of the curves obtained by sampling from the law of a bichordal \clek{} from Definition~\ref{def:resampling_kernel} in $(D_{\ul{j}};x_{\ul{j}};\beta_{\ul{j}})$.
\end{itemize}

Let $d_\mathrm{H}$ denote the Hausdorff distance in $\ol{\D}$. We consider the space of configurations $\ul{\gamma} = (\gamma_1^L,\gamma_1^R,\ldots,\gamma_N^L,\gamma_N^R)$ of curves in $(\D;\ul{x})$ such that each pair $\gamma_i^L,\gamma_i^R$ connects two marked points in $\ul{x}$ which lie on the boundary of a connected component of $\D \setminus \bigcup_{i'\neq i} \gamma_{i'}^L \cup \gamma_{i'}^R$. We endow the set of configurations $\ul{\gamma}$ with the distance
\[ d(\ul{\gamma},\ul{\wt{\gamma}}) = \sum_{q \in \{L,R\}} \sum_{i=1}^N d_\mathrm{H}(\gamma_i^q, \wt{\gamma}_i^q).\]
Then $\Phi$ is a Borel-measurable (with respect to the metric above) Markov kernel by the same argument as in~\cite[\S A.1]{msw2020nonsimple}. 

Note that if $\ul{\eta}$ is a multichordal \clek{} according to Definition~\ref{def:exterior_link_pattern_multiple_sle}, then the law of $\ul{\gamma}$ (the collection of left and right boundaries of the curves in $\ul{\eta}$) is invariant under the kernel $\Phi$. Hence, Lemma~\ref{lem:unique_multichordal} is a direct consequence of the following lemma.

\begin{lemma}
\label{lem:unique_invariant}
Given $(\D;\ul{x};\beta) \in \eldomain{2N}$, there is at most one probability measure on sets of non-pairwise crossing arcs $\ul{\gamma}$ in $(\D;\ul{x};\beta)$ which is invariant under $\Phi$.
\end{lemma}

\begin{proof}
We recall that the set of invariant probability measures for a Markov kernel on a Polish space is convex and the ergodic invariant measures are the extremal points of this convex set (\cite[Chapter 6]{varadhan-probability}, for instance). Therefore it suffices to prove that two ergodic invariant measures cannot be mutually singular.

Let $\mu$ and $\wh\mu$ be two $\Phi$-invariant ergodic probability measures. Let $\ul{\gamma}$, $\ul{\wh{\gamma}}$ be sampled independently according to $\mu$, $\wh\mu$, respectively.  We apply $\Phi$ to $\gamma$ and $\ul{\wh{\gamma}}$ independently. When resampling four strands with endpoints in the same component, by Lemma~\ref{lem:function_positive} they can link up in either way with positive probability, thus there is a positive probability that after $N-1$ iterations of $\Phi$ we move from $\ul{\gamma}$ and $\ul{\wh{\gamma}}$ to two configurations $\ul{\gamma}^0$ and $\ul{\wh{\gamma}}^0$, respectively, where both $\ul{\gamma}^0$ and $\ul{\wh{\gamma}}^0$ have the linking pattern $\{\{1,2\},\{3,4\},\ldots,\{2N-1,2N\}\}$. There is also positive probability that each $\gamma_\ell^{0,q}$, $\wh{\gamma}_\ell^{0,q}$ for $q \in \{L,R\}$ is contained in some fixed $\ol{U}_\ell \subseteq \ol{\D}$, and $\ol{U}_{\ell} \cap \ol{U}_{\ell'} = \varnothing$ for $\ell \neq \ell'$ (this can be proven by an induction argument).

At the next iteration, with positive probability, $\Phi$ selects $(j_1,j_2,j_3,j_4) = (1,2,3,4)$ for both $\ul{\gamma}^0$ and $\ul{\wh{\gamma}}^0$ and does not change the linking pattern of $\ul{\gamma}^0$ and $\ul{\wh{\gamma}}^0$. On this positive probability event, by absolute continuity in the domain, we can couple the resampling step so that with positive probability $\eta^{0}_\ell = \wh\eta^{0}_\ell$ for $\ell\in\{1,2\}$, where $\eta^{0}_\ell, \wh\eta^{0}_\ell$ are the resampled paths from the bichordal \clek{} in Definition~\ref{def:resampling_kernel} conditioned on $\alpha'=\intadjacent$, and we still have $\eta^{0}_\ell, \wh\eta^{0}_\ell \subseteq \ol{U}_\ell$ for $\ell\in\{1,2\}$. So that, after the resampling, $\gamma^{0,q}_1=\wh{ \gamma}^{0,q}_1$, $\gamma^{0,q}_2=\wh{ \gamma}^{0,q}_2$, for $q\in\{L,R\}$.

Continuing this way after completing this round of iterations we eventually obtain two configurations $\ul{\gamma}^1$, $\ul{\wh{\gamma}}^1$ that coincide with positive probability. Hence $\mu$ and $\wh{\mu}$ cannot be mutually singular.
\end{proof}

\subsection{Existence of multichordal \clek{}}
\label{subsec:multichordal_existence}

We will  argue in this section that multichordal \clek{} can be constructed for \emph{all} $(\D;\ul{x};\beta) \in \eldomain{2N}$.  This is the statement of Proposition~\ref{prop:multichordal-well-def}, which is the main result of this section. 

\begin{proposition}
\label{prop:multichordal-well-def}
Let $(\D;\ul{x};\beta)\in \eldomain{2N}$. The multichordal \clek{} in $(\D;\ul{x};\beta)$ exists.
\end{proposition}

Until concluding the proof of Proposition~\ref{prop:multichordal-well-def}  we will temporarily denote by $\cleconfclass{2N} \subseteq \eldomain{2N}$ the set of $(D;\ul{x};\beta) \in \eldomain{2N}$ such that there exists at least one law   satisfying the Definition~\ref{def:exterior_link_pattern_multiple_sle} inside $(D;\ul{x};\beta)$.
We have shown in Lemma~\ref{lem:unique_multichordal} that for $(\D;\ul{x};\beta) \in \cleconfclass{2N}$  the law of the multichordal  \clek{} is unique. Hence for 
$(\D;\ul{x};\beta) \in \cleconfclass{2N}$  the law of the  multichordal \clek{} in  $(\D;\ul{x};\beta)$ is well-defined and we will denote it by $\mcclelaw{D;\ul{x};\beta}$.

Our goal is to show $\cleconfclass{2N} = \eldomain{2N}$. We have seen in Section~\ref{subsec:bichordal_cle} that this is true when $N=1,2$. We start in Section~\ref{subsec:markovian_expl}  by describing a \emph{Markovian way} of exploring multichordal \clek{}  and showing that the law of the remainder conditionally on such an exploration is again a  multichordal \clek{}. We will continue in  Section~\ref{subsec:achieveability_config} by showing that with large probability we can conduct  explorations that make it possible to transition from a marked domain in  $\cleconfclass{2N}$ to other neighboring domains. To conclude the proof for $N > 2$, we describe in Section~\ref{sec:proof_existence_multichordal} another partial exploration of multichordal \clek{} on a domain with $2(N-1)$ marked points that produces a multichordal \clek{} in \emph{some} domain with $2N$ marked points. This shows that the set $\cleconfclass{2N}$ is not empty. Using the results from the previous subsections, we then conclude the proof of Proposition~\ref{prop:multichordal-well-def}.

\subsubsection{Markovian explorations  of multichordal  \clek{}}\label{subsec:markovian_expl}

Recall from Section~\ref{se:main_msle} that if  $(\ul\eta,\Gamma)$ is a multichordal \clek{} in $(D;\ul x;\beta) \in \eldomain{2N}$, then we denote by $\eta_k$ the strand in  $\ul\eta$ emanating from $x_k$ for each $k=1,\dots, 2N$.
 
\begin{definition}[Markovian exploration]
\label{def:expl_markovian} 
Given $(D;\ul{x};\beta) \in \eldomain{2N}$, suppose that $(\ul\eta,\Gamma)$ is a multichordal \clek{} in $(D;\ul x;\beta)$. 
One step of the exploration is defined as follows: Pick $1 \leq k \leq 2N$, and let $\tau$ be a stopping time for $\CF_t = \sigma(\eta_k(s) : s \leq t)$.

A Markovian exploration of $\ul{\eta}$ is a collection of random times $\ul{\tau} = (\tau_1,\dots,\tau_{2N})$ such that there exists a sequence of exploration steps $(k_n,\tau^{(n)}_{k_n})_{n \in \N}$ where $1 \leq k_n \leq 2N$ is deterministic and $\tau^{(n)}_{k_n}$ is a stopping time for
\[ \CF^{(n)}_t = \sigma\left(\eta_{k_n}(s) : s \leq t,\ \eta_{k_m}\big|_{[0,\tau^{(m)}_{k_m}]} : m < n \right) \]
and such that $\tau_k$ is the limit of $(\tau^{(n)}_{k_n})_{k_n = k}$ for each $k$.
\end{definition}

\begin{lemma}\label{lem:multichordal_stopping_time}
Suppose that $(D;\ul{x};\beta) \in \cleconfclass{2N}$ and $(\ul\eta,\Gamma)$ is a multichordal \clek{} in $(D;\ul x;\beta)$. 
Let $1 \leq k \leq 2N$, and let $\tau$ be a stopping time for $\CF_t = \sigma(\eta_k(s) : s \leq t)$. Let $(D_\tau; \ul{x}_\tau; \beta_\tau)$ be the triple consisting of the component of $D \setminus \eta_k([0,\tau])$ with $\eta_k(\tau)$ on its boundary, the elements of $\ul{x} \setminus \{x_k\}$ together with $\eta_k(\tau)$ which are on $\partial D_\tau$, and the induced exterior link pattern $\beta_\tau$. Then the conditional law of the remainder of $\ul{\eta}$ given $\eta_k|_{[0,\tau]}$ is that a multichordal \clek{} in $(D_\tau;\ul{x}_\tau;\beta_\tau)$.
\end{lemma}

\begin{proof}
We argue that the remaining law is invariant under the resampling step from Definition~\ref{def:resampling_kernel}.  
Let $j_1<j_2<j_3<j_4$ be given. 
We aim to show that  conditionally on $(\eta_i)_{i \neq j_1,j_2,j_3,j_4}$ the law of the remaining two chords is a bichordal \clek{}. 
If $k\notin\{j_1,j_2,j_3,j_4\}$, it is immediate to conclude. 

Suppose $k \in \{j_1,j_2,j_3,j_4\}$. The law we obtain by
conditioning on $\eta_k|_{[0,\tau]}$ first and then on $(\eta_i)_{i \neq j_1,j_2,j_3,j_4}$ is the same as first conditioning on $(\eta_i)_{i \neq j_1,j_2,j_3,j_4}$ and then exploring $\eta_k|_{[0,\tau]}$.
We know by definition that conditionally on all the strands $(\eta_i)_{i \neq j_1,j_2,j_3,j_4}$, the remaining law is a bichordal \clek{} in the component $D_{\ul{j}}$ containing $x_{j_1},x_{j_2},x_{j_3},x_{j_4}$. By Lemma~\ref{le:4strands_cle} resp.~\ref{le:4strands_cle_simple} the same is true when further conditioning on $\eta_{k}|_{[0,\tau]}$.
\end{proof}

\begin{lemma}\label{lem:limit_exploration}
Suppose that $(D;\ul{x};\beta) \in \cleconfclass{2N}$ and $(\ul\eta,\Gamma)$ is a multichordal \clek{} in $(D;\ul x;\beta)$. 
Let $\ul{\tau}$ be a Markovian exploration of $\ul{\eta}$ as in Definition~\ref{def:expl_markovian}. Then the conditional law of the remainder of $\ul\eta$ given $(\eta_k\big|_{[0,\tau_k]})_{k=1,\dots,2N}$ is given by a multichordal \clek{} independently in each component of $D \setminus \bigcup_k \eta_k[0,\tau_k]$ with the marked points $\eta_k(\tau_k)$ that are on its boundary, and the induced link pattern.
\end{lemma}

In particular, this lemma shows that if $W_{\ul u}$ is the event that  the resulting marked domain after the exploration is conformally equivalent to $(\D; \ul u;\beta)$ and if $(\D; \ul{x};\beta)\in\cleconfclass{2N}$ is such that $\mcclelaw{\D; \ul{x};\beta}\left[ W_{\ul u}\right]>0$ then $(\D; \ul{u};\beta)\in \cleconfclass{2N}$.

\begin{proof}[Proof of Lemma~\ref{lem:limit_exploration}]
 By iteratively applying Lemma~\ref{lem:multichordal_stopping_time}, we know that this is true for any finite number of exploration steps. We want to show that the property remains true when taking limits of exploration steps.
 
 Let $j_1<j_2<j_3<j_4$ be given. We aim to show that conditionally on $(\eta_i)_{i \neq j_1,j_2,j_3,j_4}$ the law of the remaining two chords is a bichordal \clek{}. The law we obtain by conditioning on $(\eta_k\big|_{[0,\tau_k]})_{k=1,\dots,2N}$ first and then on $(\eta_i)_{i \neq j_1,j_2,j_3,j_4}$ is the same as first conditioning on $(\eta_i)_{i \neq j_1,j_2,j_3,j_4}$ and then exploring $(\eta_k\big|_{[0,\tau_k]})_{k=1,\dots,2N}$. By Lemma~\ref{lem:multichordal_stopping_time}, for each $n \in \N$ if we condition on $(\eta_i)_{i \neq j_1,j_2,j_3,j_4}$ and $(\eta_{k_m}\big|_{[0,\tau^{(m)}_{k_m}]})_{m \le n}$, then the law of the remainder is a bichordal \clek{} in the remaining component. Using the fact that $\sigma(\eta_{k_m}\big|_{[0,\tau^{(m)}_{k_m}]} : m \le n) \to \sigma(\eta_k\big|_{[0,\tau_k]} : k=1,\dots,2N)$ and the continuity of the bichordal \clek{} in Lemma~\ref{lem:bichordal_continuous} we conclude that this remains true in the limit.
\end{proof}

\subsubsection{Achieveability of marked configurations}\label{subsec:achieveability_config}

In this subsection we are going to show that with large probability we can transition from configurations in $\cleconfclass{2N}$ to nearby configurations using Markovian explorations of the \clek{} strands. We will first show the achievability result restricted to an auxiliary event that we define now.  

Suppose we have $(\D;{ \ul{x}};\beta)\in \cleconfclass{2N}$, let $(\ul{\eta},\Gamma)$ be a multichordal \clek{} in $(\D; \ul{x};\beta)$, and let $\alpha$ be the internal link pattern of $\ul{\eta}$. We denote by $\ul\eta_{\wh i}$ the tuple formed by $(\eta_j)_{j\neq i,\alpha(i)}$. For $\delta_0 > 0$, we let
\begin{equation}\label{eq:event_F_cle}
F(\delta_0) = \bigcap_{i\in\{1,\dots,2N\}} \left\{ B(x_i,\delta_0) \cap \ul\eta_{\wh i}= \varnothing \right\} .
\end{equation}
In other words, $F(\delta_0)$ is the event that for all $i$ the strand $\eta_i$ does not intersect $B(x_j,\delta_0)$ for $j \neq i,\alpha(i)$.

The following lemma will be the main input to obtain continuity.  In its statement we will show a lower bound for the probability of  the following event:
Suppose $(\ul{\eta}, \Gamma)$ is a multichordal \clek{} in $(\D; \ul{x};\beta)\in \cleconfclass{2N}$, and $(\D; \ul{u})$ is another marked  domain of $2N$ distinct prime ends.   Given a Markovian exploration  of $\ul{\eta}$ as in Definition~\ref{def:expl_markovian}, we denote by    $W_{\ul u}$   the event that  the resulting marked domain after the exploration is conformally equivalent to $(\D; \ul u;\beta)$.

\begin{lemma}\label{lem:continuity_in_points}
There exists a constant $C>0$ depending only on $N$ such that the following is true. Let $\delta_0, \delta_1 > 0$. Suppose $(\D; \ul{x};\beta)\in \cleconfclass{2N}$ and $\abs{x_i-x_j} \ge \delta_0$ for every $i \neq j$. Let $\eta$ be a multichordal \clek{} in $(\D; \ul{x};\beta)$. Let $\ul{u}^*=(u_1^*, \ldots, u_{2N}^*)$ where $u_i^* \in \partial\D \cap B(x_i,\delta_1\delta_0)$ for each $i$.
Then there exists a Markovian exploration of $\ul{\eta}$ that stops before any of the strands exit $B(x_i,\delta_1^{1/2}\delta_0)$ such that
\[\mcclelaw{\D; \ul{x};\beta}\left[ \left({W_{\ul u^*}}\right)^c \cap F(\delta_0)\right] \leq C \delta_1^{1/2} . \]
\end{lemma}

\begin{proof}
We use the notation introduced before the statement. We can assume that $\delta_1 > 0$ is sufficiently small. Throughout the proof, we fix a small constant $c_0 > 0$ (depending on $N$).

We denote by  $\p_{\eta_i}[\cdot\giv \ul\eta_{\wh i}]$ the conditional law $\eta_i$ given $\ul\eta_{\wh i}$. The conditional law is an \slek{} in the connected component $D_i$ of $\D \setminus \ul\eta_{\wh i}$ containing $x_i,x_{\alpha(i)}$.

Let $W_i$ be the event that there exists some $t \in [0,c_0\delta_1\delta_0^2]$ such that $\eta_i([0,t]) \subseteq B(x_i,c_0^{1/2}\delta_1^{1/2}\delta_0)$ and $g_t(\eta_i(t)) = u_i^*$ where $g_t$ is the conformal map from $\D \setminus \eta_i([0,t])$ to $\D$ with $g_t'(0) = e^t$. We claim that, assuming $\abs{x_i-u_i^*} < \delta_1\delta_0$, 
\[
\p_{\eta_i}[W_i^c \mid \ul\eta_{\wh i}] \one_{F(\delta_0)} \lesssim c_0^{-1/2}\delta_1^{1/2} 
\]
where the implicit constant is independent of $\delta_0$.

Indeed, we can consider a conformal map from $D_i$ to $\h$. On the event $\{ B(x_i,\delta_0) \cap \ul\eta_{\wh i} = \varnothing \}$, the derivative of the conformal map is bounded within $B(x_i,\delta_0/2) \cap \D$. Consider the harmonic measure of the left (resp.\ right) side of $\eta_i$ seen from $z_i = (1-\delta_0/2)x_i$. Recall that the Loewner driving function of \slek{} is given by a Brownian motion. By Brownian scaling we conclude that, for $\delta_1$ small, the probability that the harmonic measure reaches $\delta_1$ before time $c_0\delta_1$ is at least $1-O(c_0^{-1/2}\delta_1^{1/2})$.

On the event $W_i$, we have that $\abs{g_t(x_j)-x_j} \lesssim c_0\delta_1\delta_0$ for each $j \neq i$. We denote by $\ul{x}^{(1)}=(x_1^{(1)},\dots,x_{2N}^{(1)})$ the image of the marked points configuration $\ul{x}$ under $g_t$. Note that we still have $\abs{x_j^{(1)}-u_j^*} = O(\delta_1\delta_0)$ for $j \neq i$.

By Lemma~\ref{lem:multichordal_stopping_time} the remainder of the curves is a multichordal \clek{} in $(\D;\ul{x}^{(1)};\beta)$. Therefore we can repeat this procedure for each $i=1,\dots,2N$. Each marked point will be moved by $O(c_0\delta_1\delta_0)$ while exploring $\eta_j$ for $j\neq i$. We obtain a Markovian exploration $(\tau_1,\dots,\tau_{2N})$ of $\ul{\eta}$ such that $\eta_i([0,\tau_i]) \subseteq B(x_i,c_0^{1/2}\delta_1^{1/2}\delta_0)$ and on an event $W$ with $\p[W^c \cap F(\delta_0)] \lesssim c_0^{-1/2}\delta_1^{1/2}$ we have $\abs{x_i^{(2N)}-u_i^*} = O(c_0\delta_1\delta_0)$ for each $i$ where $\ul{x}^{(2N)}=(x_1^{(2N)},\dots,x_{2N}^{(2N)})$ are the image of the tips after mapping to $\D$.

Again, by Lemma~\ref{lem:multichordal_stopping_time} the remainder of the curves is a multichordal \clek{} in $(\D;\ul{x}^{(2N)};\beta)$. Therefore we can repeat this round of explorations with $\delta_1$ replaced by $c_0\delta_1$. We obtain a sequence of Markovian explorations $(\tau_1^{(m)},\dots,\tau_{2N}^{(m)})$ for $m \in \N$ and events $W^{(m)}$ with $\p[(W^{(m)})^c \cap F(\delta_0)] \lesssim c_0^{-1/2}(c_0^{m-1}\delta_1)^{1/2}$ such that $\eta_i([0,\tau_i^{(m)}]) \subseteq B(x_i,O(c_0^{1/2}\delta_1^{1/2}\delta_0))$ and on $W^{(1)}\cap\dots\cap W^{(m)}$ we have $\abs{x_i^{(2N,m)}-u_i^*} = O(c_0^m\delta_1\delta_0)$ for each $m,i$ where $\ul{x}^{(2N,m)}=(x_1^{(2N,m)},\dots,x_{2N}^{(2N,m)})$ are the image of the tips after mapping to $\D$.

For $c_0$ sufficiently small, letting $\tau_i^{(\infty)} = \lim_{m\to\infty} \tau_i^{(m)}$, we conclude that on the event $\bigcap_{m\in\N} W^{(m)}$ we have $\lim_{m\to\infty} x_i^{(2N,m)} = u_i^*$ for each $i$, and by Lemma~\ref{lem:limit_exploration} the conditional law of the remainder of $\ul{\eta}$ given $(\eta_i[0,\tau_i^{(\infty)}])_{i=1,\dots,2N}$ is a multichordal \clek{} in $(\D;\ul{u}^*;\beta)$. That is, $\bigcap_{m\in\N} W^{(m)} \subseteq {W_{\ul u^*}}$ for this exploration, and the sum of the failure probabilities is bounded by
\[ \sum_{m\in\N} \p[(W^{(m)})^c \cap F(\delta_0)] \lesssim \sum_{m\in\N} (c_0^{m-1}\delta_1)^{1/2} \lesssim \delta_1^{1/2} . \]
\end{proof}

At this point we are only able to guarantee $\mcclelaw{\D; \ul{x};\beta}[W_{\ul{u}^*}] > 0$ when $\delta_1$ is chosen smaller than $\mcclelaw{\D; \ul{x};\beta}[F(\delta_0)]^2$ which may depend on $\ul{x}$. In order to prove existence of multichordal \clek{} for each point configuration, however, we want to pick step sizes $\delta_1$ uniformly. We will show below in Lemma~\ref{lem:F_event} that \emph{assuming} the existence of multichordal \clek{} for each point configuration, we can bound $\mcclelaw{\D; \ul{x};\beta}[F(\delta_0)]$ uniformly from below. To break the circularity, we use induction in $N$. Conditionally on one link, we can use the uniform bound on the probabilities of $F(\delta_0)$ for $N-1$ together with conformal invariance to guarantee $\mcclelaw{\D; \ul{x};\beta}[W_{\ul{u}^*}] > 0$ for a uniformly chosen step size $\delta_1 > 0$. This will allow us to choose the step size $\delta_1$ uniformly.

\begin{lemma}\label{lem:fixed_points_multichordal}
Suppose $(\ul\eta,\Gamma)$ is a multichordal $\CLE_\kappa$ in $(\D;\ul{x};\beta)\in\eldomain{2N}$. For every $z\in\ol{\D}$ a.s.\ no strand in $\ul\eta$ hits $z$ except possibly at its endpoints. 
\end{lemma}

\begin{proof}
Conditionally on all the strands in $\ul\eta$ but one, the law of the remaining strand is an \slek{} in the remaining domain. Hence it does not hit fixed points.
\end{proof}

\begin{lemma}\label{lem:cle_strands_in_regions}
Suppose $(\ul\eta,\Gamma)$ is a multichordal CLE in $(\D;\ul{x};\beta)\in\eldomain{2N}$. Let $\gamma_i$, $i = 1,\dots,N$, be fixed simple curves in $\ol{\D}$, each connecting two points in $\ul{x}$, and such that $\gamma_i \cap \gamma_j = \varnothing$ for $i \neq j$. Then, for any open neighborhood $U$ of $\bigcup_i \gamma_i$, the probability that each strand of $\ul\eta$ stays inside $U$ is positive.
\end{lemma}

\begin{proof}
 The proof follows directly using invariance under $\Phi$ as in the proof of Lemma~\ref{lem:unique_invariant},  and~\cite[Lemma~2.3]{mw2017intersections}.
\end{proof}

\begin{lemma}
\label{lem:F_event}
Let $N \in \N$. Assume that multichordal \clek{} exists for every $(\D; \ul{x};\beta) \in \eldomain{2N}$.

For every $\delta>0$, $q>0$ there exists $\delta_0>0$ such that the following is true. Suppose $(\D; \ul{x};\beta)\in \cleconfclass{2N}$ and $\abs{x_i-x_j} \ge \delta$ for every $i \neq j$. Then
\[ \mcclelaw{\D; \ul{x};\beta}[ F(\delta_0)^c ] < q . \]
\end{lemma}

\begin{proof}
We argue that for each $(\D;\ul{x};\beta)\in \eldomain{2N}$ and $q > 0$ there exist $\delta_0 > 0$ and $\delta_1 > 0$ (depending on $\ul{x}$) such that
\[ \mcclelaw{\D; \ul{u};\beta}[ F(\delta_0)^c ] < q \]
for every $\ul{u} \in B(\ul{x},\delta_1)$. By the compactness of the set $\{ \ul{x} \in \partial\D^{2N} \mid \abs{x_i-x_j} \ge \delta \}$ the statement of the lemma will follow.

By Lemma~\ref{lem:fixed_points_multichordal}, there is $\delta_0 > 0$ (possibly depending on $\ul{x}$) such that $\mcclelaw{\D; \ul{x};\beta}[ F(\delta_0)^c ] < q$. Combining this with Lemma~\ref{lem:continuity_in_points}, we find $\delta_1 > 0$ such that for every $\ul{u} \in B(\ul{x},\delta_1)$ there exists a Markovian exploration $\ul{\tau}$ of $\ul{\eta}$ that stops before any of the strands $\eta_i$ exit $B(x_i,\delta_0/10)$ such that $\mcclelaw{\D; \ul{x};\beta}[ W_{\ul{u}}^c ] < 2q$. Letting $\varphi$ be the corresponding conformal map to $\D$, we have that $\varphi^{-1}(B(u_i,\delta_0/2)) \subseteq B(x_i,\delta_0)$. In particular, if we let $\wt{F}(\delta_0/2)$ be the event that $F(\delta_0/2)$ occurs for the remainder of $\ul\eta$ after mapping to $\D$, then $\wt{F}(\delta_0/2)^c \cap W_{\ul{u}} \subseteq F(\delta_0)^c$. By Lemma~\ref{lem:limit_exploration}, we have
\[ 
 \mcclelaw{\D; \ul{x};\beta}[ \wt{F}(\delta_0/2)^c \mid \CF_{\ul{\tau}} ] \,\one_{W_{\ul{u}}}
 = \mcclelaw{\D; \ul{u};\beta}[ F(\delta_0/2)^c ] \,\one_{W_{\ul{u}}} .
\]
Therefore we conclude
\[
 \mcclelaw{\D; \ul{u};\beta}[ F(\delta_0/2)^c ] 
 = \frac{\mcclelaw{\D; \ul{x};\beta}[ \wt{F}(\delta_0/2)^c \cap W_{\ul{u}} ]}{\mcclelaw{\D; \ul{x};\beta}[ W_{\ul{u}} ]}
 \le \frac{\mcclelaw{\D; \ul{x};\beta}[ F(\delta_0)^c ]}{\mcclelaw{\D; \ul{x};\beta}[ W_{\ul{u}} ]}
 \le \frac{q}{1-2q} .
\]
\end{proof}

\begin{lemma}\label{lem:continuity_in_points_ind}
Let $N \in \N$. Assume that multichordal \clek{} exists for every $(\D; \ul{x};\beta) \in \eldomain{2(N-1)}$.

For each $\delta > 0$ there exists a constant $\delta_1>0$ depending only on $\delta,N$ such that the following is true. Suppose $(\D; \ul{x};\beta)\in \cleconfclass{2N}$ and $\abs{x_i-x_j} \ge \delta$ for every $i \neq j$. There exists $p_0 > 0$ (possibly depending on $\ul{x}$) such that the following holds. Let $\eta$ be a multichordal \clek{} in $(\D; \ul{x};\beta)$. Let $\ul{u}^*=(u_1^*, \ldots, u_{2N}^*)$ where $u_i^* \in \partial\D \cap B(x_i,\delta_1)$ for each $i$. 
Then there exists a Markovian exploration of $\ul{\eta}$ that stops before any of the strands exit $B(x_i,\delta/100)$ such that
\[\mcclelaw{\D; \ul{x};\beta}\left[ {W_{\ul u^*}} \right] \geq p_0 . \]
\end{lemma}

\begin{proof}
By the conformal invariance of multichordal \clek{} (Lemma~\ref{lem:conformal_invariant}), we can assume that $x_{2N-1} = i$, $x_{2N} = -i$. Let $E_0$ be the event that $\eta_{2N}$ connects $x_{2N}$ and $x_{2N-1}$ and stays within distance $\delta/10$ to the counterclockwise boundary arc of $\partial\D$ from $x_{2N-1}$ to $x_{2N}$. Let $p_0 = \mcclelaw{\D; \ul{x};\beta}[E_0] > 0$. We have $p_0 > 0$ as a consequence of Lemma~\ref{lem:cle_strands_in_regions} (however at this point we do not know whether $p_0$ depends continuously on $\ul{x}$).

On the event $E_0$, conditionally on $\eta_{2N}$, the remaining curves $\ul\eta_{\wh{2N}}$ have the law of a multichordal \clek{} in $(\wh{D},\wh{\ul x},\wh{\beta})$ where $\wh{D}$ is the connected component of $\D \setminus \eta_{2N}$ with $\wh{\ul x} = (x_1,\dots,x_{2N-2})$ on its boundary and $\wh{\beta}$ is the link pattern induced by $\beta$ and $\eta_{2N}$. Let $\wh{\p}[\cdot \mid \eta_{2N}]$ denote the conditional law. By Lemma~\ref{lem:F_event} applied to $N-1$ and our induction hypothesis and the conformal invariance (Lemma~\ref{lem:conformal_invariant}), we have $\wh{\p}[F(\delta_0) \mid \eta_{2N}] \,\one_{E_0} \ge p_1 \one_{E_0}$ for some $\delta_0,p_1 > 0$ depending only on $\delta$.

Suppose now that $\delta_1 > 0$ and $\ul u^*\in \partial \D^{2N} \cap B(\ul x^*, \delta_1\delta_0)$. Modulo conformal mapping we can assume that $u_{2N-1}=i$, $u_{2N}=-i$. We can then carry out the same proof as for Lemma~\ref{lem:continuity_in_points} with the only difference that we only explore the strands $(\eta_1,\dots,\eta_{2N-2})$, and during the uniformization step we map out only the initial segments of $\eta_1,\dots,\eta_{2N-2}$ and not $\eta_{2N}$, and we always map $x_{2N-1},x_{2N}$ to $i,-i$. Note that the exploration does not depend on $\eta_{2N}$. We obtain that
\[ \wh{\p}[W_{\ul{u}^*}^c \cap F(\delta_0) \mid \eta_{2N}] \,\one_{E_0} \le C\delta_1^{1/2} . \]
Picking $\delta_1 < p_1^2/(2C)^2$ (which depends only on $\delta,N$), we get
\[ 
\mcclelaw{\D; \ul{x};\beta}\left[ {W_{\ul u^*}} \right] 
\geq \mcclelaw{\D; \ul{x};\beta}[E_0]\, \mcclelaw{\D; \ul{x};\beta}[W_{\ul{u}^*} \cap F(\delta_0) \mid E_0] 
\ge p_0(p_1 - C\delta_1^{1/2}) \ge p_0p_1/2 . 
\]
\end{proof}

\begin{lemma}\label{le:mcle_existence_one_to_all}
 Let $N \in \N$. Assume that multichordal \clek{} exists for every $(\D; \ul{x};\beta) \in \eldomain{2(N-1)}$. Then, if $(\D; \ul{y};\beta_0) \in \cleconfclass{2N}$ for some $(\D; \ul{y};\beta_0) \in \eldomain{2N}$, then $(\D; \ul{x};\beta_0) \in \cleconfclass{2N}$ for every marked point configuration $\ul{x}$ of $2N$ points.
\end{lemma}

\begin{proof}
 Let $\delta > 0$ such that $\abs{x_i-x_j} \ge \delta$ and $\abs{y_i-y_j} \ge \delta$ for every $i \neq j$. It follows from Lemma~\ref{lem:continuity_in_points_ind} and Lemma~\ref{lem:limit_exploration} that if $(\D; \ul{y};\beta_0) \in \cleconfclass{2N}$, then $(\D; \ul{u};\beta_0) \in \cleconfclass{2N}$ for every $\ul{u} \in B(\ul{y},\delta_1)$ where $\delta_1 > 0$ depends only on $\delta$. We can find a sequence of marked point configurations, each consisting of $\delta$-separated points, so that we get from $\ul{y}$ to $\ul{x}$ in a finite number of steps. We conclude that $(\D; \ul{x};\beta_0) \in \cleconfclass{2N}$.
\end{proof}

\subsubsection{Proof of existence}
\label{sec:proof_existence_multichordal}

We now turn to proving  Proposition~\ref{prop:multichordal-well-def}. By Lemma~\ref{le:mcle_existence_one_to_all}, to show existence in Proposition~\ref{prop:multichordal-well-def}, it suffices to show that for every $N \in \N$ and exterior link pattern $\beta$ between $2N$ points, the subset of configurations $(D;\ul{x};\beta) \in \cleconfclass{2N}$ on which multichordal \clek{} exists is non-empty. 

\begin{lemma}
\label{lem:confclass_nonempty}
For every $N \in \N$ and $\beta \in \linkpatterns{N}$ there exist distinct points $\ul{x} = (x_1,\ldots,x_{2N})$ given in counterclockwise order on $\partial \D$ so that $(\D;\ul{x};\beta) \in \cleconfclass{2N}$.
\end{lemma}

To obtain an exterior link pattern decorated marked domain $(\D;\ul{x};\beta) \in \cleconfclass{2N}$ in the proof of Lemma~\ref{lem:confclass_nonempty}, we will use the exploration path of a multichordal \clek{} in a domain with $2(N-1)$ marked points. We first prove Propositions~\ref{pr:mcle_exploration_path}, and this will give us the existence of some $(\D;\ul{x};\beta) \in \cleconfclass{2N}$.

\begin{proof}[Proof of Proposition~\ref{pr:mcle_exploration_path}]
The case $N=0$ is clear. Let $N\geq 1$. Let $\tau$ be a stopping time for the exploration path $\eta_{z,w}$, and let $(D_\tau;\ul{x}_\tau;\beta_\tau)$ be the remaining domain as in the proposition statement. Let $\mu$ be the conditional law of the remainder of $(\ul{\eta},\Gamma)$ given $\eta_{z,w}|_{[0,\tau]}$ and $(\ul{\eta},\Gamma)|_{\D \setminus D_\tau}$.
We need to argue that $\mu$ is invariant under the resampling kernel from Definition~\ref{def:resampling_kernel}.

Since $(\ul{\eta},\Gamma)$ is a multichordal \clek{}, by definition, conditioning on all its chords but one, the law of the remainder is given by a monochordal \clek{}, and when conditioning on all its chords but two, the law of the remainder is a bichordal \clek{}.

Suppose that $j_1,j_2,j_3,j_4 \in \{1, \dots, 2(N+1)\}$ are given, and suppose that we are on the event $E_{\ul{j}}$ that the four marked points in $x_{\ul{j}} = (x_{j_1},x_{j_2},x_{j_3},x_{j_4})$ lie on the boundary of the same connected component of $D_\tau \setminus \bigcup_{k \neq j_1,j_2,j_3,j_4} \eta_k$. Let $a,b$ be the marked points $\eta_{z,w}(\sigma)$, $\eta_{z,w}(\tau)$.  There are three possible cases:
\begin{enumerate}[(1)]
\item Both $a,b\notin x_{\ul{j}}$, i.e.\ $x_{\ul{j}}$ all belong to the original marked points $\ul{x}$. On the event $E_{\ul{j}}$ the exploration path $\eta_{z,w}$ has not linked with any of the strands emanating from the points $x_{\ul{j}}$. When conditioning on all the strands in $\ul{\eta}$ but the ones emanating from the points $x_{\ul{j}}$, the law of the remaining strands is a bichordal \clek{}. If we further condition on $\eta_{z,w}|_{[0,\tau]}$ and the strands emanating from $a,b$ (these strands may be already included in the strands explored above), then Lemma~\ref{le:bichordal_given_loops} (resp.\ Lemma~\ref{le:bichordal_given_cpi}) implies that the remaining law is still a bichordal \clek{}.

\item Both $a,b\in x_{\ul{j}}$. Let $c,d$ be the points in $x_{\ul{j}}$ that are not $a,b$. Conditioning on all the strands in $\ul{\eta}$ but the ones emanating from $c,d$ gives a monochordal \clek{}. By Lemma~\ref{le:4strands_cle} (resp.\ Lemma~\ref{le:4strands_cle_simple}) we thus have a bichordal \clek{} when we further condition on $\eta_{z,w}|_{[0,\tau]}$.

\item Only one of $a,b$ is in $x_{\ul{j}}$, assume it is $a$. Let $c,d,e$ be the other three points in $x_{\ul{j}}$. On the event $E_{\ul{j}}$, the point $b$ is linked to some other point $f$. Then $c,d,e,f\in \ul x$, and conditioning on all the strands in $\ul{\eta}$ but the ones emanating from $c,d,e,f$ yields a bichordal \clek{} which we denote by $\Gamma_2$. When we further condition on $\eta_{z,w}[0,\tau]$ and the strand joining $b$ and $f$, we could have equivalently obtained the same object from exploring the strand of $\Gamma_2$ from $f$ until $a$ and exploring the loops discovered by $\eta_{z,w}[0,\tau]$. Hence by Lemma~\ref{le:4strands_cle} and Lemma~\ref{le:bichordal_given_loops} (resp.\ Lemma~\ref{le:4strands_cle_simple} and Lemma~\ref{le:bichordal_given_cpi}) the remaining law is a bichordal \clek{}.
\end{enumerate}
In all three cases we see that conditioning on all the strands emanating from the points not in $x_{\ul{j}}$, the law of the remaining strands is a bichordal \clek{}. That is, the conditional law $\mu$ is invariant under the resampling procedure in Definition~\ref{def:resampling_kernel}.
\end{proof}

\begin{proof}[Proof of Lemma~\ref{lem:confclass_nonempty}]
We proceed by induction. The case $N=1$ is clear. Suppose that for every $\beta \in \linkpatterns{N-1}$ there is a domain $(D;\ul{y};\beta) \in \cleconfclass{2(N-1)}$. Let $\beta_N \in \linkpatterns{N}$. By planarity, there exists a link in $\beta_N$ between two adjacent indices $j,j+1$. Let $\beta_{N-1} \in \linkpatterns{N-1}$ be obtained from $\beta_N$ by removing the vertices $x_{j},x_{j+1}$ and the link between them.

By the induction hypothesis, there is $(D_{N-1};\ul{x}_{N-1};\beta_{N-1}) \in \cleconfclass{2(N-1)}$. Let $z$ be a point on the boundary arc of $\partial D_{N-1}$ on which the additional link in $\beta_N$ lies. By Proposition~\ref{pr:mcle_exploration_path}, exploring $\Gamma$ from $z$ yields a random marked domain $(D_{N};\ul{x}_{N};\beta_{N}) \in \cleconfclass{2N}$ with exterior link pattern $\beta_N$. Hence by conformal invariance (Lemma~\ref{lem:conformal_invariant}) upon conformally mapping to $\D$ we obtain a domain $(\D;\ul{x};\beta_N) \in \cleconfclass{2N}$.
\end{proof}

\begin{proof}[Proof of Proposition~\ref{prop:multichordal-well-def}]
This follows from Lemma~\ref{lem:confclass_nonempty} and Lemma~\ref{le:mcle_existence_one_to_all} by an induction in $N$.
\end{proof}

\subsection{Continuity in the marked point configuration}
\label{subsec:continuity_mcle}

In this subsection we prove Theorem~\ref{th:continuity_mcle} and collect some consequences which will be useful in the following sections.

Using Lemma~\ref{lem:F_event}, we are able to get rid of the auxiliary event in the statement of the continuity  Lemma~\ref{lem:continuity_in_points} and we obtain the following.

\begin{proposition}\label{prop:continuity_in_separated_points}
Let $N \in \N$. For every $\delta>0$, $q>0$ there exists $\delta_1>0$ depending only on $\delta,q,N$ such that the following is true. Suppose $(\D; \ul{x};\beta)\in \eldomain{2N}$ and $\abs{x_i-x_j} \ge \delta$ for every $i \neq j$. Let $\ul\eta$ be a multichordal \clek{} in $(\D; \ul{x};\beta)$. 
Then, for any $\ul{u}^*=(u_1^*, \ldots, u_{2N}^*)$ where $u_i^* \in \partial\D \cap B(x_i,\delta_1)$ for each $i$, there exists a Markovian exploration of $\ul{\eta}$ that stops before any of the strands exit $B(x_i,\delta)$ such that
\[
\mcclelaw{\D; \ul{x};\beta}\left[ \left({W_{\ul u^*}}\right)^c \right] < q 
\]
where $W_{\ul{u}^*}$ is the event defined above Lemma~\ref{lem:continuity_in_points}.
\end{proposition}

\begin{proof}
The proof follows immediately combining  Lemma~\ref{lem:F_event} and Lemma~\ref{lem:continuity_in_points} (and recalling that  $\cleconfclass{2N}=\eldomain{2N}$ by Proposition~\ref{prop:multichordal-well-def}).
\end{proof}

We now obtain the continuity statement in Theorem~\ref{th:continuity_mcle} as a consequence of Proposition~\ref{prop:multichordal-well-def}, Lemma~\ref{lem:limit_exploration}, and Proposition~\ref{prop:continuity_in_separated_points}.

\begin{proof}[Proof of Theorem~\ref{th:continuity_mcle}]
 Let $(\D;\ul{x};\beta) \in \eldomain{2N}$. It suffices to show that any sequence $\ul{y} \to \ul{x}$ contains a subsequence $(\ul{y}^{(n)})$ such that the laws $\mcclelaw{\D;\ul{y}^{(n)};\beta}$ converge weakly to $\mcclelaw{\D;\ul{x};\beta}$.
 
 Let $(\ul\eta,\Gamma)$ be a multichordal \clek{} in $(\D; \ul{x};\beta)$. Let $W_{\ul{y}}$ be the event defined above Lemma~\ref{lem:continuity_in_points}. By Proposition~\ref{prop:continuity_in_separated_points}, we can extract a subsequence $(\ul{y}^{(n)})$ such that there is a Markovian exploration $\ul{\tau}^{(n)}$ of $\ul{\eta}$ for each $n$ with $\ul{\tau}^{(n)} \to \ul{0}$ and such that almost surely $W_{\ul{y}^{(n)}}$ occurs for sufficiently large $n$. On the event $W_{\ul{y}^{(n)}}$ let $(\ul{\eta}^{(n)},\Gamma^{(n)})$ be the image of (the remainder of) $(\ul{\eta},\Gamma)$ to $\D$ under the corresponding conformal map $g_{\ul{\tau}^{(n)}}$. On the event $(W_{\ul{y}^{(n)}})^c$, we can define $(\ul{\eta}^{(n)},\Gamma^{(n)})$ arbitrarily as a multichordal \clek{} in $(\D;\ul{y}^{(n)};\beta)$. By Lemma~\ref{lem:limit_exploration}, the law of $(\ul{\eta}^{(n)},\Gamma^{(n)})$ is a multichordal \clek{} in $(\D;\ul{y}^{(n)};\beta)$, and $(\ul{\eta}^{(n)},\Gamma^{(n)}) \to (\ul{\eta},\Gamma)$ almost surely with respect to the distance \eqref{eq:topology_loop_ensemble}. This shows the result.
\end{proof}

As a consequence of Theorem~\ref{th:continuity_mcle}, we conclude that the probability that the strands of a multichordal \clek{} assume a given link pattern is continuous and uniformly bounded from below for all marked point configurations $\ul{x}$ when the pairwise distances between pairs of marked points are bounded away from $0$.

\begin{corollary}\label{co:hookup_positive_continuous}
 Let $N \in \N$ and $\beta \in \linkpatterns{N}$. Let $\alpha$ be the induced interior link pattern by a multichordal \clek{}. For any given $\alpha_0 \in \linkpatterns{N}$, the function that maps $\ul{x}$ to $\mcclelaw{\D; \ul{x};\beta}[ \alpha = \alpha_0 ]$ is continuous and strictly positive for any marked point configuration $\ul{x}$. The conditional law of $\ul{\eta}$ given $\alpha = \alpha_0$ is equal to $\mcslelaw{\D;\ul{x};\alpha_0}$ from Theorem~\ref{thm:multichordal_link}.
\end{corollary}

\begin{proof}
The continuity follows from Theorem~\ref{th:continuity_mcle}. It remains to argue that the probability that $\alpha = \alpha_0$ is strictly positive. By definition of multichordal \clek{}, the law of $\ul\eta$ is invariant under the kernel $\Phi$ from Subsection~\ref{subsec:multichordal_uniqueness}.  The statement then follows from Lemma~\ref{lem:function_positive} iterating the resampling procedure using $\Phi$ as in the proof of Lemma~\ref{lem:unique_invariant}. 

Finally, the conditional law of $\ul{\eta}$ given $\alpha = \alpha_0$ satisfies the resampling property in Definition~\ref{def:interior_link_pattern_multiple_sle} due to Definition~\ref{def:exterior_link_pattern_multiple_sle} and the definition of the bichordal \slek{}.
\end{proof}

Theorem~\ref{th:continuity_mcle} is also useful for giving us uniform control on the probability that the loops and chords of a multichordal \clek{} behave in a certain way. We give one such statement below.

\begin{corollary}\label{co:unif_prob_marked_points}
 Suppose $p \in (0,1)$, and we have a non-decreasing sequence of events $(E_n)$ such that each $E_n$ is an open set in the topology \eqref{eq:topology_loop_ensemble} and such that for each $(\D;\ul{x};\beta) \in \eldomain{2N}$ we have
 \[ \lim_{n \to \infty} \mcclelaw{\D;\ul{x};\beta}[E_n] > p . \]
 Then, for each $\delta > 0$ we have
 \[
  \lim_{n \to \infty} \inf_{\abs{x_i-x_j} \ge \delta} \mcclelaw{\D;\ul{x};\beta}[E_n] > p 
 \]
 where the infimum is taken over all $(\D;\ul{x};\beta) \in \eldomain{2N}$ with $\abs{x_i-x_j} \ge \delta$ for each $i\neq j$.
\end{corollary}

\begin{proof}
 By Theorem~\ref{th:continuity_mcle} we have
 \[\liminf_{\ul{y}\to\ul{x}} \mcclelaw{\D;\ul{y};\beta}[E_n] \geq  \mcclelaw{\D;\ul{x};\beta}[E_n] > p\] 
 for each $(\D; \ul{x}; \beta) \in \eldomain{2N}$ and sufficiently large $n$. By the compactness of the set of $\ul{x}$ with $\abs{x_i-x_j} \ge \delta$ for each $i\neq j$, the statement holds for sufficiently large $n$ uniformly for all such $\ul{x}$.
\end{proof}

\subsection{Partial explorations of \clek{}}\label{subsec:partially_explored_cles}

In this subsection we prove Theorem~\ref{thm:cle_partially_explored}. The results in this subsection are independent of the existence results in Section~\ref{subsec:multichordal_existence} and can be used as an alternative argument to show that multichordal \clek{} exists for some marked point configuration $\ul{x}$ given $N \in \N$ and $\beta \in \linkpatterns{N}$.

Before giving the proof of Theorem~\ref{thm:cle_partially_explored}  we collect an auxiliary lemma.
Recall from Section~\ref{subsec:cle} that \clek{} is a.s.\ locally finite.  As an immediate consequence of local finiteness of \clek{}  we have that the number of crossings by loops across an annulus is a.s.\ finite, hence we have the following.

\begin{lemma}
\label{lem:finitely-marked-points-marginal}
Let $(\D;\ul{x};\beta) \in \eldomain{2N}$ and let $\Gamma$ be a nested multichordal \clek{} in $(\D;\ul{x};\beta)$.  For every $(U,V) \in \domainpair{D}$, the number of marked points arising on the boundary of the domain $V^{*,U}$ is a.s.\ finite, and they are a.s.\ distinct.
\end{lemma}

\begin{proof}
We first argue for a \clek{} (i.e.\ $N=0$). By local finiteness, the number of crossings of $V \setminus U$  by loops in $\Gamma$ is finite. Hence, a.s.\ only finitely many strands are left partially unexplored by $\Gamma_\outside^{*,V,U}$.

The marked points are clearly distinct when $\kappa \in (8/3,4]$. To see that the marked points are distinct when $\kappa \in (4,8)$, consider the first strand of $\Gamma_\outside^{*,V,U}$ traced by the CLE exploration tree that hits $\partial U$ at some point $z$. The other strands a.s.\ do not hit $z$. Repeating this argument, we see that the marked points of $\Gamma_\outside^{*,V,U}$ are a.s.\ distinct.

For a multichordal \clek{}, we just need to argue that the same property holds for the chords $\ul\eta$ in $\Gamma$. Since each individual chord is an \slek{} given the others, it a.s.\ does not hit given points.
\end{proof}

\begin{proof}[Proof of Theorem~\ref{thm:cle_partially_explored} for $\kappa \in (4,8)$]
Let us argue for the case when $\Gamma$ is a \clek{} (i.e.\ $N=0$). The case when $\Gamma$ is a general multichordal \clek{} follows from the exact same argument.

Consider the partially explored \clek{} process  $\Gamma_\outside^{*,V,U}$ in $D$, let $\beta$ be the exterior link pattern induced by the strands of  $\Gamma_\outside^{*,V,U}$, and let  $\mu$ be the conditional law of the remainder $\Gamma_\inside^{*,V,U}$ of~$\Gamma$. 
We need to show that the conditional law of the remainder $\Gamma_\inside^{*,V,U}$ of $\Gamma$ given $\Gamma_\outside^{*,V,U}$ is a multichordal \clek{} in $(V^{*,U};\ul{x};\beta)$.

To this end, we will argue that $\mu$ is invariant under the resampling kernel in Definition~\ref{def:resampling_kernel}. Let $\ul{\eta}$ be sampled according to $\mu$, and let $\ul{j} = (j_1,j_2,j_3,j_4)$ be given. Suppose that we are on the event $E_{\ul{j}}$ that the points $x_{\ul{j}}$ are on the boundary of the same connected component $D_{\ul{j}}$ of $D \setminus \bigcup_{k \notin \ul{j}} \eta_k$. We need to explain that given $\Gamma_\outside^{*,V,U}$ and $(\eta_k)_{k \notin \ul{j}}$, the conditional law of the remaining two strands is a bichordal \clek{} in $(D_{\ul{j}};x_{\ul{j}};\beta_{\ul{j}})$.

Let us first prove the following statement. For each $\varepsilon > 0$, let $O_\varepsilon$ be the $\varepsilon$-neighborhood of $D \setminus V$. Then, conditionally on $\Gamma_\outside^{*,V,U}$, $(\eta_k)_{k \notin \ul{j}}$, \emph{and the remaining loops $\Gamma(O_\varepsilon) \subseteq \Gamma$ that intersect $O_\varepsilon$}, the remainder of $\Gamma$ has the law of a bichordal \clek{}. Letting $\varepsilon \searrow 0$, we see that every loop that does not intersect $D \setminus V$ will not intersect $O_\varepsilon$ for $\varepsilon$ small enough. Therefore the original claim will follow from the continuity of the bichordal \clek{} law (Lemma~\ref{lem:bichordal_continuous}).

To conclude the proof, we now argue that there is a countable family of explorations of $\Gamma$ so that one of them discovers all and exclusively the loops and strands of $\Gamma_\outside^{*,V,U}$, $(\eta_k)_{k \notin \ul{j}}$, and $\Gamma(O_\varepsilon)$. These explorations will have the property that the conditional laws of the unexplored parts are bichordal \clek{}, which will conclude the proof.

Suppose that $\CL_1,\CL_2$ are the two strands of $\Gamma_\outside^{*,V,U}$ with endpoints in $x_{\ul{j}}$. Since $V$ is assumed to be simply connected, the set $D \setminus V$ is connected to $\partial D$. We can therefore assume (upon relabeling $\CL_1,\CL_2$) that $\CL_2$ does not separate $\CL_1$ from $\partial D$ in $D \setminus V$. By Lemma~\ref{le:cle_finite_chain} each loop and strand in $\Gamma_\outside^{*,V,U}$ can be connected to $\partial D$ through a finite number of loops and strands in $\Gamma_\outside^{*,V,U} \cup \Gamma(O_\varepsilon)$.

We can explore loops of $\Gamma$ repeatedly with \clek{} exploration paths starting from rational boundary points of the so-far explored regions within $O_\varepsilon$. When we discover a loop that intersects $\partial U$, we explore both ends until they hit $\partial U$. Then, for each end, we can choose to either stop the strand or continue exploring until it re-enters $O_\varepsilon$ and hits $\partial U$ again. (After completing such a loop we can continue exploring either on the outside or in either component on the inside of the loop.) By Lemma~\ref{le:cle_finite_chain}, there is one such exploration that discovers $\CL_1$ in finitely many steps. Given this, the remainder of the exploration is a monochordal \clek{} in the unexplored domain. By further exploring in the domain with two marked points following the same procedure, we eventually discover $\CL_2$, and we again explore $\CL_2$ from both ends until they hit $\partial U$.

By iteratively applying Lemmas~\ref{le:sle_given_cle},~\ref{le:4strands_cle}, and~\ref{le:bichordal_given_loops}, we see that the remainder of the exploration has the conditional law of a bichordal \clek{}. Finally, again applying Lemma~\ref{le:cle_finite_chain} and Lemma~\ref{le:bichordal_given_loops}, we can explore the remaining loops in $\Gamma_\outside^{*,V,U}$, $(\eta_k)_{k \notin \ul{j}}$, and $\Gamma(O_\varepsilon)$. The remainder of $\Gamma$ given the exploration is again a bichordal \clek{}. This concludes the proof.
\end{proof}

We now address the case $\kappa \in (8/3,4]$. The proof strategy is similar to the case $\kappa \in (4,8)$, but we need to define the explorations differently. For notational simplicity, we assume in the following that $D$ is bounded (this imposes no loss of generality due to conformal invariance). Let $O_\varepsilon$ be the $\varepsilon$-neighborhood of $D \setminus V$. Fix $z \in O_\varepsilon$. We define an iterative \slekrb{\kappa-6}{1} (resp.\ \slefmu{0}) exploration targeted at $z$ as follows. Let $D_1 = D$ and $\varphi_1\colon D \to \D$ be a conformal map with $\varphi_1(z) = 0$. For each $n \in \N$, we repeat the following steps:
\begin{itemize}
 \item Let $w \in \partial\D$ be chosen to maximise the distance to $\D \setminus \varphi_n(O_\varepsilon \cap D_n)$. Run an \slekrb{\kappa-6}{1} (resp.\ \slefmu{0}) from $\varphi_n^{-1}(w)$ until it exits $O_\varepsilon$. Explore both sides of the incomplete loop until they close or hit $\partial U$.
 \item Stop the exploration if a loop intersecting both $O_\varepsilon$ and $\partial U$ is discovered or if $z$ is swallowed. Otherwise let $D_{n+1}$ be the connected component of the unexplored region containing $z$, let $\varphi_{n+1}\colon D_{n+1} \to \D$ be a conformal map with $\varphi_{n+1}(z) = 0$, and repeat.
\end{itemize}

\begin{lemma}\label{le:cpi_iteration_succeeds}
The exploration described above almost surely terminates in a finite number of steps.
\end{lemma}

\begin{proof}
Note that almost surely $z$ is inside a loop of $\Gamma$. Suppose that the exploration never stops. We claim that there is some $\delta > 0$ such that $\dist(w, \D \setminus \varphi_n(O_\varepsilon \cap D_n)) \ge \delta$ for each $n$. This follows because the probability that a Brownian motion starting from $z$ exits $D_n$ before $O_\varepsilon$ is bounded from below uniformly in $n$. But this yields a contradiction since in each iteration the \slekrb{\kappa-6}{1} (resp.\ \slefmu{0}) runs at least until exiting $B(w,\delta)$, and therefore the conformal radius of $D_n$ from $z$ decreases to $0$ almost surely.
\end{proof}

We also need the \slekrb{\kappa-6}{1} (resp.\ \slefmu{0}) exploration of a monochordal \clek{} as described in Lemma~\ref{le:4strands_simple}. Let $(\eta,\Gamma)$ be a monochordal \clek{} in $(D;x,y) \in \markeddomain{2}$. Let $\partial_R D$ denote the counterclockwise boundary arc from $x$ to $y$. Let $O \subseteq \C$ be a connected open set that intersects $\partial_R D$ (we will later take $O = O_\varepsilon$). Let $U \subseteq D$ be a set with positive distance to $O$. Fix $z \in O$. Let $D_1 = D$ and $\varphi_1\colon D \to \h$ be the conformal map with $\varphi_1(x) = 0$, $\varphi_1(y) = \infty$, and $\im\varphi_1(z) = 1$. For each $n \in \N$, we repeat the following steps:
\begin{itemize}
 \item Let $w \in \R^+$ be chosen to maximise $\dist(w, \h \setminus \varphi_n(O \cap D_n)) \wedge w \wedge 1/\abs{w-z}$. Run an \slekrb{\kappa-6}{1} (resp.\ \slefmu{0}) exploration for the monochordal \clek{} in $(D_n;x,y)$ from $\varphi_n^{-1}(w)$ until it exits~$O$. Explore both sides of the incomplete loop until they close or hit $\partial U$. (If the exploration hits $\eta$, it follows $\eta$ as described in Definition~\ref{def:mcle_cpi_exploration}.)
 \item Stop the exploration if $\eta$ or a loop intersecting both $O$ and $\partial U$ is discovered or if $z$ is swallowed. Otherwise let $D_{n+1}$ be the connected component of the unexplored region containing $z$, let $\varphi_{n+1}\colon D_{n+1} \to \h$ be the conformal map with $\varphi_{n+1}(x) = 0$, $\varphi_{n+1}(y) = \infty$, and $\im\varphi_{n+1}(z) = 1$, and repeat.
\end{itemize}

\begin{lemma}\label{le:cpi_iteration_monochordal_succeeds}
The exploration described above almost surely terminates in a finite number of steps.
\end{lemma}

\begin{proof}
The proof is similar to that of Lemma~\ref{le:cpi_iteration_succeeds}, but the conformal invariance is more restrictive here due to the two marked points.

Suppose that the exploration never stops. As before, there is some $\delta > 0$ such that $\dist(w, \h \setminus \varphi_n(O \cap D_n)) \ge \delta$ for each $n$. Note also that the harmonic measure of $\partial_R D_n$ seen from $z$ is bounded from below uniformly in $n$. We distinguish two cases. If the harmonic measure of $\partial_L D_n$ seen from $z$ decreases to $0$, then the law of the exploration converges to the law of a \slekrb{\kappa-6}{1} (resp.\ \slefmu{0}) in $\h$, and we obtain a contradiction as in the proof of Lemma~\ref{le:cpi_iteration_succeeds}.

On the other hand, if the harmonic measure of $\partial_L D_n$ seen from $z$ stays bounded from below, then $\arg\varphi_n(z)$ is bounded away from $0$ and $\pi$ uniformly in $n$. Moreover, the law of the \slekrb{\kappa-6}{1} (resp.\ \slefmu{0}) exploration is continuous in the point $w$ (because given the chord from $0$ to $\infty$, the conditional law of the \slekrb{\kappa-6}{1} (resp.\ \slefmu{0}) in the complementary component is continuous with respect to moving the point $w$). Therefore we again obtain that the conformal radius of $D_n$ from $z$ decreases to $0$ almost surely, yielding a contradiction.
\end{proof}

\begin{proof}[Proof of Theorem~\ref{thm:cle_partially_explored} for {$\kappa \in (8/3,4]$}]
By conformal invariance, we can assume $D = \D$. We use the same notation as in the proof for $\kappa \in (4,8)$. Again, it suffices to show the statement with the additional conditioning on $\Gamma(O_\varepsilon)$ for each $\varepsilon$. Consider the following countable collection of partial \clek{} explorations defined by iteratively performing the following steps in any order a finite number of times:
\begin{itemize}
 \item Choose some $z \in O_\varepsilon \cap \Q^2$ and perform the iterative \slekrb{\kappa-6}{1} (resp.\ \slefmu{0}) exploration targeting $z$ described above Lemma~\ref{le:cpi_iteration_succeeds} (resp.~\ref{le:cpi_iteration_monochordal_succeeds}) if the currently unexplored region has $0$ (resp.~$2$) marked points.
 \item Explore one of the currently unfinished strands until it links with another marked point or until it re-enters $O_\varepsilon$ and hits $\partial U$ again.
\end{itemize}
Finally, in the unexplored region we explore all the remaining loops in $\Gamma(O_\varepsilon)$ that have not been (partially) discovered so far. By the Lemmas~\ref{le:sle_given_cle_gen},~\ref{le:4strands_cle_simple}, and~\ref{le:bichordal_given_loops}, for each of these explorations, the remainder of $\Gamma$ has the conditional law of a bichordal \clek{}.

Since the \slekrb{\kappa-6}{1} explorations are not determined by the \clek{}, we couple these explorations on a common probability space so that the resulting \clek{} $\Gamma$ is the same.

We need to argue that for each $\ul{j} = (j_1,j_2,j_3,j_4)$, almost surely on the event $E_{\ul{j}}$, at least one of the explorations discovers exactly $\Gamma_\outside^{*,V,U}$, $(\eta_k)_{k \notin \ul{j}}$, and $\Gamma(O_\varepsilon)$ in finitely many steps. Indeed, let $\CL_1,\CL_2$ be the two strands of $\Gamma_\outside^{*,V,U}$ with endpoints in $x_{\ul{j}}$. Since $V$ is assumed to be simply connected, the set $D \setminus V$ (and hence $O_\varepsilon$) is connected to $\partial D$. Each of the strands $\CL_1,\CL_2$ separates $O_\varepsilon$ into at least two components (since they intersect $O_\varepsilon$ and run until $\partial U$ which is outside $O_\varepsilon$). If $z \in O_\varepsilon \cap \Q^2$ is chosen in a component not adjacent to $\partial D$, then by Lemma~\ref{le:cpi_iteration_succeeds} resp.~\ref{le:cpi_iteration_monochordal_succeeds}, applying the exploration targeting $z$ repeatedly will eventually discover $\CL_1$ resp.~$\CL_2$ (the exploration will not swallow $z$ before discovering $\CL_1$ resp.~$\CL_2$, therefore stops only when discovering one of the finitely many loops intersecting both $O_\varepsilon$ and $\partial U$).
\end{proof}

\section{Separation of strands and independence across scales}
\label{sec:strands}

In this section we will prove an independence across scales result for CLE which roughly says that if an event for the restriction of the CLE to an annulus has probability close to $1$, then it is extremely likely that it occurs in a large fraction of concentric annuli. That is, the subsequent annuli can be considered approximately independent. The challenge in making this precise is that when we partially explore a \clek{} on the outside of a ball, the remainder is a \emph{multichordal} \clek{} according to Theorem~\ref{thm:cle_partially_explored}. We will see that when the marked points of the partially explored \clek{} are sufficiently separated, then we can control the probabilities of events for multichordal \clek{} using the continuity arguments from Section~\ref{subsec:continuity_mcle}. However, it is \emph{a priori} not easy to gain control over the behavior of a multichordal \clek{} for a generic choice of $(\D; \ul{x}; \beta) \in \eldomain{2N}$. The main tool for circumventing this will be the coupling of \clek{} with the GFF. We will review the basics of this coupling in Section~\ref{subsec:ig} and then we will state and prove our bounds on the separation of strands in Section~\ref{subsec:separation_bounds}. We then prove the independence across scales results in Section~\ref{subsec:ind_across_scales}.

\subsection{Review of imaginary geometry}
\label{subsec:ig}

We will now review some of the basics of the coupling of SLE with the Gaussian free field (GFF) from \cite{ss2013continuumcontour,s2016zipper,ms2016ig1,ms2017ig4} (see also \cite{dub2009gff}).  In this section, we will sometimes consider $\kappa$ and $16/\kappa$ simultaneously, and in that case we will fix $\kappa \in (0,4)$ and let $\kappa' = 16/\kappa \in (4,\infty)$.  We also let
\[ \lambda = \frac{\pi}{\sqrt{\kappa}},\quad \lambda' = \frac{\pi}{\sqrt{\kappa'}},\quad\text{and}\quad \chi = \frac{2}{\sqrt{\kappa}} - \frac{\sqrt{\kappa}}{2}.\]
We assume that the reader has some familiarity with the GFF; see \cite{s2007gff} for an introduction.

\subsubsection{Flow lines}
\label{subsubsec:flowlines}

Suppose that $h$ is a GFF on $\h$ with boundary conditions given by $-\lambda$ (resp.\ $\lambda$) on~$\R_-$ (resp.\ $\R_+$).  It is shown in \cite{s2016zipper,ms2016ig1} that there exists a unique coupling of $h$ with an \slek{} curve~$\eta$ from~$0$ to $\infty$ so that the following is true.  Let $(g_t)$ be the Loewner flow for $\eta$, $W$ the associated driving function, and let $f_t = g_t - W_t$ be the centered Loewner flow.  Then for every a.s.\ finite stopping time $\tau$ for $\eta$ we have that $\eta([0,\tau])$ is a local set for $h$ and
\[ h \circ f_\tau^{-1} - \chi \arg (f_\tau^{-1})' \stackrel{d}{=} h.\]
Moreover, in this coupling we have that $\eta$ is a.s.\ determined by $h$ and we refer to $\eta$ as the \emph{flow line} of $h$ from $0$ to $\infty$.

More generally, suppose that we have fixed points
\begin{equation}
\label{eqn:boundary_points}
-\infty = x_{\ell+1}^L < x_\ell^L < \cdots < x_1^L \leq x_0^L = 0 = x_0^R < x_1^R < \cdots < x_r^R < x_{r+1}^R = +\infty
\end{equation}
on $\partial \h$ as well as $a_i^q \in \R$ for $q=L$, $0 \leq i \leq \ell+1$ and $q=R$, $0 \leq i \leq r+1$ where $a_0^L = -\lambda$ and $a_0^R = \lambda$.  Let $h$ be a GFF on $\h$ with boundary conditions given by $a_i^L$ in $(x_{i+1}^L,x_i^L]$ for $0 \leq i \leq \ell$ and by $a_i^R$ in $[x_i^R,x_{i+1}^R)$ for $0 \leq i \leq r$.  Let $\ul{\rho} = (\ul{\rho}^L;\ul{\rho}^R)$ where
\[ \rho_i^L = \frac{a_{i-1}^L - a_i^L}{\lambda} \quad\text{for}\quad 1 \leq i \leq \ell \quad\text{and}\quad \rho_i^R = \frac{a_i^R - a_{i-1}^R}{\lambda} \quad\text{for}\quad 1 \leq i \leq r.\] 
Then there exists a unique coupling of $h$ with an \slekr{\ul{\rho}} process $\eta$ in $\h$ from $0$ to $\infty$ with force points located at $x_\ell^L < \cdots < x_1^L$ and $x_1^R < \cdots < x_r^R$, defined up until it hits the continuation threshold, so that the following is true.  Let $(f_t)$ be the centered Loewner flow for $\eta$, let $W$ the associated driving function, and let $\tau$ be an a.s.\ finite stopping time for $\eta$.  Then $\eta([0,\tau])$ is a local set for $h$ and $h \circ f_\tau^{-1} - \chi (\arg f_\tau^{-1})'$ is a GFF on $\h$ with boundary conditions given by $a_0^L = -\lambda$ in $(f_\tau(x_0^L), 0^-]$, $a_i^L$ in $(f_\tau(x_{i+1}^L),f_\tau(x_i^L)]$ for $1 \leq i \leq \ell$ and by $a_0^R = \lambda$ in $[0^+, f_\tau(x_0^R))$, $a_i^R$ in $[f_\tau(x_i^R),f_\tau(x_{i+1}^R))$ in $(x_i^R,x_{i+1}^R]$ for $1 \leq i \leq r$.  Moreover, in this coupling we have that $\eta$ is a.s.\ determined by $h$ and we refer to~$\eta$ as the \emph{flow line} of~$h$ from~$0$ to~$\infty$.

Using the change of coordinates formula for imaginary geometry, we can define the flow lines for GFFs on general simply connected domains $D \subseteq \C$ connecting $x,y \in \partial D$ distinct.  To this end, let $\varphi \colon D \to \h$ be a conformal transformation with $\varphi(x) = 0$ and $\varphi(y) = \infty$.  Set $\wt{h} = h \circ \varphi - \chi \arg \varphi'$.  Then $\wt{\eta} = \varphi^{-1}(\eta)$ is the flow line of $\wt{h}$ from $x$ to $y$.

Using the change of coordinates formula, it is possible to define the flow line of a GFF $h$ on $\h$ connecting any two boundary points (by conformally mapping them to $0$ and $\infty$).  For each $\theta \in \R$ we can also define the flow line of $h$ with angle $\theta$ starting from $x$ to be the flow line of $h + \theta \chi$ starting from $x$.  The manner in which the flow lines interact with each other is described in \cite{ms2016ig1}.  In particular, suppose that $x_1 < x_2$, $\theta_1,\theta_2 \in \R$, and $\eta_{x_i}^{\theta_i}$ is the flow line of $h$ from $x_i$ to $\infty$ of angle $\theta_i$.  If $\theta_1 > \theta_2$, then $\eta_{x_1}^{\theta_1}$ stays to the left of $\eta_{x_2}^{\theta_2}$.  If $\theta_1 = \theta_2$, then $\eta_{x_1}^{\theta_1}$ merges with $\eta_{x_2}^{\theta_2}$ upon intersecting and does not subsequently separate.  Finally, if $\theta_1 \in (\theta_2-\pi,\theta_2)$, then $\eta_{x_1}^{\theta_1}$ crosses $\eta_{x_2}^{\theta_2}$ from left to right upon intersecting.

\subsubsection{Counterflow lines}
\label{subsubsec:counterflowlines}

We can also couple \slekp{} curves with the the GFF.  In this case, they are referred to as \emph{counterflow lines} (rather than flow lines) because of the interpretation of the coupling.

Suppose that $h$ is a GFF on $\h$ with boundary conditions given by $\lambda'$ (resp.\ $-\lambda'$) on $\R_-$ (resp.\ $\R_+$).  Then it is shown in \cite{s2016zipper,ms2016ig1} that there exists a unique coupling of $h$ with an \slekp{} curve $\eta'$ from $0$ to $\infty$ so that the following is true.  Let $(g_t)$ be the Loewner flow for $\eta'$, $W$ the associated driving function, and $f_t = g_t - W_t$ be the centered Loewner flow.  Then for every a.s.\ finite stopping time $\tau$ for $\eta'$ we have that $\eta'([0,\tau])$ is a local set for $h$ and
\[ h \circ f_\tau^{-1} - \chi \arg (f_\tau^{-1})' \stackrel{d}{=} h.\]
Moreover, in this coupling we have that $\eta'$ is a.s.\ determined by $h$ and we refer to $\eta'$ as the \emph{counterflow line} of $h$ from $0$ to $\infty$.

More generally, suppose that we have fixed boundary points as in~\eqref{eqn:boundary_points} as well as $a_i^q \in \R$ for $q=L$, $0 \leq i \leq \ell+1$ and $q=R$, $0 \leq i \leq r+1$ where $a_0^L = \lambda'$ and $a_0^R = -\lambda'$.  Let $h$ be a GFF on $\h$ with boundary conditions given by $a_i^L$ in $(x_{i+1}^L,x_i^L]$ for $0 \leq i \leq \ell$ and by $a_i^R$ in $[x_i^R,x_{i+1}^R)$ in $(x_i^R,x_{i+1}^R]$ for $0 \leq i \leq r$.  Let $\ul{\rho} = (\ul{\rho}^L;\ul{\rho}^R)$ where
\[ \rho_i^L = \frac{a_i^L - a_{i-1}^L}{\lambda'} \quad\text{for}\quad 1 \leq i \leq \ell \quad\text{and}\quad \rho_i^R = \frac{a_{i-1}^R - a_i^R}{\lambda'} \quad\text{for}\quad 1 \leq i \leq r.\] 
Then there exists a unique coupling of $h$ with an \slekpr{\ul{\rho}} process $\eta'$ in $\h$ from $0$ to $\infty$, defined up until it hits the continuation threshold, so that the following is true.  Let $(f_t)$ be the centered Loewner flow for $\eta'$, let $W$ the associated driving function, and let $\tau$ be an a.s.\ finite stopping time for $\eta'$.  Then $\eta([0,\tau])$ is a local set for $h$ and $h \circ f_\tau^{-1} - \chi (\arg f_\tau^{-1})'$ is a GFF on $\h$ with boundary conditions given by $a_0^L = \lambda'$ in $(f_\tau(x_0^L), 0^-]$, $a_i^L$ in $(f_\tau(x_{i+1}^L),f_\tau(x_i^L)]$ for $1 \leq i \leq \ell$ and by $a_0^R = -\lambda'$ in $[0^+, f_\tau(x_0^R))$, $a_i^R$ in $[f_\tau(x_i^R),f_\tau(x_{i+1}^R))$ in $(x_i^R,x_{i+1}^R]$ for $1 \leq i \leq r$.  Moreover, in this coupling we have that $\eta'$ is a.s.\ determined by $h$ and we refer to~$\eta$ as the \emph{counterflow line} of~$h$ from~$0$ to~$\infty$.

Using the change of coordinates formula for imaginary geometry, we can define the counterflow lines for GFFs on general simply connected domains $D \subseteq \C$ connecting $x,y \in \partial D$ distinct.  To this end, let $\varphi \colon D \to \h$ be a conformal transformation with $\varphi(x) = 0$ and $\varphi(y) = \infty$.  Set $\wt{h} = h \circ \varphi - \chi \arg \varphi'$.  Then $\wt{\eta}' = \varphi^{-1}(\eta')$ is the counterflow line of $\wt{h}$ from $x$ to $y$.

Using the change of coordinates formula, it is possible to define the counterflow line of a GFF $h$ on $\h$ connecting any two boundary points (by conformally mapping them to $0$ and $\infty$).  The manner in which the counterflow lines and flow lines of a GFF interact is also described in \cite{ms2016ig1}.  Suppose that $h$ is a GFF on $\h$ with piecewise constant boundary data.  One convenient framework to describe this is in the case that the counterflow line $\eta'$ of $h$ under consideration travels from $\infty$ to $0$.  In this case, its left (resp.\ right) boundary (as viewed from $0$) is equal to the flow line of $h$ from $0$ to $\infty$ with angle $\pi/2$ (resp.\ $-\pi/2$).  More generally, $\eta'$ contains the range of any flow line with angle in $[-\pi/2,\pi/2]$ from $0$ to $\infty$.

\subsubsection{Level lines}

In the special case $\kappa = 4$, we have $\chi = 0$, and the flow lines are referred to as \emph{level lines} of the GFF. The counterflow lines of $h$ are then just the level lines of $-h$. The level lines enjoy some special properties, for instance the reversal of the level line of $h$ from $x$ to $y$ is the level line of $-h$ from $y$ to $x$. Also, the change of coordinates becomes particularly simple since $\chi = 0$. For $c \in \R$, we define the level line of $h$ of height $c$ as the level line of $h-c$.

\subsubsection{Flow lines from interior points}

When $\kappa \in (0,4)$, it is also possible to start flow lines of the GFF starting from interior points \cite{ms2017ig4}.  The easiest setting in which one can consider these flow lines is with a whole-plane GFF.  A whole-plane GFF does not have well-defined values and one must consider it modulo a global additive constant.  Since the flow lines of a GFF only depend on the values of the field modulo a global multiple of $2\pi \chi$, it is thus natural to consider the whole-plane GFF~$h$ modulo a global multiple of $2\pi \chi$.  In this case, the flow line of $h$ from $0$ to $\infty$ is a whole-plane \slekr{2-\kappa} process.  Just like for interior flow lines, these flow lines are local sets for $h$ and the boundary conditions of the field given the flow line up to a stopping time take the same form.  One can more generally define the flow line starting from an interior point and a given angle to be the flow line of the field plus $\theta \chi$ starting from that interior point.  These flow line interact with each other in the same way as flow lines starting from boundary points.  Flow lines of the GFF starting from interior points are also defined for the GFF on a domain other than $\C$.  In this case, their law is absolutely continuous with respect to whole-plane \slekr{2-\kappa}.

\subsubsection{Space-filling \slekp{}}
\label{subsubsec:space_filling_sle}

We can use the flow lines of the GFF starting from interior points to define an ordering of space.  Namely, suppose that $h$ is a whole-plane GFF with values modulo a global multiple of $2\pi \chi$.  For each $z \in \C$ with rational coordinates we let $\eta_z^L$ be the flow line of $h$ with angle $\pi/2$.  Then for such $z,w$ distinct, we have that $\eta_z^L$, $\eta_w^L$ a.s.\ merge.  We say that $z$ comes before (resp.\ after) $w$ if $\eta_z^L$ merges into $\eta_w^L$ on its right (resp.\ left) side.  Then it is shown in \cite{ms2017ig4} that there exists a space-filling curve $\eta'$ which visits these points according to this ordering.  This gives the definition of space-filling \slekp{}.  One can similarly order space using the flow lines $\eta_z^R$ of angle $-\pi/2$ and one gets the same ordering.  In fact, $\eta_z^L$ and $\eta_z^R$ together give the left and right boundaries of $\eta'$ stopped upon hitting $z$.

The ordering induced by space-filling \slekp{} also makes sense if one considers a GFF on a domain in $\C$ rather than all of $\C$ if the boundary conditions are appropriate.  If one targets it at any fixed boundary point then it agrees with the counterflow line targeted at that boundary point.  It also agrees with the space-filling path associated with \clekp{} in the coupling described in Section~\ref{subsubsec:cle_gff} below, and this is how the local finiteness of \clekp{} was proved in \cite{ms2017ig4}.

\subsubsection{Coupling with \clekp{}}
\label{subsubsec:cle_gff}

For $\kappa' \in (4,8)$, the \clekp{} is naturally coupled with the GFF using the framework described just above.  In particular, suppose that $h$ is a GFF on $\h$ with boundary conditions given by $-\lambda' + \pi \chi$ on $\partial \h$.  For each $y \in \partial \h$ we let $\eta_y'$ be the counterflow line of $h$ from $\infty$ to $y$.  Then $\eta_y'$ is an \slekpr{\kappa'-6} in $\h$ from $\infty$ to $0$ with its force point located infinitesimally to the left of $\infty$ (when standing at $\infty$ and looking towards $0$).  Moreover, if $(y_n)$ is any countable dense subset of $\partial \h$ then we have that the counterflow lines $\eta_{y_n}'$ are coupled together in the same way as the branches of the exploration tree of a \clekp{} as described in Section~\ref{subsec:cle}.  This gives us the coupling of the boundary intersecting loops of the \clekp{} with $h$.  Iterating this construction in each of the complementary components gives the entire nested \clekp{}.  (We also note that if we were instead to take the boundary conditions to be given by $\lambda'-\pi \chi$, then the counterflow line from $\infty$ to any $y \in \partial \h$ is an \slekpr{\kappa'-6} in $\h$ from $\infty$ to $y$ with the force point infinitesimally to the right of $\infty$ when standing at $\infty$ and looking towards $0$.)

\subsubsection{Coupling with \clek{}}
\label{subsubsec:simple_cle_gff}

We now recall a coupling of \clek{}, $\kappa \in (8/3,4)$, with a GFF established in \cite{msw2017cleperc}. Assume that $\kappa \in (8/3,4)$ so that $\kappa' \in (4,6)$.

Let $h$ be a GFF in $\h$ with boundary values $\lambda',\lambda'-2\pi\chi$ on $\R^-,\R^+$, respectively. Consider the counterflow lines of $h$ from $0$ targeted at a countable dense set of boundary points. (Here the counterflow line is not viewed upside down; it would have angle $\pi$ when viewed upside down.) This exploration tree defines the collection of \bclekp{0} by the procedure described in Section~\ref{se:nonsimple_cle} (in this case, it is exactly the collection of boundary-intersecting loops of a \clekp{}). One can also define from the same exploration tree the collection of \emph{false loops} of the \bclekp{0} by the reflected procedure (which is a collection of boundary-intersecting loops with a different law, namely a \bclekp{\kappa'-6}). We say that the complementary connected components that are surrounded clockwise by the exploration tree are inside the true loops of the \bclekp{0}, and the complementary connected components that are surrounded counterclockwise are inside the false loops of the \bclekp{0}.

In each connected component $D$ inside the true loops of the \bclekp{0}, consider the angle $5\pi/2$ flow lines of $h|_D$ targeted at a countable dense set of boundary points of $D$. The boundaries of the complementary connected components that are surrounded counterclockwise define a \bclek{-\kappa/2} in $D$. The boundaries of the remaining complementary connected components (which are surrounded clockwise) define the \emph{false loops} of the \bclek{-\kappa/2}. The collection of true loops of the \bclek{-\kappa/2} for each such $D$ form a part of the \clek{} loops. It is shown in \cite[Theorem~7.4]{msw2017cleperc} that this construction is equivalent to the \slekrb{\kappa-6}{1} exploration of the \clek{}.

Note that each true loop of the \bclek{-\kappa/2} starting from its first point visited by its exploration tree is also a flow line of $h$ (due to the interaction rules \cite{ms2016ig1} since it is reflected in the opposite direction of the counterflow line).

If we repeat the construction inside each of the false loops of the \bclekp{0} and of the \bclek{-\kappa/2}, then the collection of true loops of the \bclek{-\kappa/2} obtained from all iterations define a non-nested \clek{}. If we further repeat the construction inside each of the loops of the \clek{}, we obtain a nested \clek{}.

Note that the boundary values of $h$ in each of the false loops is, after mapping to $\h$, again given by $\lambda',\lambda'-2\pi\chi$ on $\R^-,\R^+$, respectively (modulo $2\pi\chi$). Inside the \clek{} loops, the boundary values of $h$, after mapping to $\h$, are given by $\lambda'-\pi\chi,\lambda'-3\pi\chi$ on $\R^-,\R^+$, respectively (modulo $2\pi\chi$). Recall that in the definition of the \slekrb{\kappa-6}{1} exploration tree (see Section~\ref{se:simple_cle}), the orientation in each nesting level of the \clek{} is swapped, so that we consider the \slekrb{\kappa-6}{-1} instead of the \slekrb{\kappa-6}{1}. In the coupling with the GFF, inside the components surrounded by an odd number of \clek{} loops, we thus let the angle of the counterflow line be $3\pi$ (again not viewed upside down; or $2\pi$ when viewed upside down), and the angle of the flow lines defining the next level of the nested \clek{} is then $\pi/2$. In this way, the angles of the flow lines defining the \clek{} loops in each nesting level is given by $\pi/2$ (modulo $2\pi$).

\subsubsection{Coupling with \cle{4}}
\label{subsubsec:cle4_gff}

A \cle{4} can be coupled with a GFF in a similar way as in Section~\ref{subsubsec:simple_cle_gff}, however one needs to use a symmetric version. The \slefmu{0} exploration tree is obtained by iterating instances of \bcle{4}{-1}. (A similar construction works for the \slefmu{\mu} exploration.) The \cle{4} loops are given by the collection of level lines of heights in $(2\Z+1)\lambda$. See \cite[Section~5]{msw2017cleperc} for details.

\subsubsection{Good scales for the $\mathrm{GFF}$}\label{subsec:gff_good_scales}
In this subsection we recall the definition of the $M$-good scales for a GFF from~\cite[\S4.1-4.2]{mq2020notsle}, along with some results from~\cite{mq2020notsle} regarding their properties.  
We will also introduce a version for scales around boundary points that will also be useful in the rest of the paper.

\textit{Interior case.} Let $h$ be a GFF in a domain $D\subseteq \C$ with some fixed boundary values. In the case where $h$ is a whole-plane GFF in $\C$, we consider $h$ as defined modulo $2\pi\chi$. 
For any $z \in \C$, $r > 0$ such that $B(z,r)\subset D$,  we let $\CF_{z,r}$ be the $\sigma$-algebra generated by the values of $h$ outside of $B(z,r)$. 
 By the Markov property of the GFF, we can write $h = h_{z,r} + \Fh_{z,r}$ where $h_{z,r}$ is a GFF on $B(z,r)$ with zero boundary conditions, $\Fh_{z,r}$ is a distribution on $D$ which is harmonic in $B(z,r)$, and $h_{z,r}$ and $\Fh_{z,r}$ are independent of each other.  Note that $\Fh_{z,r}$ is measurable with respect to $\CF_{z,r}$ and $h_{z,r}$ is independent of $\CF_{z,r}$.  For $M > 0$ fixed, we say that $ (z,r)$ is $M$-good for $h$ if:
\begin{equation}\label{eq:M_good_scale}\sup_{w \in B(z,15r/16)} | \Fh_{z,r}(w) - \Fh_{z,r}(z)| \leq M.\end{equation}
We note that since $\Fh_{z,r}$ is harmonic in $B(z,r)$ we have that $\Fh_{z,r}(z)$ is equal to the average~$h_r(z)$ of~$h$ on $\partial B(z,r)$.  Let $E_{z,r}^M$ be the event that $ (z,r)$ is $M$-good, then  $E_{z,r}^M$ is $\CF_{z,r}$-measurable.  As we will see in Section~\ref{subsec:separation_bounds},  the $M$-good scales are useful because then the Radon-Nikodym derivative which compares the law of $h$ in $B(z,r)$ away from $\partial B(z,r)$ to a field on $B(z,r)$ with boundary conditions comparable to $\Fh_{z,r}(z) = h_r(z)$ has finite moments of all orders.  
We will recall the proof of this shortly. Since flow lines of the GFF starting from interior points only depend on the field modulo $2\pi \chi$, we can in fact compare to a GFF with bounded boundary conditions.
  
We recall the following result from~\cite{mq2020notsle} regarding the density of good scales.

\begin{lemma}\label{lem:M_good_density}Suppose that $h$ is a $\mathrm{GFF}$ in $D$ with some bounded boundary conditions. For every $a\in(0,1)$ and $b>0$ there exist $M>0$ and $c>0$ so that the following holds. Let $z\in D$ and $r>0$ such that $B(z,r)\subseteq D$. Then for every $k\in\N$, off an event of probability $ce^{-bk}$, the number of $j=1,\ldots,k$ such that $E^{M}_{z,r2^{-j}}$ occurs is at least $(1-a)k$. 
\end{lemma}
\begin{proof} In the case of a whole-plane GFF this is~\cite[Proposition~4.3]{mq2020notsle}. To see that this holds also for a GFF in a domain, let us write the whole-plane GFF as an independent sum $h^D+h_D$ where $h_D$ is a zero-boundary GFF on $D$ and $h^D$ is a random distribution that is harmonic in $D$. The probability that $\sup_{w \in B(z,r/2)} \abs{h^D(z)-h^D(w)} \le 1$ is bounded from below. Therefore the statement of Lemma~\ref{lem:M_good_density} for $h_D$ follows from the whole-plane version. 
\end{proof}

As mentioned, at good scales $(z,r)$ we can compare the values of $h$ away from $\partial B(z,r)$ to those of a GFF with prescribed  boundary conditions, moreover the Radon-Nikodym derivatives between these laws are bounded in $L^p$ for every $p\in(1,\infty)$~\cite[\S4.1-4.2]{mq2020notsle}. We briefly recall the argument.
Let $ {\Fg}$ be a function which is harmonic in $\D$ with some prescribed boundary values.    Let $\psi_{z,r}(w) = r^{-1}(w -z)$ so that $\psi_{z,r}$ maps $B(z,r)$ to $\D$ and $\Fg_{z,r} = \Fg \circ \psi_{z,r}$.  Let $\phi$ be a fixed $C_0^\infty$ function such that $\phi|_{B(0, 7/8)} \equiv 1$, $\phi|_{B(0,15/16)^c} \equiv 0$, and let $\phi_{z,r} = \phi \circ  \psi_{z,r}$.  

In order to    interpolate between the field $h$ and the field $h-\Fh_{z,r}+\Fg_{z,r}$ (which has boundary values given by $\Fg_{z,r}$), we   will consider the following field $ \wt {h}_{z,r}$.
Let\[ C_{z,r}=2\pi\chi \lfloor \Fh_{z,r}(z)/(2\pi\chi)\rfloor,\]
and
\[ \wt {h}_{z,r} = h|_{B(z,r)}-C_{z,r}  +(\Fg_{z,r} - (\Fh_{z,r} - C_{z,r}))(1-\phi_{z,r}).\]
We note that $\wt h_{z,r} $ agrees with $h$ in $B(z,7r/8)$ up to a multiple of $2\pi \chi$, which means that the flow lines of $\wt {h}_{z,r} $ and $h$ in $B(z,7r/8)$ agree.
Moreover, in  $A(z,15r/16,r)$ we have that $\wt {h}_{z,r} =h-\Fh_{z,r}+\Fg_{z,r}$.
If $(z,r)$ is $M$-good for $h$, we have that the restriction of $\Fg_{z,r}-\Fh_{z,r}$  to $B(z,15r/16)$ has Dirichlet energy bounded by a constant depending only on $M$. As a consequence, the Radon-Nikodym derivative of the conditional law of $\wt {h}_{z,r}$ given $\CF_{z,r}$ w.r.t.\ the law of $h-\Fh_{z,r}+\Fg_{z,r}$ has    finite moments of all orders which depend only on $M$.

\textit{Boundary case.}  We now describe the setting in the case of boundary points. We will use the same nomenclature regarding good scales, it will be clear from the context if we are referring  to an interior or a boundary $M$-good scale. 

Assume that $h$ is a GFF in $\h$ with some bounded boundary values. 
For $r\in(0,1)$, we let $\CF_{0,r}$ be the $\sigma$-algebra generated by the values of $h$ outside of $B(0,r) \cap\h$. 
Again, we can write $h=\wt  h_{0, r} + \wt\Fh_{0,  r}$ where $ \wt  h_{0,  r}$ is a zero boundary GFF in $B(0,  r)\cap \h$ and $\wt\Fh_{0,  r}$ is a distribution which  is independent of $\wt h_{0,  r}$, harmonic in $B(0,  r)\cap \h$, and coincides with $h$ in $\ol{\h}\setminus( B(0,  r)\cap\h)$. 
For $M > 0$ we    say that $( 0,   r)$ is   $M$-good   for~$h$ if  
\begin{equation}
\label{eq:M_good_boundary}
\sup_{w \in B(0,15   r/ 16)\cap \h} |\wt \Fh_{0,  r}(w) | \leq M.
\end{equation}
Similarly as above, letting $E_{0,r}^{\partial, M}$   the event that $ (0,r)$ is $M$-good we have that  $E_{0,r}^{\partial, M}$ is $\CF_{0,r}$-measurable.  
Again, on $M$-good scales we can compare the law of $h$ in $B(0,r) \cap \h$ away from $\partial B(0,r)$ to a field on $B(0,r)\cap\h$, only that now their boundary values need to agree on $\R$. Then the Radon-Nikodym derivative which compares the laws of the two fields has finite moments of all orders.  Moreover, the following density result also  holds.

\begin{lemma}\label{lem:M_good_bd_density}Suppose that $h$ is a $\mathrm{GFF}$ in $\h$ with some bounded boundary conditions. For every $a\in(0,1)$ and $b>0$ there exist $M>0$ and $c>0$ so that the following holds. For every $k\in\N$, off an event of probability $ce^{-bk}$, the number of $j=1,\ldots,k$ such that $E^{\partial, M}_{0,2^{-j}}$ occurs is at least $(1-a)k$. 
\end{lemma}

\begin{proof} We see that we can reproduce the proof of~\cite[Lemma~4.4]{mq2020notsle}  in the boundary set up. Indeed 
the analogue of~\cite[Lemma~4.4]{mq2020notsle} holds when restricting to $w\in\partial B(0,15   r/ 16)\cap \h$; for all points  $w \in B(0,15   r/ 16)\cap \h$ it  then follows from the maximum principle. Hence, noting that whether $( 0,   r)$ is $M$-good is 
a measurable event w.r.t.\ the  $\sigma$-algebra  $\CF_{0,r}$ generated by $ h|_{\h \setminus B(0,   r)}$, we have that the statement of the lemma can be deduced in the same way as~\cite[Proposition~4.3]{mq2020notsle}. 
\end{proof}

We now deduce, in the case of boundary good scales, the $L^p$-bounds for the Radon-Nikodym derivatives between the law of $h$ (away from $\partial B(0,r) \cap \h$) and the law of a reference field, for every $p\in(1,\infty)$.

Let $ {\Fg}$ be a function which is harmonic in $\D\cap\h$ with some prescribed boundary values. 
Let $\psi^*_{  r}(w)=   r^{-1}w$ so that $\psi^*_{r}$ maps $B(0,r)\cap\h$ to $\D\cap\h$.
Let $\wt \phi $ be $C^\infty_0$ such that  $\wt \phi|_{B(0,7/8)} \equiv 1$, $\wt \phi|_{B(0,15/16)^c} \equiv 0$. Let  $\wt \phi_{  r}=\wt \phi \circ\psi^{*}_{  r}  $ and let $\wt\Fg_r$ be the harmonic function in $B(0,r) \cap\h$ whose boundary values agree with $\Fg \circ \psi^{*}_{  r}$ on $\partial B(0,r) \cap \h$ and with $h$ on $[-r,r]$. Consider the following field
\[\wt {h}_{ 0,  r} =h|_{B(0,  r)\cap\h}  + (\wt\Fg_{  r} - \wt \Fh_{0,  r})(1-\wt \phi_{  r}).\]
Note that $\wt {h}_{ 0,  r}$ coincides with $ h$ in  $ B(0,7  r/8)\cap\h$.
We also have that $\wt {h}_{ 0,  r}=h-\wt \Fh_{0,  r}+\wt\Fg_{  r}$ in  $A(0,15r/16,r) \cap\h$.

Note that the boundary values of $\wt\Fg_{r}-\wt\Fh_{0,r}$ are zero on $[-r,r]$. Therefore, if $(0,r)$ is $M$-good for $h$, we again have that the restriction of $\wt\Fg_{r}-\wt\Fh_{0,r}$  to $B(0,15r/16)\cap\h$ has Dirichlet energy bounded by a constant depending only on $M$. Thus, the Radon-Nikodym derivative of the conditional law of $\wt {h}_{0,r} $ given $\CF_{0,r}$ w.r.t.\ the law of $h-\wt \Fh_{0,r}+\wt \Fg_{r}$ has finite moments of all orders which depend only on $M$.

\subsection{Separation bounds}
\label{subsec:separation_bounds}

We will use the coupling of \clek{} with the GFF described in Sections~\ref{subsubsec:cle_gff}--\ref{subsubsec:cle4_gff} to gain control on the separation of the strands of the partially explored \clek{}. The coupling is not local in the sense that if $U$ is an open set then one cannot determine the loops which intersect~$U$ by observing only the field values in $U$ (see Figure~\ref{fi:spf_ordering}). We will instead analyze a random variable that is a (local) function of the GFF on an annulus and bounds from above the number of crossings that \clek{} loops make across the annulus.

\begin{definition}
\label{def:delta-sep-domain}
 Let $\delta > 0$ and $0 < r < R$. We say that (the marked points of) $\Gamma_\outside^{*,B(z,R),B(z,r)}$ are $\delta$-separated if $\abs{y_i-y_j} \ge \delta r$ for each distinct pair $y_i,y_j$ of marked points of $\Gamma_\outside^{*,B(z,R),B(z,r)}$.
\end{definition}

The main results of this subsection are Lemma~\ref{lem:separation_density} and its boundary version Lemma~\ref{lem:separation_density_bd}.

\begin{lemma}\label{lem:separation_density} 
Let $\kappa \in (8/3,8)$. Fix $u\in(0,1)$. For every $a\in(0,1)$, $b>0$, there exists $\delta>0$ so that the following holds.
Let $D\subseteq\C$ be a simply connected domain and let $\Gamma$ be a nested \clek{}  in $D$.  Let $j_0 \in\Z$ and $z\in D$ be such that $B(z,2^{-j_0})\subseteq D$.  For $j > j_0$, let $E^\separated_{z,j}$ be the event that the marked points of $\Gamma_\outside^{*,B(z,(1+u)2^{-j}),B(z,2^{-j})}$ are $\delta$-separated. For $k \in \N$ let $G_{z,j_0,k}$ be the event that the number of $j=j_0+1,\ldots,j_0+k$ such that $E^\separated_{z,j}$ occurs is at least $(1-a)k$. Then
\[ \p[(G_{z,j_0,k})^c] = O(e^{-bk}) \]
where the implicit constant does not depend on $z,j_0,k$.
\end{lemma}

The proof of Lemma~\ref{lem:separation_density} is based on the following lemmas which we state and prove separately for the cases $\kappa \in (4,8)$, $\kappa \in (8/3,4)$, and $\kappa = 4$.

\begin{lemma}
\label{le:separation_gff}
Let $\kappa' \in (4,8)$. Fix $u \in (0,1)$. There is a sequence of $\delta_n > 0$ and events $(E_n)$ measurable with respect to the values of the GFF modulo $2\pi\chi$ on $A(0,1-u,1+u)$ with $\lim_{n\to\infty} \p[E_n] \to 1$ and such that the following holds.

Let $\Gamma$ be a \clekp{} in a simply connected domain $D$ coupled with $h$ as in Section~\ref{subsubsec:cle_gff}. Suppose that $z \in \D$ and $r>0$ are such that $B(z,(1+u)r) \subseteq D$. Then for each $n$, on the event that $h(z+r\cdot) \in E_n$ the marked points of $\Gamma_\outside^{*,B(z,(1+u)r),B(z,r)}$ are $\delta_n$-separated.
\end{lemma}

\begin{proof}
 Recall that the crossings of the space-filling \slekp{} associated with the exploration of $\Gamma$ are determined by the values of $h$ on the annulus. However, we need to be a bit careful in case the first point of a loop traced by the space-filling \slekp{} is within the annulus because in that case the crossing of the loop would not correspond to a single segment of the space-filling \slekp{}.

 Fix $0 < u_1 < u/4$ and let $F^1$ be the event that the following holds. Let $\eta^L$ (resp.\ $\eta^R$) be an angle $-\pi/2$ (resp.\ $\pi/2$) flow line starting from a point on $\partial A(0,1-u/2,1+u/2)$, stopped upon exiting $A(0,1-u,1+u)$. Suppose that $\eta^L$ hits the left side of $\eta^R$ with an angle difference of $\pi$, and let $v$ be the first hitting point. Then $v \notin A(0,1,1+u_1)$. We have that $\p[F^1] \to 1$ as $u_1 \searrow 0$ since there are only finitely many disjoint flow line crossings of $A(0,1-u/2,1-u/4)$ resp.\ $A(0,1+u/4,1+u/2)$ (due to the continuity of the space-filling \slekp{}), and for each given pair of flow lines $\eta^L,\eta^R$ as above, the probability that their first intersection point $v$ is in $A(0,1,1+u_1)$ goes to $0$ (due to scale-invariance).

 We argue that if a \clekp{} loop $\CL$ crosses $A(0,1,1+u)$, its first point traced by the space-filling \slekp{} corresponds to a point $v$ in the definition of the event $F^1$ above, and hence does not lie in $A(0,1,1+u_1)$. Indeed, suppose that $V$ is the connected component containing $\CL$ at the time when the space-filling \slekp{} starts tracing $\CL$, and let $v \in \CL \cap \partial V$ be the corresponding point. Then $\partial V$ necessarily exits $B(0,1+u)$. Suppose that $\CL$ is traced clockwise. Then it crosses $A(0,1,1+u)$, and its right boundary hits $\partial V$ first at $v$. We therefore find $\eta^L$ (resp.\ $\eta^R$) as in the event $F^1$ above given by a segment of $\partial V$ (resp.\ of the right boundary of a segment of $\CL$). The case when $\CL$ is traced counterclockwise is analogous.

 Given $u_1$, let $\CD_n = 2^{-n}\Z^2 \cap A(0,1-u_1/2,1+u_1/2)$. For each $w \in \CD_n$, let $\eta^L_w$ (resp.\ $\eta^R_w$) be the flow line of $h$ with angle $\pi/2$ (resp.\ $-\pi/2$) starting from $w$ and stopped upon exiting $A(0,1-u_1/2,1+u_1/2)$. For each pair $w,w' \in \CD_n$ we consider the connected components bounded by the four flow lines $\eta^L_w$, $\eta^L_{w'}$, $\eta^R_w$, $\eta^R_{w'}$ in case they merge before exiting $A(0,1-u_1/2,1+u_1/2)$; we refer to them as pockets. Let $F_n$ be the event that every point on $\partial B(0,1)$ is contained in a pocket. We have $\p[F_n] \to 1$ as $n \to \infty$ due to the continuity of the space-filling \slekp{}.
 
 Suppose that we are on the event that $h(z+r\cdot) \in F^1 \cap F_n$. Then for each strand of $\Gamma_\outside^{*,B(z,(1+u)r),B(z,r)}$ that reaches $\partial B(z,r)$, the space-filling \slekp{} makes a crossing through the annulus $A(z,r,(1+u_1)r)$. Let $x$ be a marked point of $\Gamma_\outside^{*,B(z,(1+u)r),B(z,r)}$. Then on the event that $h(z+r\cdot) \in F^1 \cap F_n$ the point $x$ is contained in (the image upon scaling and translating of) a pocket and is the first or last point on $\partial B(z,r)$ (depending on whether the exploration tree is crossing in or out of the annulus) visited by the counterflow line within the pocket.
 
 The pockets and the counterflow lines within the pockets are measurable functions of the values of the GFF modulo $2\pi\chi$ within the annulus $A(0,1-u_1/2,1+u_1/2)$. Since (non-space-filling) \slekp{} does not hit fixed points, the minimal distance between all the first (resp.\ last) points visited by the counterflow lines within the pockets is a.s.\ positive. Therefore the probability that the minimal distance is at least $\delta_n$ becomes arbitrarily close to $1$ when $\delta_n$ is chosen small enough. That is, we can set $E_n$ be the event that $F^1 \cap F_n$ occurs and the minimal distance between each such pair of points is at least $\delta_n$.
\end{proof}

\begin{lemma}
\label{le:separation_gff_simple}
Let $\kappa \in (8/3,4)$. Fix $u \in (0,1)$. There is a sequence of $\delta_n > 0$ and events $(E_n)$ measurable with respect to the values of the GFF modulo $2\pi\chi$ on $A(0,1-u,1+u)$ with $\lim_{n\to\infty} \p[E_n] \to 1$ and such that the following holds.

Let $\Gamma$ be a \clek{} in a simply connected domain $D$ coupled with $h$ as in Section~\ref{subsubsec:simple_cle_gff}. Suppose that $z \in \D$ and $r>0$ are such that $B(z,(1+u)r) \subseteq D$. Then for each $n$, on the event that $h(z+r\cdot) \in E_n$ the marked points of $\Gamma_\outside^{*,B(z,(1+u)r),B(z,r)}$ are $\delta_n$-separated.
\end{lemma}

\begin{proof}
 Similarly as before, we first deal with the event that the first point of a loop traced by the exploration tree is in $A(0,1-u_1,1+u_1)$ for some small $u_1$.

 Fix $0 < u_1 < u/8$ and let $F^1$ be the event that the following holds. Let $\eta^R$ be an angle $-\pi/2$ flow line starting from a point on $\partial A(0,1-u/2,1+u/2)$, stopped upon exiting $A(0,1-u,1+u)$. Let $\eta$ be an angle $5\pi/2$ flow line of $h|_{D\setminus\eta^R}$ starting from a point on the left side of $\eta^R$ on $\partial A(0,1-u/4,1+u/4)$, stopped upon exiting $A(0,1-u/2,1+u/2)$. (That is, the angle difference between $\eta$ and $\eta^R$ is $3\pi$, and $\eta$ can be reflected in the opposite direction of $\eta^R$ when hitting the latter.) Let $v$ be the last hitting point. Then $v \notin A(0,1-u_1,1+u_1)$. The same holds if $\eta$ is an angle $-7\pi/2$ flow line starting from the right side of $\eta^R$. We have that $\p[F^1] \to 1$ as $u_1 \searrow 0$ since there are only finitely many disjoint flow line crossings of $A(0,1-u/4,1-u/8)$ resp.\ $A(0,1+u/8,1+u/4)$ (due to the continuity of the space-filling \slekp{}), and for each given pair of flow lines $\eta^R,\eta$ as above, the probability that their last intersection point $v$ is in $A(0,1-u_1,1+u_1)$ goes to $0$ (due to scale-invariance).

 We argue that if a \clek{} loop $\CL$ crosses $A(0,1,1+u)$, its first point traced by the exploration tree corresponds to a point $v$ in the definition of the event $F^1$ above, and hence does not lie in $A(0,1-u_1,1+u_1)$. Suppose that $V'$ is the complementary connected component of the true loop of the \bclekp{0} containing $\CL$, let $v \in \CL \cap \partial V'$ be the first point of $\CL$ traced by the \bclek{-\kappa/2} exploration in $V'$, and let $v' \in \partial V'$ be the first point traced by the \bclekp{0} exploration. Then $V'$ necessarily exits $B(0,1+u)$. Suppose that $\CL$ is traced counterclockwise and $\partial V'$ is traced clockwise. Then the counterflow line tracing $\partial V'$ has angle $0$ (modulo $2\pi$), and its right boundary has angle $-\pi/2$. Assume first that both arcs of $\partial V'$ between $v$ and $v'$ exit $A(0,1-u/2,1+u/2)$. Then we find $\eta^R$ (resp.\ $\eta$) as in the event $F^1$ above given by a segment of $\partial V'$ (resp.\ of the \bclek{-\kappa/2} exploration). Now consider the case when the counterclockwise arc of $\partial V'$ from $v$ to $v'$ is contained in $A(0,1-u/2,1+u/2)$. Then $\eta^R$ finishes tracing $\partial V'$ at $v'$, but we can extend it until it exits $A(0,1-u,1+u)$. The flow line $\eta$ will merge with the \bclek{-\kappa/2} exploration due to the interaction rules \cite{ms2016ig1}. Finally, consider the case when the counterclockwise arc of $\partial V'$ from $v'$ to $v$ is contained in $A(0,1-u/2,1+u/2)$. Then the portion of $\partial V'$ making up $\eta^R$ starts only at $v'$. But since $\CL$ exits $B(0,1+u)$, we can find an angle $-\pi/2$ flow line $\eta^R$ from a point on $\partial A(0,1-u/2,1+u/2)$ that merges with that portion of $\partial V'$ due to the interaction rules \cite{ms2016ig1}.

 The case when $\CL$ is traced clockwise and $\partial V'$ is traced counterclockwise is the same except that we interchange ``left'' and ``right''. The counterflow line has angle $-\pi$ (modulo $2\pi$), so that its left boundary also has angle $-\pi/2$ (modulo $2\pi$).

 Now, given $u_1$, consider all excursions in $A(0,1-u_1,1+u_1)$ of angle $\pi/2$ flow lines of $h$ that reach $\partial B(0,1)$. For each such excursion with an endpoint on $\partial B(0,1+u_1)$, consider its first (resp.\ last, depending on the orientation) point on $\partial B(0,1)$. Note that the set of these points on $\partial B(0,1)$ is finite since there are only finitely many disjoint flow line crossings of $A(0,1-u_1,1)$ resp.\ $A(0,1,1+u_1)$ (due to the continuity of the space-filling \slekp{}). Let $F_n$ be the event that the distance between each two such points on $\partial B(0,1)$ is at least $1/n$. Then $\p[F_n] \to 1$ as $n \to \infty$.

 Let $E_n = F^1 \cap F_n$. It follows that on the event that $h(z+r\cdot) \in E_n$, each strand of $\Gamma_\outside^{*,B(z,(1+u)r),B(z,r)}$ that reaches $\partial B(z,r)$ contains a flow line crossing of $A(z,r,(1+u_1)r)$. Consequently, the marked points have distance at least $r/n$ from each other.
\end{proof}

\begin{lemma}
\label{le:separation_gff_cle4}
Fix $u \in (0,1)$. There is a sequence of $\delta_n > 0$ and events $(E_n)$ measurable with respect to the values of the GFF modulo $2\lambda$ on $A(0,1,1+u)$ with $\lim_{n\to\infty} \p[E_n] \to 1$ and such that the following holds.

Let $\Gamma$ be a \cle{4} in a simply connected domain $D$ coupled with $h$ as in Section~\ref{subsubsec:cle4_gff}. Suppose that $z \in \D$ and $r>0$ are such that $B(z,(1+u)r) \subseteq D$. Then for each $n$, on the event that $h(z+r\cdot) \in E_n$ the marked points of $\Gamma_\outside^{*,B(z,(1+u)r),B(z,r)}$ are $\delta_n$-separated.
\end{lemma}

\begin{proof}
By the construction given in \cite[Sections~5.1 and~6.1]{kms2022removability}, we can detect the collection of level line crossings of an annulus. More precisely, for a GFF $\wt{h}$ in $2\D$, they construct a local set $K$ that contains all level line crossings of $A(0,1,1+u)$ with heights in $(2\Z+1)\lambda$.

Suppose now that $h$ is a GFF on a domain $D$ containing $2\D$. Let $\wt{h}$ be a field on $2\D$ given by $\wt{h}|_{B(0,3/2)} = h|_{B(0,3/2)}$, and $\wt{h}|_{2\D \setminus B(0,3/2)}$ is sampled according to the conditional law with respect to a zero-boundary GFF on $2\D$, conditionally independent of $h|_{D \setminus B(0,3/2)}$.

Let $E_n$ be the event that the endpoints of the aforementioned level line crossings of $\wt{h}$ on $\partial B(0,1)$ are $1/n$-separated. Although the set $K \cap A(0,1,1+u)$ is not local, the proof of \cite[Lemma~6.1]{kms2022removability} together with a change of measure argument shows that the level line crossings of $h$ and of $\wt{h}$ agree.

This concludes the proof since the unfinished strands of $\Gamma_\outside^{*,B(z,(1+u)r),B(z,r)}$ are given exactly by the level line crossing of $A(z,r,(1+u)r)$.
\end{proof}

\begin{proof}[Proof of Lemma~\ref{lem:separation_density}]
 Let $\Gamma$ be coupled with a GFF $h$ as in Sections~\ref{subsubsec:cle_gff}--\ref{subsubsec:cle4_gff}. Let $E_n$ be the events from Lemma~\ref{le:separation_gff}--\ref{le:separation_gff_cle4}. It suffices to show for large enough $n$ the analogous stronger statement where we require that the number of $j=j_0+1,\ldots,j_0+k$ such that $h(z+2^{-j}\cdot) \in E_n$ is at least $(1-a)k$.
 
 Let $E_{z,r}^M$ be the event for $h$ as in~\eqref{eq:M_good_scale}. By Lemma~\ref{lem:M_good_density} we can let $M \in \N$ be large enough so that with probability $1-O(e^{-b k})$ at least $(1-a/2)k$ scales $j=j_0+1,\ldots,j_0+k$ are $M$-good. To finish the proof, we argue that the probability that $E_{z,2^{-j+1}}^M \cap \{h(z+2^{-j}\cdot) \notin E_n\}$ occurs at least $(a/2)k$ scales is at most $O(e^{-bk})$. Let $\CF_j = \CF_{z,2^{-j}}$ be the $\sigma$-algebra generated by the values of $h$ outside $B(z,2^{-j})$. Then the event $E_{z,2^{-j+1}}^M \cap \{h(z+2^{-j}\cdot) \notin E_n\}$ is $\CF_{j+1}$-measurable. Recall from Section~\ref{subsec:gff_good_scales} that on the event $E_{z,2^{-j+1}}^M$, the conditional law of $h|_{B(z,(1+u)2^{-j})}$ modulo $2\pi\chi$ (resp.\ $2\lambda$ in the case $\kappa=4$) is comparable to the law $\wt{\p}$ of the corresponding restriction of a zero-boundary GFF on $B(z,2^{-j+1})$, and the Radon-Nikodym derivative is $L^p$-bounded for every $p \in (1,\infty)$ with bound depending only on $M$. Hence, for each $j$,
 \[
  \p\left[ h(z+2^{-j}\cdot) \notin E_n \mmiddle| \CF_{j-1} \right] \one_{E_{z,2^{-j+1}}^M} \lesssim \wt{\p}[ E_n^c ]^{1/2} .
 \]
 By choosing $n$ large, the latter probability can be made arbitrarily small. The claim follows.
\end{proof}

We now give the boundary version of Lemma~\ref{lem:separation_density}. We define $\delta$-separation between the marked points of $\Gamma_\outside^{*,B(0,R)\cap\h,B(0,r)\cap\h}$ in the same way as in Definition~\ref{def:delta-sep-domain}.

\begin{lemma}\label{lem:separation_density_bd} 
Let $\kappa \in (8/3,8)$. Fix $u\in(0,1)$. For every $a\in(0,1)$, $b>0$, there exists $\delta>0$ so that the following holds.
Let $\Gamma$ be a nested \clek{}  in $\h$. For $j \in \N$, let $E^\separated_{j}$ be the event that the marked points of $\Gamma_\outside^{*,B(0,(1+u)2^{-j})\cap\h,B(0,2^{-j})\cap\h}$ are $\delta$-separated. For $k \in \N$ let $G_{k}$ be the event that the number of $j=1,\ldots,k$ such that $E^\separated_{j}$ occurs is at least $(1-a)k$. Then
\[ \p[(G_{k})^c] = O(e^{-bk}) . \]
\end{lemma}

\begin{proof}
The proof is   the same as   the proof of Lemma~\ref{lem:separation_density},  the only difference is that we use Lemma~\ref{lem:M_good_bd_density} in place of 
Lemma~\ref{lem:M_good_density}. 
\end{proof}

\subsection{Independence across scales}
\label{subsec:ind_across_scales}

In this section, we prove the independence across scales results for \clek{}. We will again formulate a version in the interior and a version on the boundary. The proofs are based on the separation of strands result from Section~\ref{subsec:separation_bounds} and the continuity of multichordal \clek{} in Theorem~\ref{th:continuity_mcle}.

We describe the setup of Proposition~\ref{pr:cle_ind_across_scales}. Let $\kappa \in (8/3,8)$. Suppose that $D\subseteq\C$ is a simply connected domain and let $\Gamma$ be a nested \clek{} in $D$. Let $u\in(0,1/8)$ be fixed. Let $j_0 \in\Z$ and $z\in D$ be such that $B(z,2^{-j_0})\subseteq D$.

For $j>j_0$, let $U_j=B(z,2^{-j})$, $V_j= B(z,(1+u)2^{-j})$, and let $\CG_{j}$ be the filtration generated by the partially explored \clek{} $\Gamma_\outside^{*,V_j,U_j}$. Let $\varphi_j$ be the unique conformal map from $V_j^{*,U_j}$ to $\D$ with $\varphi_j(z)=0$, $\varphi_j'(z)>0$. Let $(\D;\ul{s}^j;\beta_j)$ be the exterior link pattern decorated marked domain which arises under $\varphi_j$. Let $\Gamma_{\ul{s}^j,\beta_j}$ be the image of $\Gamma$ under $\varphi_j$. Conditionally on $\CG_j$, we have that $\Gamma_{\ul{s}^j,\beta_j}$ is a multichordal \clek{} in $(\D;\ul{s}^j;\beta_j)$.

Let $W_0 = A(0,5/8,1)$. Let $\CE_{W_0}$ be the class of events $E$ that are open sets with respect to the topology \eqref{eq:topology_loop_ensemble} and are measurable with respect to $\Gamma_\outside^{*,B(0,5/8),B(0,1/2)}$. For $E \in \CE_{W_0}$, we let $E_{z,j}$ be the event that $E$ occurs for $\Gamma_{\ul{s}^j,\beta_j}$.

\begin{proposition}\label{pr:cle_ind_across_scales}
 Suppose that we have a non-decreasing sequence of events $E^n \in \CE_{W_0}$ such that
 \begin{equation}\label{eq:annulus_event_likely}
  \lim_{n\to\infty} \mcclelaw{\D;\ul{x};\beta}[E^n] = 1
 \end{equation}
 for every $(\D;\ul{x};\beta) \in \eldomain{2N}$ and $N\in\N_0$.
 
 Given any $a\in(0,1)$ and $b>0$, there exists $n \in \N$ such that the following holds. Let $D\subseteq\C$ be a simply connected domain and let $\Gamma$ be a nested \clek{} in $D$. Let $j_0 \in\Z$ and $z\in D$ be such that $B(z,2^{-j_0})\subseteq D$. For $j > j_0$, let $E^n_{z,j}$ be as defined in the paragraph above. Let $\wt G_{z,j_0,k}$ be the event that the number of $j = j_0+1,\ldots,j_0+k$ such that $E^n_{z,j}$ occurs is at least $(1-a)k$. Then
 \[
  \p[(\wt G_{z,j_0,k})^c] = O(e^{-b k}) 
 \]
 where the implicit constant does not depend on $z,j_0,k$.
\end{proposition}

\begin{proof}
 For $z \in D$, $j \in \Z$, and $\delta > 0$, let $E^\separated_{z,j}$ be the event that the strands of $\Gamma_\outside^{*,V_j,U_j}$ are $\delta$-separated. By Lemma~\ref{lem:separation_density} we can let $\delta>0$ be small enough so that with probability $1-O(e^{-b k})$ the events $E^\separated_{z,j}$ occur for at least $(1-a/2)k$ values of $j = j_0+1,\ldots, j_0 + k$. To conclude, we argue that for $n$ large enough the probability that $(E^n_{z,j})^c \cap E^\separated_{z,j}$ occurs for more than $(a/2)k$ scales $j = j_0+1,\ldots, j_0 + k$ is at most $O(e^{-bk})$.
 
 Given $\delta$, by \eqref{eq:annulus_event_likely} and Corollary~\ref{co:unif_prob_marked_points} we have $\mcclelaw{\D;\ul{x};\beta}[E^n] \to 1$ uniformly in all $\delta$-separated configurations $\ul{x}$. Letting $\CG_j$ denote the $\sigma$-algebra generated by $\Gamma_\outside^{*,V_j,U_j}$, we then have
 \[ \p[ E^n_{z,j} \mid \CG_j ] \,\one_{E^\separated_{z,j}} \ge p \one_{E^\separated_{z,j}} \]
 where $p$ can be made arbitrarily close to $1$ by choosing $n$ large. Since for each $j$ the event $E^n_{z,j}$ is measurable with respect to $\CG_{j+1}$, the result follows.
\end{proof}

We now describe the setup of Proposition~\ref{pr:cle_ind_across_scales_bd} which is a boundary version of Proposition~\ref{pr:cle_ind_across_scales}. Let $\Gamma$ be a nested \clek{} in $\h$. Let $u\in(0,1/8)$ be fixed and for $j\in\N$, let $U_j=B(0,2^{-j})\cap\h$, $V_j=B(0,(1+u)2^{-j})\cap\h$, and let $\CG_{j}$ be the filtration generated by the partially explored \clek{} $\Gamma_\outside^{*,V_j,U_j}$. Let $\varphi_j$ be the conformal map from $V_j^{*, U_j}$ to $\D\cap\h$ such that $\varphi_j(0)=0$ and the rightmost  (resp.\ leftmost) intersection point of the exploration with $\R_-$ (resp.\ $\R_+$) is mapped to $-1$ (resp.\ $1$).

Let $(\D\cap\h;\ul{s}^j;\beta_j)$ be the exterior link pattern decorated marked domain arising from  $\varphi_j$. Let  $\Gamma_{\ul{s}^j,\beta_j}$ be the image of $\Gamma$ under $\varphi_j$. Conditionally on $\CG_{j}$, we have that $\Gamma_{\ul{s}^j,\beta_j}$ is a multichordal \clek{}  in  $(\D\cap\h; \ul{s}^j;\beta_j)$.

Let $\wt{W}_0 = A(0,5/8,1) \cap \h$. Let $\CE_{\wt{W}_0}$ be the class of events $E$ that are open sets with respect to the topology \eqref{eq:topology_loop_ensemble} and are measurable with respect to $\Gamma_\outside^{*,B(0,5/8)\cap\h,B(0,1/2)\cap\h}$. For $E \in \CE_{\wt{W}_0}$, we let $E_{j}$ be the event that $E$ occurs for $\Gamma_{\ul{s}^j,\beta_j}$.

\begin{proposition}\label{pr:cle_ind_across_scales_bd}
 Suppose that we have a non-decreasing sequence of events $E^n \in \CE_{\wt{W}_0}$ such that
 \begin{equation}\label{eq:annulus_event_likely_bd}
  \lim_{n\to\infty} \mcclelaw{\D\cap\h;\ul{s};\beta}[E^n] = 1
 \end{equation}
 for every $(\D\cap\h;\ul{s};\beta) \in \eldomain{2N}$ and $N\in\N_0$ where the marked points $\ul{s}$ lie on $\partial\D \cap \h$.
 
 Given any $a\in(0,1)$ and $b>0$, there exists $n \in \N$ such that the following holds. Let $\Gamma$ be a nested \clek{} in $\h$. For $j \in \N$, let $E^n_{j}$ be as defined in the paragraph above. Let $\wt G_{k}$ be the event that the number of $j = 1,\ldots,k$ such that $E^n_{j}$ occurs is at least $(1-a)k$. Then
 \[
  \p[(\wt G_{k})^c] = O(e^{-b k}) .
 \]
\end{proposition}

\begin{proof}
 The proof is the same as for Proposition~\ref{pr:cle_ind_across_scales} where we use Lemma~\ref{lem:separation_density_bd} in place of Lemma~\ref{lem:separation_density}.
\end{proof}

We record another variant of the independence across scales result where the events are exact scalings of each other (instead of the image under the conformal map $\varphi_j$). Proving this variant requires the stronger continuity result from Section~\ref{sec:tv_convergence} below. We state them here for the reader's convenience. We note that the proofs in the remainder of this paper do not use the Propositions~\ref{pr:cle_ind_across_scales_scaling_int} and~\ref{pr:cle_ind_across_scales_scaling_bd}.

Let $W_0 = A(0,5/8,1)$, and for $\varepsilon > 0$ let $W_0^\varepsilon$ denote the $\varepsilon$-neighborhood of $W_0$. Let
\begin{equation}\label{eq:cle_local_sigma_alg}
 \CF_{W_0} = \bigcap_{\varepsilon>0} \sigma\left( \Gamma_\inside^{*,W_0^\varepsilon,W_0} \right)
\end{equation}
where $\Gamma_\inside^{*,W_0^\varepsilon,W_0}$ is defined in the same way as in Definition~\ref{def:partially-explored-cle}. For $E \in \CF_{W_0}$ and $z\in D$, $j\in\Z$ we let $E_{z,j}$ denote the event that $E$ occurs for $2^j(\Gamma+z)$ (so that $E_{z,j}$ is measurable with respect to the \clek{} configuration within $A(z,(5/8)2^{-j},2^{-j})$).

\begin{proposition}\label{pr:cle_ind_across_scales_scaling_int}
 Suppose that we have a sequence of events $E^n \in \CF_{W_0}$ such that
 \[
  \lim_{n\to\infty} \mcclelaw{D;\ul{x};\beta}[E^n] = 1
 \]
 for every $(D;\ul{x};\beta) \in \eldomain{2N}$ and $N\in\N_0$ with $B(0,9/8) \subseteq D$.
 
 Given any $a\in(0,1)$ and $b>0$, there exists $n \in \N$ such that the following holds. Let $D\subseteq\C$ be a simply connected domain and let $\Gamma$ be a nested \clek{} in $D$. Let $j_0 \in\Z$ and $z\in D$ be such that $B(z,2^{-j_0})\subseteq D$. For $j > j_0$, let $E^n_{z,j}$ be as defined in the paragraph above. Let $\wt G_{z,j_0,k}$ be the event that the number of $j = j_0+1,\ldots,j_0+k$ such that $E^n_{z,j}$ occurs is at least $(1-a)k$. Then
 \[
  \p[(\wt G_{z,j_0,k})^c] = O(e^{-b k}) 
 \]
 where the implicit constant does not depend on $z,j_0,k$.
\end{proposition}

\begin{proof}
 The proof is exactly the same as for Proposition~\ref{pr:cle_ind_across_scales} except that instead of Theorem~\ref{th:continuity_mcle} we use Proposition~\ref{prop:mccle_tv_convergence_int} together with the local compactness of the Carathéodory topology. (Note that here we do not need to assume the sequence $(E^n)$ to be non-decreasing because the continuity in total variation implies the continuity of probabilities uniformly for all events.)
\end{proof}

Last but not least, we state the boundary version of Proposition~\ref{pr:cle_ind_across_scales_scaling_int}. Let $\wt{W}_0 = A(0,5/8,1) \cap \h$, and let $\CF_{\wt{W}_0}$ be defined as in \eqref{eq:cle_local_sigma_alg}. For $E \in \CF_{\wt{W}_0}$ and $j\in\N$ we let $E_{j}$ denote the event that $E$ occurs for $2^j\Gamma$ (so that $E_{j}$ is measurable with respect to the \clek{} configuration within $A(0,(5/8)2^{-j},2^{-j}) \cap \h$).

\begin{proposition}\label{pr:cle_ind_across_scales_scaling_bd}
 Suppose that we have a sequence of events $E^n \in \CF_{\wt{W}_0}$ such that
 \[
  \lim_{n\to\infty} \mcclelaw{D;\ul{x};\beta}[E^n] = 1
 \]
 for every $(D;\ul{x};\beta) \in \eldomain{2N}$ and $N\in\N_0$ with $B(0,9/8)\cap\h \subseteq D \subseteq \h$ where the marked points $\ul{x}$ lie on $\partial D \cap \h$.
 
 Given any $a\in(0,1)$ and $b>0$, there exists $n \in \N$ such that the following holds. Let $\Gamma$ be a nested \clek{} in $\h$. For $j \in \N$, let $E^n_{j}$ be as defined in the paragraph above. Let $\wt G_{k}$ be the event that the number of $j = 1,\ldots,k$ such that $E^n_{j}$ occurs is at least $(1-a)k$. Then
 \[
  \p[(\wt G_{k})^c] = O(e^{-b k}) .
 \]
\end{proposition}

The proof of Proposition~\ref{pr:cle_ind_across_scales_scaling_bd} is the same as for the previous propositions where we now use Proposition~\ref{prop:mccle_tv_convergence_bd} as an input.

\section{Resampling target pivotals}
\label{sec:resampling_tools}

In this section we develop a framework to resample the strands of \clek{}'s with $\kappa \in (4,8)$ across small regions with the aim of changing how the loops are linked together. The main result of this section is Proposition~\ref{prop:resampling} (and its boundary version Proposition~\ref{prop:resampling_H}). These  roughly state that for  any \emph{suitable resampling procedure}, on a large probability event, and at an arbitrarily large density of scales, the geometry of the \clek{} is sufficiently good so that we have uniformly positive probability of the resampling being successful.

We then describe in Section~\ref{se:resampling_results} how we can apply the resampling procedure to create loops with certain target shapes. We will describe events that give us a positive probability of creating a loop that disconnects a region, and events that give us a positive probability of breaking all loops that cross a given annulus. Our results can be seen as a local independence of \clek{} which says that the law of the local configuration remains approximately the same when we change it at a different location.

These results will be useful in other works. In \cite{amy2025tightness}, we use this to compare the law of the \clek{} locally to the law of a region that is disconnected by a single loop.

\emph{Throughout the entire Section~\ref{sec:resampling_tools}, we assume that $\kappa \in (4,8)$.} Resampling procedures can be defined for each $\kappa \in (8/3,8)$ but our focus will be on the pivotal points which exist only for $\kappa \in (4,8)$.

\subsection{Setup}
\label{subsec:resampling_setup}

In this subsection we will describe a procedure for resampling connections in a multichordal \clek{}. This will be the  \emph{suitable resampling procedure} we referred to above that will be used in  Proposition~\ref{prop:resampling} and Proposition~\ref{prop:resampling_H}.

We start by clarifying what we mean by \emph{resampling procedure}: Let $\Gamma$ be a multichordal \clek{} in $(D;\ul{x};\beta)$, and let $W \subseteq D$. We construct another multichordal \clek{} in $(D;\ul{x};\beta)$ coupled with $\Gamma$ by repeatedly applying the following procedure. Select regions $(U,V) \in \domainpair{D}$ with $V \subseteq W$ randomly in a way that is independent of the CLE configuration within $W$, and let $\wt{\Gamma}$ be such that $\wt{\Gamma}_\outside^{*,V,U} = \Gamma_\outside^{*,V,U}$ and the remainder of $\wt{\Gamma}$ is sampled from its conditional law given $\wt{\Gamma}_\outside^{*,V,U}$ (independently of the remainder of $\Gamma$). We say that $\Gamma^\resampled$ is a \emph{resampling of $\Gamma$ within $W$} if it is obtained by applying such a procedure a (random) number of times.

Whenever two strands of loops intersect, their intersection points are called pivotal points because changing the connections between the strands hitting the point affects the macroscopic behavior of the loops. We now formally describe what we mean by ``changing the links at a given collection of pivotal points''.

\begin{definition}[Target pivotals]
\label{def:resampling_target}   
Let $\Gamma$ be a multichordal \clek{} in $(D;\ul{x};\beta) \in \eldomain{2N}$.  Let $W\subseteq D$ be open and $r > 0$. Suppose that we are given the following random variables.
\begin{enumerate}[(i)]
\item\label{it:n_r_def} $M^*<\infty$.
\item  A set $\{(\ell_i^1,\ell^2_i)\}_{i=1,\ldots,M^*}$ of pairs of arcs in $\Gamma$ with $\ell_i^1,\ell^2_i \subseteq W$ and such that none of them overlap with each other.
\item A set of $M^*$ points $\ul{y}$ such that
\begin{itemize}
 \item $y_i\in  \ell_i^1\cap \ell^2_i$ and $B(y_i,4r)\subset W$ for $i=1,\ldots,M^*$,
 \item the endpoints of $\ell^1_i, \ell^2_i$ lie outside $B(y_i,4r)$,
 \item $B(y_i,4r)\cap B(y_j,4r)=\varnothing$ and $(\ell_i^1 \cup \ell_i^2) \cap B(y_j,4r) = \varnothing$ for all $i\neq j$.
\end{itemize}
\end{enumerate}
We denote these target pivotals by $\mathfrak F_{W,r} = \{(\ell^1_i,\ell^2_i,y_i)\}_{i=1,\ldots,M^*}$.
\end{definition}

We want to resample $\Gamma$ in small regions of $W$ so that the linking pattern between each pair $(\ell_i^1,\ell_i^2)$ is switched. This is the subject of the following definition.

\begin{definition}
\label{def:successful_resampling}
Let $\Gamma$ be a multichordal \clek{} in $(D;\ul{x};\beta)$. Let $W \subseteq D$, $r>0$, and let $\mathfrak F_{W,r} = \{(\ell^1_i,\ell^2_i,y_i)\}$ be target pivotals. Let $\Gamma^\resampled$ be a resampling of $\Gamma$ within $W$.  We say that  $\Gamma^\resampled$ is $\mathfrak F_{W,r}$-successful if the following hold.
\begin{enumerate}[(i)]
 \item The loops and strands of $\Gamma$ and $\Gamma^\resampled$ are identical when restricted to the set $D \setminus \bigcup_{i=1}^{M^*} B(y_i,4r)$.
 \item Let $\tau_i^j$ (resp.\ $\ol{\tau}_i^j$) be the first (resp.\ last) time that $\ell_i^j$ is in $\ol{B(y_i,4 r)}$, for $i=1,\ldots,M^*$, $j=1,2$. Then in $\Gamma^\resampled$ the connections between the four strands $\ell_i^1|_{[0,\tau_i^1]}$, $\ell_i^1|_{[\ol{\tau}_i^1,1]}$, $\ell_i^2|_{[0,\tau_i^2]}$, $\ell_i^2|_{[\ol{\tau}_i^2,1]}$ (assuming that they are parameterized on $[0,1]$) are flipped, i.e.\ the strand $\ell^{\resampled,1}_i$ (resp.\ $\ell^{\resampled,2}_i$) in $\Gamma^\resampled$ starting with $\ell_i^1|_{[0,\tau_i^1]}$ (resp.\ finishing with $\ell_i^1|_{[\ol{\tau}_i^1,1]}$) is linked to $\ell_i^2|_{[0,\tau_i^2]}$ or $\ell_i^2|_{[\ol{\tau}_i^2,1]}$. Moreover, it holds again that $\ell^{\resampled,1}_i \cap \ell^{\resampled,2}_i \cap B(y_i,r) \neq \varnothing$.
\end{enumerate}
\end{definition}

\subsection{Resampling in the interior} \label{subsec:interior_resampling}

We now state and prove the main result of this section, in the case of resampling around points in the interior of the domain. Proposition~\ref{prop:resampling} gives that with large probability, we have uniformly positive probability for the resampling being successful at a large density of scales around interior points, for a suitable resampling procedure. 

\begin{proposition}
\label{prop:resampling}
Suppose that $\Gamma$ is a nested \clek{} in a simply connected domain $D$. Given any $M_0 \in \N$, $r > 0$, $a\in(0,1)$, and $b>0$, there exists $p \in (0,1)$ such that the following holds.

Let $j_0 \in\Z$ and $z\in D$ be such that $B(z,2^{-j_0})\subseteq D$. There exists a resampling $\Gamma^\resampled_{z,j}$ of $\Gamma$ within $W_{z,j} \defeq A(z,(5/8)2^{-j},(7/8)2^{-j})$ for each $j>j_0$ with the following property. Let $G^p_{z,j}$ be the event for $\Gamma$ that for any choice of target pivotals $\mathfrak F_{W_{z,j},r2^{-j}} = \{(\ell^1_i,\ell^2_i,y_i)\}_{i=1,\ldots,M^*}$ with $M^* \le M_0$, the conditional probability given $\Gamma$ that $\Gamma^\resampled_{z,j}$ is $\mathfrak F_{W_{z,j},r2^{-j}}$-successful is at least $p$. For $k\in\N$ let $\wt G_{z,j_0,k}$ be the event that the number of $j = j_0+1,\ldots,j_0+k$ such that $G^p_{z,j}$ occurs is at least $(1-a)k$. Then
\[
 \p[(\wt G_{z,j_0,k})^c] = O(e^{-b k}) 
\]
where the implicit constant does not depend on $z,j_0,k$.
\end{proposition}

For the proof of Proposition~\ref{prop:resampling} we will partly use the same setup and notation as in Section~\ref{subsec:ind_across_scales}. We will deduce Proposition~\ref{prop:resampling} from the independence across scales Proposition~\ref{pr:cle_ind_across_scales}.

We describe a concrete resampling procedure. Consider a multichordal \clek{} process $\Gamma$ in $(D;\ul{x};\beta) \in \eldomain{2N}$.  Assume $r>0$ is fixed and $W\subseteq D$ is open. Let
\begin{equation}\label{eq:grid_annulus}
 \CD_{W,r} = \{ w \in r\Z^2 \mid B(w,2r) \subseteq W \} .
\end{equation}
We pick $\wt{M} \in \N_0$ randomly according to a geometric distribution, independently of $\Gamma$ (the exact choice of the distribution of $\wt{M}$ is not relevant as long as its support is $\N_0$). Then we repeat the following procedure $\wt{M}$ times:
\begin{enumerate}[1.]
\item Select a point $w \in \CD_{W,r}$ uniformly at random and consider the partially explored \clek{} process $\Gamma_\outside^{*,B(w,2r),B(w,r)}$. By Theorem~\ref{thm:cle_partially_explored}, the conditional law of the remainder $\Gamma_\inside^{*,B(w,2r),B(w,r)}$ given $\Gamma_\outside^{*,B(w,2r),B(w,r)}$ is given by a multichordal \clek{} in the unexplored domain.
\item Resample $\Gamma$ according to its conditional law given $\Gamma_\outside^{*,B(w,2r),B(w,r)}$.
\end{enumerate}
Let $\Gamma^\resampled_{W,r}$ denote the multichordal \clek{} that arises after performing the resampling step described above on $\Gamma$.

Recall Corollary~\ref{co:hookup_positive_continuous} that in a $\delta$-separated marked domain the probability that the resampling procedure creates any given interior link pattern is uniformly bounded from below where the bound depends on the separation $\delta$. Therefore, if for each $w \in \CD_{W,r}$ the marked points of $\Gamma_\outside^{*,B(w,2r),B(w,r)}$ are $\delta$-separated, then the conditional probability that the resampling is $\mathfrak{F}_{W,r}$-successful is positive and uniformly bounded from below.

\begin{lemma}\label{le:resampling_success_positive}
 Fix $W \subseteq D$ and $r>0$. For any $M_0 \in \N$, $\delta > 0$, there exists $p>0$ such that the following is true. Suppose that $\Gamma$ is a multichordal \clek{} in $(D;\ul{x};\beta) \in \eldomain{2N}$. Apply the resampling procedure described just above to $\Gamma$. Let $E_\delta$ be the event that for every $w \in \CD_{W,r}$ we have that $\Gamma_\outside^{*,B(w,2r),B(w,r)}$ is $\delta$-separated. Then for any choice of target pivotals $\mathfrak F_{W,r} = \{(\ell^1_i,\ell^2_i,y_i)\}_{i=1,\ldots,M^*}$ with $M^* \le M_0$ we have
 \[ \p[ G^\resampled_{W,r} \mid \Gamma ] \,\one_{E_\delta} \ge p\,\one_{E_\delta} \]
 where $G^\resampled_{W,r}$ is the event that the resampling is $\mathfrak F_{W,r}$-successful.
\end{lemma}

\begin{proof}
 Suppose we are given any target pivotals $\mathfrak F_{W,r} = \{(\ell^1_i,\ell^2_i,y_i)\}_{i=1,\ldots,M^*}$ with $M^* \le M_0$. If we perform the resampling operation, then with uniformly positive probability (depending on $M_0,r$) we sample $\wt{M} = M^*$ and the points $w_i \in \CD_{W,r/2}$ such that $w_i \in B(y_i,r/2)$ for every $i=1,\ldots,M^*$. In particular, the strands $\ell^1_i$, $\ell^2_i$ cross the annulus $A(w_i,r,2r)$. On the event $E_{\delta}$ the strands of $\Gamma_\outside^{*,B(w_i,2r),B(w_i,r)}$ are $\delta$-separated. Therefore, by Corollary~\ref{co:hookup_positive_continuous}, the probability that $\ell^1_i,\ell^2_i$ switch their linking pattern after performing the resampling is bounded from below uniformly in $\delta$. Moreover, by continuity (Theorem~\ref{th:continuity_mcle}) we see that with positive probability bounded from below the new strands $\ell^{\resampled,1}_i,\ell^{\resampled,2}_i$ intersect again inside $B(w_i,r/2)$. (Although we have not phrased the topology \eqref{eq:topology_loop_ensemble} to keep track of the intersection points between the chords, the proof of Theorem~\ref{th:continuity_mcle} shows that continuity also holds when we require the intersection sets to converge.) Repeating this argument $M^*$ times for each target pivotal, we conclude.
\end{proof}

Recall the setup of Section~\ref{subsec:ind_across_scales}. Suppose that $D\subseteq\C$ is a simply connected domain and let $\Gamma$ be a nested  \clek{} in $D$. Let $u\in(0,1/8)$ be fixed. Let $j_0 \in\Z$ and $z\in D$ be such that $B(z,2^{-j_0})\subseteq D$.

For $j>j_0$, let $U_j=B(z,2^{-j})$, $V_j= B(z,(1+u)2^{-j})$, and let $\CG_{j}$ be the filtration generated by the partially explored \clek{} $\Gamma_\outside^{*,V_j,U_j}$. Let $\varphi_j$ be the unique conformal map from $V_j^{*,U_j}$ to $\D$ with $\varphi_j(z)=0$, $\varphi_j'(z)>0$. Let $(\D;\ul{s}^j;\beta_j)$ be the exterior link pattern decorated marked domain which arises under $\varphi_j$. Let $\Gamma_{\ul{s}^j,\beta_j}$ be the image of $\Gamma$ under $\varphi_j$. Conditionally on $\CG_j$, we have that $\Gamma_{\ul{s}^j,\beta_j}$ is a multichordal \clek{} in $(\D;\ul{s}^j;\beta_j)$.

Fix $5/8<r_1<r_2<7/8$, and let $W = A(0,r_1,r_2)$. We apply the resampling procedure described above to $\Gamma_{\ul{s}^j,\beta_j}$ with $W$ and $r/100$. Let $(\Gamma_{\ul{s}^j,\beta_j})^\resampled_{W,r/100}$ be the multichordal \clek{} arising from the resampling, and let $\Gamma^\resampled_{z,j}$ be the \clek{} obtained from mapping back under $\varphi_j^{-1}$. Note that any target pivotals $\mathfrak F_{W_{z,j},r2^{-j}}$ for $\Gamma$ induce some target pivotals $\mathfrak F_{W,r/100}$ for $\Gamma_{\ul{s}^j,\beta_j}$ via the map $\varphi_j$. If $(\Gamma_{\ul{s}^j,\beta_j})^\resampled_{W,r/100}$ is $\mathfrak{F}_{W,r/100}$-successful, then $\Gamma^\resampled_{z,j}$ is $\mathfrak F_{W_{z,j},r2^{-j}}$-successful, according to Definition~\ref{def:successful_resampling}.

Recalling Lemma~\ref{le:resampling_success_positive}, we now see that Proposition~\ref{prop:resampling} is a consequence of the independence across scales Proposition~\ref{pr:cle_ind_across_scales}.

\begin{lemma}
\label{le:event_everywhere_separated}
Fix $\delta_0 > 0$, $r>0$, and $W = A(0,r_1,r_2)$.  Suppose that $(\D;\ul{x};\beta) \in \eldomain{2N}$ is $\delta_0$-separated as in Definition~\ref{def:delta-sep-domain} and let $\Gamma$ be a multichordal \clek{} in $(\D;\ul{x};\beta)$.  Let $\CD_{W,r}$ be as in \eqref{eq:grid_annulus}.  Let $E_\delta$ be the event that for every $w \in \CD_{W,r}$ we have that $\Gamma_\outside^{*,B(w,2r),B(w,r)}$ is $\delta$-separated.  Then
\[ 
\mcclelaw{\D;\ul{x};\beta}[E_\delta] \to 1 
\quad \text{as } \delta \to 0
\]
at a rate that depends only on $\delta_0$ and $r$.
\end{lemma}

\begin{proof}
 By Lemma~\ref{lem:finitely-marked-points-marginal}, the minimum distance between distinct marked points of $\Gamma_\outside^{*,B(w,2r),B(w,r)}$ is a.s.\ positive. Therefore we have $\mcclelaw{\D;\ul{x};\beta}[E_\delta] \to 1$ as $\delta \to 0$ for any fixed $(\D;\ul{x};\beta) \in \eldomain{2N}$. The uniform convergence for all $\delta_0$-separated configurations $\ul{x}$ then follows from Corollary~\ref{co:unif_prob_marked_points}.
\end{proof}

\begin{proof}[Proof of Proposition~\ref{prop:resampling}]
This follows from Proposition~\ref{pr:cle_ind_across_scales}, Lemma~\ref{le:event_everywhere_separated}, and Lemma~\ref{le:resampling_success_positive}.
\end{proof}

\subsection{Resampling at the boundary} \label{subsec:boundary_resampling}

We now provide a version of the resampling result, analogous to statement of Proposition~\ref{prop:resampling},  in the case the resampling occurs is performed in annuli around a point $z$ on the boundary of the domain $D$. 

\begin{proposition}
\label{prop:resampling_H}
Suppose that $\Gamma$ is a nested \clek{} in $\h$. Given any $M_0 \in \N$, $r > 0$, $a\in(0,1)$, and $b>0$, there exists $p \in (0,1)$ such that the following holds.

There exists a resampling $\Gamma^\resampled_{j}$ of $\Gamma$ within $W_{j} \defeq A(0,(5/8)2^{-j},(7/8)2^{-j}) \cap \h$ for each $j\in\N$ with the following property. Let $G^p_{j}$ be the event for $\Gamma$ that for any choice of target pivotals $\mathfrak F_{W_{j},r2^{-j}} = \{(\ell^1_i,\ell^2_i,y_i)\}_{i=1,\ldots,M^*}$ with $M^* \le M_0$, the conditional probability given $\Gamma$ that $\Gamma^\resampled_{j}$ is $\mathfrak F_{W_{j},r2^{-j}}$-successful is at least $p$. For $k\in\N$ let $\wt G_k$ be the event that the number of $j = 1,\ldots,k$ such that $G^p_{j}$ occurs is at least $(1-a)k$. Then
\[
 \p[(\wt G_k)^c] = O(e^{-b k}) .
\]
\end{proposition}

The proof of Proposition~\ref{prop:resampling_H} follows along the same lines as the proof of Proposition~\ref{prop:resampling} where we now consider the boundary setup in Section~\ref{subsec:ind_across_scales}.

Let $\Gamma$ be a nested \clek{} in $\h$. Let $u\in(0,1/8)$ be fixed and for $j\in\N$, let $U_j=B(0,2^{-j})\cap\h$, $V_j=B(0,(1+u)2^{-j})\cap\h$, and let $\CG_{j}$ be the filtration generated by the partially explored \clek{} $\Gamma_\outside^{*,V_j,U_j}$. Let $\varphi_j$ be the conformal map from $V_j^{*, U_j}$ to $\D\cap\h$ such that $\varphi_j(0)=0$ and the rightmost  (resp.\ leftmost) intersection point of the exploration with $\R_-$ (resp.\ $\R_+$) is mapped to $-1$ (resp.\ $1$).

Let $(\D\cap\h;\ul{s}^j;\beta_j)$ be the exterior link pattern decorated marked domain arising from  $\varphi_j$. Let  $\Gamma_{\ul{s}^j,\beta_j}$ be the image of $\Gamma$ under $\varphi_j$. Conditionally on $\CG_{j}$, we have that $\Gamma_{\ul{s}^j,\beta_j}$ is a multichordal \clek{}  in  $(\D\cap\h; \ul{s}^j;\beta_j)$.

Fix $5/8<r_1<r_2<7/8$, and let $W =A(0,r_1,r_2)\cap\h$. We apply the resampling procedure described above Lemma~\ref{le:resampling_success_positive} to $\Gamma_{\ul{s}^j,\beta_j}$ with $W$ and $r/100$. Let $(\Gamma_{\ul{s}^j,\beta_j})^\resampled_{W,r/100}$ be the multichordal \clek{} arising from the resampling, and let $\Gamma^\resampled_{j}$ be the \clek{} obtained from mapping back under $\varphi_j^{-1}$. Note that any target pivotals $\mathfrak F_{W_{j},r2^{-j}}$ for $\Gamma$ induce some target pivotals $\mathfrak F_{W,r/100}$ for $\Gamma_{\ul{s}^j,\beta_j}$ via the map $\varphi_j$. If $(\Gamma_{\ul{s}^j,\beta_j})^\resampled_{W,r/100}$ is $\mathfrak{F}_{W,r/100}$-successful, then $\Gamma^\resampled_{j}$ is $\mathfrak F_{W_{j},r2^{-j}}$-successful, according to Definition~\ref{def:successful_resampling}.

With this setup, the proof of Proposition~\ref{prop:resampling_H} is exactly the same as for Proposition~\ref{prop:resampling} where instead of Proposition~\ref{pr:cle_ind_across_scales} we use the boundary version Proposition~\ref{pr:cle_ind_across_scales_bd}.

\subsection{Linking and breaking loops}
\label{se:resampling_results}

In this subsection we provide some resampling procedures with the target of linking or breaking loops in a specific way. In Lemma~\ref{lem:disconnect_interior} and~\ref{lem:disconnect_boundary} we will consider the success probability of linking loops together into one loop that disconnects a region from the outside. In Lemma~\ref{lem:break_loops} and~\ref{lem:break_loops_bd} we will consider the success probability of breaking all loops that cross a given annulus. The result is that on an extremely likely event, at a large ratio of scales around a point $z$ the probability that the resampling in the annulus around $z$ is successful is bounded from below. We will consider variants of these results when $z$ is in the interior and when $z$ is on the boundary of the domain.

We assume throughout that $\kappa \in (4,8)$. Recall the definition of a resampling procedure given at the beginning of Section~\ref{subsec:resampling_setup}. In particular, if $\Gamma^\resampled$ is a resampling of $\Gamma$ within a set $W \subseteq D$, then the configuration of loops and strands in $D \setminus W$ is unchanged.

\begin{lemma}
\label{lem:disconnect_interior}
Suppose that $\Gamma$ is a nested \clek{} in a simply connected domain $D$, and let $\Upsilon_\Gamma$ be its gasket. For each $a\in(0,1)$, $b>0$ there exists $p > 0$ such that the following holds. 

Let $z\in D$ and $j_0\in\Z$ be such that $B(z,2^{-j_0})\subseteq D$. There exists a resampling $\Gamma^\resampled_{z,j}$ of $\Gamma$ within $A(z,2^{-j},2^{-j+1})$ for each $j>j_0$ with the following property. Let $G^\resampled_{z,j}$ be the event that $\Gamma^\resampled_{z,j}$ has a loop $\wt\CL$ that
\begin{itemize}
 \item disconnects every point in $\Upsilon_{\Gamma^\resampled_{z,j}} \cap B(z,2^{-j})$ from $\partial B(z,2^{-j+1})$,
 \item the gasket and the collection of loops of $\Gamma$ and $\Gamma^\resampled_{z,j}$ remain the same in each connected component of $\C \setminus \wt\CL$ that intersects $B(z,2^{-j})$.
\end{itemize}
For $k\in\N$, let $\wt G_{z,j_0,k}$ be the event that the number of $j=j_0+1,\ldots,j_0+k$ so that
\[ \p[ G^\resampled_{z,j} \giv \Gamma ] \geq p\]
is at least $(1-a)k$.  Then
\[ \p[(\wt G_{z,j_0,k})^c] = O(e^{-b k}) \]
where the implicit constant does not depend on $z,j_0,k$.
\end{lemma}

\begin{lemma}
\label{lem:disconnect_boundary}
Suppose that $\Gamma$ is a nested \clek{} in $\h$. For each $a\in(0,1)$, $b>0$ there exists $p > 0$ such that the following holds. 

There exists a resampling $\Gamma^\resampled_{j}$ of $\Gamma$ within $A(0,2^{-j},2^{-j+1}) \cap \h$ for each $j\in\N$ with the following property. Let $G^\resampled_{j}$ be the event that $\Gamma^\resampled_{j}$ has a loop $\wt\CL$ that
\begin{itemize}
 \item disconnects $[-2^{-j},2^{-j}]$ from $\partial B(0,2^{-j+1})$,
 \item the collection of loops of $\Gamma$ and $\Gamma^\resampled_{j}$ remain the same in each connected component of $\h \setminus \wt\CL$ adjacent to $[-2^{-j},2^{-j}]$.
\end{itemize}
For $k\in\N$, let $\wt G_{k}$ be the event that the number of $j=1,\ldots,k$ so that
\[ \p[ G^\resampled_{j} \giv \Gamma ] \geq p\]
is at least $(1-a)k$.  Then
\[ \p[(\wt G_{k})^c] = O(e^{-b k}) . \]
\end{lemma}

\begin{lemma}
\label{lem:break_loops}
Suppose that $\Gamma$ is a nested \clek{} in a simply connected domain $D$. For each $a\in(0,1)$, $b>0$ there exists $p > 0$ such that the following holds. 

Let $z\in D$ and $j_0\in\Z$ be such that $B(z,2^{-j_0})\subseteq D$. There exists a resampling $\Gamma^\resampled_{z,j}$ of $\Gamma$ within $A(z,2^{-j},2^{-j+1})$ for each $j>j_0$ with the following property. Let $G^\resampled_{z,j}$ be the event that
\begin{itemize}
 \item no loop in $\Gamma^\resampled_{z,j}$ crosses the annulus $A(z,2^{-j},2^{-j+1})$,
 \item denoting $\CC$ the loops in $\Gamma$ that cross $A(z,2^{-j},2^{-j+1})$, the collection of loops of $\Gamma$ and $\Gamma^\resampled_{z,j}$ remain the same in each connected component of $\C \setminus \bigcup\CC$ that is not surrounded by a loop in $\CC$ and intersects $B(z,2^{-j})$, and the gasket of $\Gamma$ in these components is contained in the gasket of $\Gamma^\resampled_{z,j}$.
\end{itemize}
For $k\in\N$, let $\wt G_{z,j_0,k}$ be the event that the number of $j=j_0+1,\ldots,j_0+k$ so that
\[ \p[ G^\resampled_{z,j} \giv \Gamma ] \geq p\]
is at least $(1-a)k$.  Then
\[ \p[(\wt G_{z,j_0,k})^c] = O(e^{-b k}) \]
where the implicit constant does not depend on $z,j_0,k$.
\end{lemma}

\begin{lemma}\label{lem:break_loops_bd}
Suppose that $\Gamma$ is a nested \clek{} in $\h$. For each $a\in(0,1)$, $b>0$ there exists $p > 0$ such that the following holds. 

There exists a resampling $\Gamma^\resampled_{j}$ of $\Gamma$ within $A(0,2^{-j},2^{-j+1}) \cap \h$ for each $j\in\N$ with the following property. Let $G^\resampled_{j}$ be the event that
\begin{itemize}
 \item no loop in $\Gamma^\resampled_{j}$ crosses the annulus $A(0,2^{-j},2^{-j+1}) \cap \h$,
 \item denoting $\CC$ the loops in $\Gamma$ that cross $A(0,2^{-j},2^{-j+1}) \cap \h$, the collection of loops of $\Gamma$ and $\Gamma^\resampled_{j}$ remain the same in each connected component of $\C \setminus \bigcup\CC$ that is not surrounded by a loop in $\CC$ and is adjacent to $[-2^{-j},2^{-j}]$, and the gasket of $\Gamma$ in these components is contained in the gasket of $\Gamma^\resampled_{j}$.
\end{itemize}
For $k\in\N$, let $\wt G_{k}$ be the event that the number of $j=1,\ldots,k$ so that
\[ \p[ G^\resampled_{j} \giv \Gamma ] \geq p\]
is at least $(1-a)k$.  Then
\[ \p[(\wt G_{k})^c] = O(e^{-b k}) . \]
\end{lemma}

The main input in showing these results is making use of the resampling procedure developed in Sections~\ref{subsec:interior_resampling} and~\ref{subsec:boundary_resampling}, along with the corresponding estimates for the success probabilities from Proposition~\ref{prop:resampling} and Proposition~\ref{prop:resampling_H}. To complete the proofs, we need to describe suitable target pivotals that occur with high probability in a large fraction of scales and such that, if successfully resampled, create the desired behavior.

\begin{figure}[ht]
\centering
\includegraphics[width=0.45\textwidth]{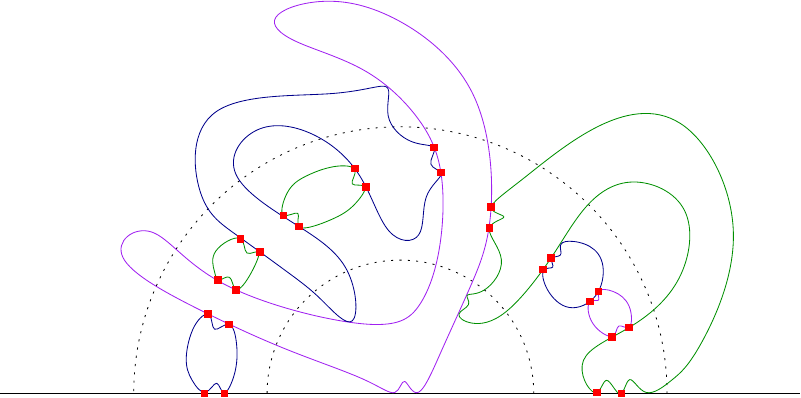}\hspace{0.05\textwidth}\includegraphics[width=0.45\textwidth]{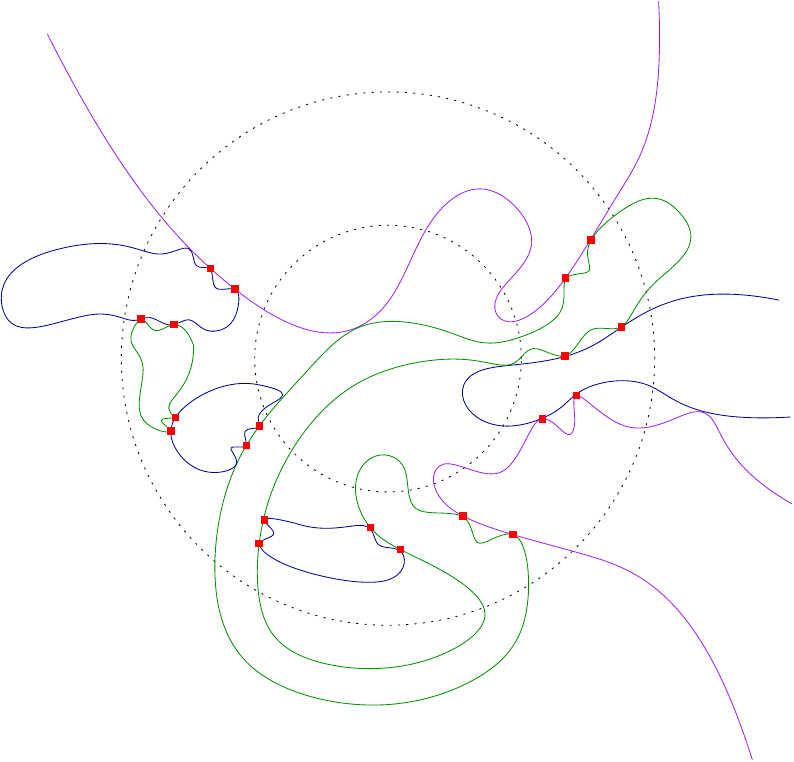}
\caption{Illustration of the events from Lemma~\ref{lem:finite_chain_bd} (left) and Lemma~\ref{lem:finite_chain_int} (right). Here only the outer boundaries of the loops are shown.}\label{fi:cle_link}
\end{figure}

We begin by describing the event that allows linking loops as in the Lemmas~\ref{lem:disconnect_interior},~\ref{lem:disconnect_boundary}. See the Figure~\ref{fi:cle_link} for an illustration. We prove in Lemma~\ref{lem:finite_chain_bd} (resp.~\ref{lem:finite_chain_int}) that there a.s.\ is a finite chain of loops intersecting in an annulus and disconnecting the inner and outer boundary of the annulus. By linking the loops together, we create a single loop $\wt\CL$ as in the statement of Lemma~\ref{lem:disconnect_boundary} (resp.~\ref{lem:disconnect_interior}).

We define what we mean by a \emph{chain of intersecting loops} and their \emph{inner and outer strands}.
\begin{definition}
\label{def:disconnecting-chain}
 Let $\Gamma$ be a \clek{} in $D$. A chain of loops that intersect in $U \subseteq D$ consists of the following:
 \begin{enumerate}[(i)]
  \item A finite collection $(\ell^L_i,\ell^R_i)_{i=1,\ldots,m}$ of segments of loops of $\Gamma$ so that $\ell^L_i \subseteq U$, $\ell^R_i \subseteq U$ for each $i=1,\ldots,m$ and none of them overlap with each other.
  \item We assume that if we parameterize $\ell^L_i$ (resp.\ $\ell^R_i$) in the counterclockwise (resp.\ clockwise) direction of the loop it belongs to, then the right side of $\ell^L_i$ intersects the left side of $\ell^R_i$ for each $i=1,\ldots,m$. We let $x_i,y_i$ be distinct points on their intersection and assume that $y_i$ comes after $x_i$ in the parameterization above.
  \item\label{it:chain_bdry_pts} Optionally additional points $x_0,y_0,x_{m+1},y_{m+1} \in \partial D$.
  \item For each $i=1,\ldots,m$, we let $\ell^I_i$ (resp.\ $\ell^O_i$) be segment of the loop containing $\ell^L_i$ starting at $x_i$ (resp.\ $y_i$), going clockwise (resp.\ counterclockwise), and terminating at the next $x_{i'}$ (resp.\ $y_{i'}$) where $i' \in \{0,\ldots,m+1\}\setminus\{i\}$.
  \item We call the \emph{inner strand} (resp.\ \emph{outer strand}) of the chain the union $\ell^I = \bigcup_{i=1}^m \ell^I_i$ (resp.\ $\ell^O = \bigcup_{i=1}^m \ell^O_i$), and assume that $\ell^I \cap \ell^O = \varnothing$. We remark that there can be several connected components of $\ell^I$ or $\ell^O$; see the Figure~\ref{fi:cle_link}.
 \end{enumerate}
 If $\Gamma$ is a multichordal \clek{}, then we define a chain of loops in the same way where we also allow the strands to be segments of chords, and we say that the chords $\eta_k$ of $\Gamma$ emanating from odd (resp.\ even) $k$ are parameterized in the counterclockwise (resp.\ clockwise) direction.
\end{definition}

The additional boundary points in \eqref{it:chain_bdry_pts} are needed in the case we want a chain of loops that disconnect a boundary interval (as in the left image of Figure~\ref{fi:cle_link}). We do not need them when the inner strand is contained in the interior of the domain (as in the right image of Figure~\ref{fi:cle_link}).

The following lemma is the key property that we use for the proofs of this subsection.

\begin{lemma}\label{lem:finite_chain_bd}
Let $\Gamma$ be a \clek{} in $\h$. Let $U \subseteq \h$ be a Jordan domain such that $\partial U \setminus \R$ consists of exactly two connected components $\partial_\rmin U$, $\partial_\rmout U$. We assume that $\partial_\rmout U$ separates $\partial_\rmin U$ from $\infty$. Then almost surely there is a chain of loops of $\Gamma$ intersecting in $U$ such that, with the notation from Definition~\ref{def:disconnecting-chain},
\begin{itemize}
 \item $\ell^I$ connects the two segments of $\partial U \cap \R$,
 \item $\ell^I$ does not intersect $\partial_\rmout U$, and $\ell^O$ does not intersect $\partial_\rmin U$.
\end{itemize}
\end{lemma}

\begin{proof}
 It follows from Lemma~\ref{le:cle_finite_chain} that there is almost surely a chain of loops intersecting in $U$ that connects the two segments of $\partial U \cap \R$ (i.e.\ without the additional requirement that $\ell^I$ does not intersect $\partial_\rmout U$). One can see this e.g.\ by fixing a point $z \in U$ and considering the loop surrounding $z$. Then applying Lemma~\ref{le:cle_finite_chain} twice for each interval of $\partial U \cap \R$ we see that it can be connected to both through a chain of loops. (Note that whenever two loop segments intersect, they a.s.\ intersect at infinitely many points, so that we can always find distinct points $x_i,y_i$ on their intersection.)
 
 Fix another Jordan domain $U_1 \subseteq U$ as in the lemma statement with $\partial_\rmin U_1, \partial_\rmout U_1 \subseteq U$. Let $\wt{\ell}^I$ be the inner strand of a chain of loops intersecting in $U_1$ and connecting the two segments of $\partial U_1 \cap \R$. It might be that $\wt{\ell}^I$ intersects $\partial_\rmout U$ or $\wt{\ell}^O$ intersects $\partial_\rmin U$. We will add more loops to the chain in order to fulfill the second requirement.
 
 Consider the loops $\wt{\ell}_i$ of the chain such that $\wt{\ell}^I_i \cap \partial_\rmout U \neq \varnothing$. We say that such loops $\wt{\ell}_i$ \emph{contribute to the failing}. Note that each loop that contributes to the failing makes at least $4$ crossings between $\partial U_1 \cap \h$ and $\partial U \cap \h$. By the local finiteness of CLE, the number of such crossings is a.s.\ finite.
 
 Let $\wt{U}_1$ be a connected component of $U_1 \setminus \wt{\ell}^I$ that is adjacent to a loop that contributes to the failing. We can repeat the first step of the proof to find a chain of loops that intersect in $\wt{U}_1$ and connect two boundary segments of $\wt{U}_1$ incident to the failing loop. (In the case the two crossings of $\wt{\ell}^I_i$ intersect, we could alternatively just add the intersecting strands to the chain.) If the inner strand of this new chain intersects $\partial_\rmout U$ again, then there is another loop that contributes to the failing, hence makes at least $4$ crossings between $\partial U_1 \cap \h$ and $\partial U \cap \h$. Iterating this procedure for every loop that contributes to the failing, we see that the iteration must successfully terminate after a finite number of steps, since the number of such crossings is a.s.\ finite.
 
 We repeat the same procedure when $\wt{\ell}^O$ intersects $\partial_\rmin U$. By viewing this as the image under the conformal map $z \mapsto -1/z$, this is exactly the same as the previous case. By combining all the chains together, we obtain a new chain as desired.
\end{proof}

\begin{lemma}\label{lem:finite_chain_bd_mcl}
Let $\Gamma$ be a multichordal \clek{} in $(\D\cap\h; \ul{s}; \beta)$ where all the marked points are on $\partial\D \cap \h$. Let $\wt{r}<1$ and let $U \subseteq B(0,\wt{r}) \cap \h$ be a Jordan domain such that $\partial U \setminus \R$ consists of exactly two connected components $\partial_\rmin U$, $\partial_\rmout U$. We assume that $\partial_\rmout U$ separates $\partial_\rmin U$ from $\partial\D \cap \h$. Then almost surely there is a chain of loops of $\Gamma$ intersecting in $U$ such that, with the notation from Definition~\ref{def:disconnecting-chain},
\begin{itemize}
 \item $\ell^I$ connects the two segments of $\partial U \cap \R$,
 \item $\ell^I$ does not intersect $\partial_\rmout U$.
\end{itemize}
\end{lemma}

\begin{proof}
Let $\wt{\Gamma}$ be a \clek{} in $\D \cap \h$, and let $\wt \p$ be the law of $\wt{\Gamma}$. By the same argument as in the proof of Lemma~\ref{lem:confclass_nonempty} and Proposition~\ref{prop:multichordal-well-def}, there exists an exploration of $\wt{\Gamma}$ that stops before any of the exploration paths hit $\partial B(0,1-\varepsilon) \cap \h$ and realizes a random marked domain of the type $(D_N;\ul{u}_N;\beta)$. By Lemma~\ref{lem:cle_strands_in_regions} and Proposition~\ref{prop:continuity_in_separated_points}, we can follow it by a Markovian exploration that stops before any of the strands hit $\partial B(0,1-2\varepsilon) \cap \h$ and realizes a marked domain conformally equivalent to $(\D\cap\h; \ul{s}; \beta)$ with positive probability.

Applying Lemma~\ref{lem:finite_chain_bd} to (the conformal images to $\h$ of) $\wt{\Gamma}$ and some $U_1$ with $\partial_\rmin U_1, \partial_\rmout U_1 \subseteq U$, we see that if we let $\varphi$ be the conformal map to $(\D\cap\h; \ul{s}; \beta)$, then the image of the chain for $\wt{\Gamma}$ intersecting in $U_1$ fulfills the requirement for $\Gamma$ in $U$ for $\varepsilon>0$ small enough.
\end{proof}

\begin{proof}[Proof of Lemma~\ref{lem:disconnect_boundary}]
We use the notation introduced in Subsection~\ref{subsec:boundary_resampling}. Let $5/8<r_1<r_2<7/8$, and let $W =A(0,r_1,r_2)\cap\h$. Fix $u>0$ small. Recall the definitions of $\varphi_j$ and $\Gamma_{\ul{s}^j,\beta_j}$.

We will provide target pivotals given by the intersecting chain of loops in Lemma~\ref{lem:finite_chain_bd_mcl}, applied with $U=W$. More precisely, for each $j\in\N$, let $E_j$ be the event that the event from Lemma~\ref{lem:finite_chain_bd_mcl} occurs for the multichordal \clek{} process $\Gamma_{\ul{s}^j,\beta_j}$. Consider the target pivotals given by the image under $\varphi_j^{-1}$ of $\{(\ell^L_i,\ell^R_i,y_i)\}$ from Definition~\ref{def:disconnecting-chain}. For fixed $r>0$ and $M_0 \in \N$, let $F_j$ be the event that $E_j$ occurs and
\begin{itemize}
 \item $m \le M_0$ where $m$ is the number of strands in the intersecting chain,
 \item the points $y_i$ are distance at least $8r$ away from each other, from the endpoints of $\ell^L_i,\ell^R_i$, and from the inner strand $\ell^I$.
\end{itemize}
By Lemma~\ref{lem:finite_chain_bd_mcl}, for each marked domain of the type $(\D\cap\h;\ul{s};\beta)$, the probability that $F_j$ occurs can be made arbitrarily close to $1$ by choosing $M_0$ large and $r$ small. Therefore, by the independence of scales Proposition~\ref{pr:cle_ind_across_scales_bd} we can find $M_0,r$ such that with probability $1-O(e^{-bk})$ this occurs at at least $(1-a/2)k$ of the scales $j=1,\ldots,k$.

Choosing $u>0$ sufficiently small, we also see that $\varphi^{-1}_j(W)\subseteq A(0,2^{-j-1},2^{-j}) \cap \h$ for every $j$. Given $M_0,r$ as above, we let $p>0$ be as in Proposition~\ref{prop:resampling_H}. Therefore, if both $F_j$ and the event $G^p_j$ from Proposition~\ref{prop:resampling_H} occur, then $\p[ G^\resampled_j \mid \Gamma ] \ge p$ as desired.
\end{proof}

We now turn to the proof of Lemma~\ref{lem:disconnect_interior} which is the interior version of Lemma~\ref{lem:disconnect_boundary}. The proof is entirely analogous, so we will be brief. We state the interior analogue of Lemma~\ref{lem:finite_chain_bd_mcl}. See the right image of Figure~\ref{fi:cle_link} for an illustration.

\begin{lemma}\label{lem:finite_chain_int}
 Let $\Gamma$ be a nested multichordal \clek{} in $(\D;\ul x;\beta)\in\eldomain{2N}$. Let $0<r_1<r_2<1$, the following event occurs almost surely. If the gasket $\Upsilon_\Gamma$ of $\Gamma$ intersects $B(0,r_1)$, then there is a chain of loops of $\Gamma$ intersecting in $A(0,r_1,r_2)$ such that, with the notation from Definition~\ref{def:disconnecting-chain},
\begin{itemize}
 \item every point in $\Upsilon_\Gamma \cap B(0,r_1)$ is separated from $\partial B(0,r_2)$ by the inner strand $\ell^I$,\footnote{We recall that $\ell^I$ might be formed by several connected components when some loop splits $B(0,r_1)$ into several components as in Figure~\ref{fi:cle_link}.}
 \item $\ell^I$ does not intersect $\partial B(0,r_2)$.
\end{itemize}
\end{lemma}

\begin{proof}
 We first argue that the event occurs a.s.\ for a \clek{}. We can explore the outermost loops of $\Gamma$ until we find one that intersects $B(0,r_2)$. If no outermost loop of $\Gamma$ intersects $B(0,r_2)$, then $\Upsilon_\Gamma \cap B(0,r_2) = \varnothing$, and the event in the lemma statement is satisfied. Suppose we find a loop $\CL_0$ that intersects $B(0,r_2)$. The remainder is a \clek{} in each connected component of the unexplored region. We can apply Lemma~\ref{lem:finite_chain_bd} for each connected component $U$ of $A(0,r_1,r_2) \setminus \CL_0$ that is adjacent to both $\partial B(0,r_1)$ and $\partial B(0,r_2)$. By the local finiteness of CLE, there are only finitely many of such components. We obtain a chain of loops as desired by combining $\CL_0$ and the chains we find for these components.
 
 The statement for multichordal \clek{} follows from the statement for \clek{} by arguing exactly the same as in the proof of Lemma~\ref{lem:finite_chain_bd_mcl}.
\end{proof}

\begin{proof}[Proof of Lemma~\ref{lem:disconnect_interior}]
 The proof is completely analogous to the proof of Lemma~\ref{lem:disconnect_boundary} where we use Lemma~\ref{lem:finite_chain_int}, Proposition~\ref{pr:cle_ind_across_scales}, and~Proposition~\ref{prop:resampling} instead of Lemma~\ref{lem:finite_chain_bd_mcl}, Proposition~\ref{pr:cle_ind_across_scales_bd}, and~Proposition~\ref{prop:resampling_H}. The only difference is that we do not aim to resample at all the points $\varphi_j^{-1}(y_i)$ but only at so many until all the loops in the chain are linked to one. This is in order to not change the nesting level of the points in the gasket.
\end{proof}

\begin{figure}[ht]
\centering
\includegraphics[width=0.45\textwidth]{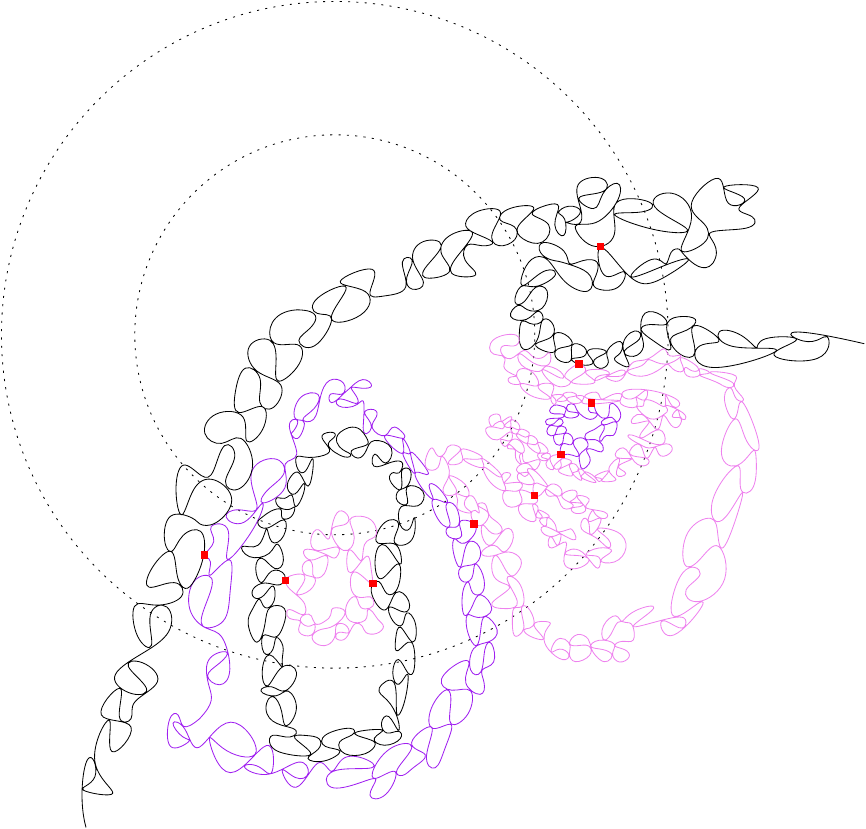}
\caption{Illustration of the event from Lemma~\ref{le:break_loops_int}.}\label{fi:cle_break}
\end{figure}

We now explain how we can break loops as requested in the Lemmas~\ref{lem:break_loops} and~\ref{lem:break_loops_bd}. For each loop $\CL$ crossing the annulus $A(z,r_1,r_2)$ we will find a chain of loops in the next level of nesting inside $\CL$, and link the strand of $\CL$ to these loops instead. See Figure~\ref{fi:cle_break} for an illustration.

\begin{lemma}\label{le:break_loops_int}
 Let $\Gamma$ be a nested multichordal \clek{} in $(\D;\ul x;\beta)\in\eldomain{2N}$. Let $0<r_1<r_2<1$, the following event occurs almost surely.
 \begin{itemize}
  \item There are finitely many crossings of $A(0,r_1,r_2)$ of loops or strands of $\Gamma$.
  \item Let $\CL$ be any loop or chord that crosses $A(0,r_1,r_2)$, let $U$ be any connected component of $A(0,r_1,r_2) \setminus \CL$ that is adjacent to both $\partial B(0,r_1)$ and $\partial B(0,r_2)$, and let $\gamma_1,\gamma_2$ be the crossings of $\CL$ adjacent to $U$. There is a chain of loops that intersect in $U$, connect $\gamma_1$ to $\gamma_2$, and such that, with the notation from Definition~\ref{def:disconnecting-chain}, its inner strand $\ell^I$ (resp.\ outer strand $\ell^O$) is contained inside $B(0,r_2)$ (resp.\ outside $B(0,r_1)$).
 \end{itemize}
\end{lemma}

\begin{lemma}\label{le:break_loops_bd}
 Let $\Gamma$ be a multichordal \clek{} in $(\D\cap\h; \ul{s}; \beta)$ where all the marked points are on $\partial\D \cap \h$. Let $0<r_1<r_2<1$, the following event occurs almost surely.
 \begin{itemize}
  \item There are finitely many crossings of $A(0,r_1,r_2) \cap \h$ of loops or strands of $\Gamma$.
  \item Let $\CL$ be any loop or chord that crosses $A(0,r_1,r_2) \cap \h$, let $U$ be any connected component of $A(0,r_1,r_2) \cap \h \setminus \CL$ that is adjacent to both $\partial B(0,r_1)$ and $\partial B(0,r_2)$, and let $\gamma_1,\gamma_2$ be the crossings of $\CL$ or the boundary interval adjacent to $U$. There is a chain of loops that intersect in $U$, connect $\gamma_1$ to $\gamma_2$, and such that, with the notation from Definition~\ref{def:disconnecting-chain}, its inner strand $\ell^I$ (resp.\ outer strand $\ell^O$) is contained inside $B(0,r_2) \cap \ol{\h}$ (resp.\ outside $B(0,r_1) \cap \h$).
 \end{itemize}
\end{lemma}

\begin{proof}[Proof of Lemma~\ref{le:break_loops_int} and~\ref{le:break_loops_bd}]
 We first argue in the case of a \clek{}. The number of crossings is a.s.\ finite due to the local finiteness of CLE. Conditionally on the loops of $\Gamma$ up to a given level of nesting, the loops in the next level are independent \clek{} in each connected component surrounded by the previous loops. Therefore the statement follows from Lemma~\ref{lem:finite_chain_bd}.
 
 To generalize the statements to multichordal \clek{} we can apply the exact same argument as in the proof of Lemma~\ref{lem:finite_chain_bd_mcl}.
\end{proof}

\begin{proof}[Proof of Lemma~\ref{lem:break_loops} and~\ref{lem:break_loops_bd}]
The proof is similar to the proofs of Lemmas~\ref{lem:disconnect_interior} and~\ref{lem:disconnect_boundary}, using the event from Lemmas~\ref{le:break_loops_int} and~\ref{le:break_loops_bd} in place of Lemmas~\ref{lem:finite_chain_int} and~\ref{lem:finite_chain_bd_mcl}. We choose the target pivotals as follows. Starting with an outermost loop $\CL$ that crosses the annulus $\varphi_j^{-1}(A(0,r_1,r_2))$, we consider a connected component $U$ of $\varphi_j^{-1}(A(0,r_1,r_2)) \setminus \CL$ that is surrounded by $\CL$. In case $U$ is adjacent to two crossings of $\CL$ but is not adjacent to both $\varphi_j^{-1}(\partial B(0,r_1))$ and $\varphi_j^{-1}(\partial B(0,r_2))$, then the two crossings must intersect, and we can resample the links between them. In case $U$ is adjacent to both $\varphi_j^{-1}(\partial B(0,r_1))$ and $\varphi_j^{-1}(\partial B(0,r_2))$, we use the chain of loops from Lemma~\ref{le:break_loops_int} (resp.~\ref{le:break_loops_bd}), and resample so that instead of crossing the annulus, the loop follows the outer (resp.\ inner) strand before crossing back. Note that when we break up a loop like this, some loops that were previously not in the outermost layer may become loops in the outermost layer. We repeat this procedure with every loop that crosses the annulus until we have broken up all crossings.
\end{proof}

\section{Continuity in total variation}
\label{sec:tv_convergence}

In this section we show that multichordal  \clek{} is continuous in  total variation w.r.t.\ the shape of its domain when we restrict to the law in regions that are away from the boundary. As in the previous sections, we will prove an interior version and a boundary version of this result. Recall that we have already established in Theorem~\ref{th:continuity_mcle} the continuity in the weak topology, and this will be used in our proof for the continuity in total variation. The proofs in this section are based only on the results of Section~\ref{sec:multichordal_ex_uniq}; they do not require Sections~\ref{sec:strands} and~\ref{sec:resampling_tools}.

\subsection{Main statements}
\label{subsec:tv_setup}

\newcommand{\eldomainsym}[1]{{\mathfrak D}_{#1}^{\mathrm{ext,sym}}}

We start by describing the topology on $\eldomain{2N}$ for which we prove the continuity result. It will be a variant of the Carathéodory topology where we add the distance between the marked point configurations.

We first describe the interior variant. Let $z\in\C$, $N\in\N_0$. We consider the following metric $\mathrm{d}^{\eldomain{2N},z}(\cdot,\cdot)$ on the subset of domains in $\eldomain{2N}$ with exterior link pattern $\beta$ between $2N$ points: Let $(D;\ul{x};\beta), (\wt D;\ul{\wt x};\beta)\in\eldomain{2N}$, we let the distance be $\infty$ if either $D$ or $\wt D$ does not contain $z$, otherwise we let $\varphi\colon \D \to D$, $\wt\varphi\colon \D \to \wt D$ be the unique conformal maps with $\varphi(0) = z$, $\varphi'(0) > 0$ (resp.\ $\wt \varphi(0) = z$, $\wt\varphi'(0) > 0$). Let $\mathrm{d}_{\textrm{loc}}$ be a metric inducing the local uniform convergence for continuous functions from $\D$ to $\C$. We let
\[
 \mathrm{d}^{\eldomain{2N},z}((D;\ul{x}), (\wt D;\ul{\wt x})):=\mathrm{d}_{\textrm{loc}}(\varphi,\wt\varphi)+\mathrm{d}_\infty(\varphi^{-1}(\ul{x}),\wt \varphi^{-1}(\ul{\wt x})) ,
\]
where  $\mathrm{d}_\infty$ is the sup norm. We recall the definition of $\domainpair{D}$ and the partially explored \clek{} from Definition~\ref{def:partially-explored-cle}.

\begin{proposition}
\label{prop:mccle_tv_convergence_int}
Let $(D;\ul{x};\beta) \in \eldomain{2N}$ and $z \in D$. Suppose that $((D_n;\ul{x}_n;\beta))$ is a sequence in $\eldomain{2N}$ such that $\mathrm{d}^{\eldomain{2N},z}((D;\ul{x}),(D_n;\ul{x}_n)) \to 0$. For each $n$ let $\Gamma_n$ have law $\mcclelaw{D_n;\ul{x}_n;\beta}$, and let $\Gamma$ have law $\mcclelaw{D;\ul{x};\beta}$. Let $(U,V)\in\domainpair{D}$ with $V \Subset D$, then the law of $(\Gamma_n)_\inside^{*,V,U}$ converges to the law of $\Gamma_\inside^{*,V,U}$ in total variation as $n \to \infty$.
\end{proposition}

We now state the boundary variant of the result. Let $\eldomainsym{2N} \subseteq \eldomain{2N}$ be the subset of marked domains $(D;\ul{x};\beta)$ such that $D$ is symmetric with respect to the real axis, $0 \in D$, and such that the points $\ul{x}$ are on the upper half of $\partial D$.

\begin{proposition}
\label{prop:mccle_tv_convergence_bd}
Let $(D;\ul{x};\beta) \in \eldomainsym{2N}$. Suppose that $((D_n;\ul{x}_n;\beta))$ is a sequence in $\eldomainsym{2N}$ such that $\mathrm{d}^{\eldomain{2N},0}((D;\ul{x}),(D_n;\ul{x}_n)) \to 0$. For each $n$ let $\Gamma_n$ have law $\mcclelaw{D_n\cap\h;\ul{x}_n;\beta}$, and let $\Gamma$ have law $\mcclelaw{D\cap\h;\ul{x};\beta}$. Let $(U,V)\in\domainpair{D\cap\h}$ with $V \Subset D$, then the law of $(\Gamma_n)_\inside^{*,V,U}$ converges to the law of $\Gamma_\inside^{*,V,U}$ in total variation as $n \to \infty$.
\end{proposition}

Note that by the local compactness of the Carathéodory topology, the results say that the function that maps $(D;\ul{x})$ to the law of $\Gamma_\inside^{*,V,U}$ is uniformly continuous on compact sets for $\mathrm{d}^{\eldomain{2N},z}$ of marked domains containing a neighborhood of $V$.

We also prove a variant of the continuity result for the law of the $\CLE_\kappa$ gasket/carpet within $U$. Suppose that $\CI \subseteq \{1,\ldots,2N\}$, and let $\Gamma$ be a multichordal \clek{} in $(D;\ul{x};\beta)$. Let $\Upsilon_{\Gamma;\CI}$ denote the set of points that can be connected to the boundary arc from $x_{i}$ to $x_{i+1 \bmod 2N}$ for some $i \in \CI$ without crossing any loop or chord of $\Gamma$. In the case $N=0$, we let $\Upsilon_{\Gamma;\CI} = \Upsilon_\Gamma$ to be just the gasket/carpet of $\Gamma$.

\begin{proposition}\label{pr:mccle_gasket_tv_convergence}
 Let $N \in \N_0$, $\CI \subseteq \{1,\ldots,2N\}$. Suppose that we have the setup of Proposition~\ref{prop:mccle_tv_convergence_int} or Proposition~\ref{prop:mccle_tv_convergence_bd}. Let $\beta^*$ (resp.\ $\beta^*_n$) denote the link pattern induced by the strands of $\Gamma_\outside^{*,V,U}$ (resp.\ $(\Gamma_n)_\outside^{*,V,U}$) where we consider the connections between the marked points in both $\ul{x}$ and $\Gamma_\inside^{*,V,U}$ (resp.\ $(\Gamma_n)_\inside^{*,V,U}$). Then the law of $((\Gamma_n)_\inside^{*,V,U}, \Upsilon_{\Gamma_n;\CI} \cap U, \beta^*_n)$ converges to the law of $(\Gamma_\inside^{*,V,U}, \Upsilon_{\Gamma;\CI} \cap U, \beta^*)$ in total variation as $n \to \infty$.
\end{proposition}

In Section~\ref{subsec:cle_kappa_tv} we will prove the total variation convergence in the case of \clek{} (without any marked points) and then in Section~\ref{subsec:mccle_kappa_tv} we will extend the result to the case with marked boundary points.

\subsection{The case of \clek{}}
\label{subsec:cle_kappa_tv}

We start by considering the case of no marked points. In this setting we can deduce the total variation convergence for $(\Gamma_n)_\inside^{*,V,U} $ from the convergence in total variation of the GFFs in the coupling for $\kappa \in (4,8)$ resp.\ the coupling with the Brownian loop soup for $\kappa \in (8/3,4]$. In Subsection~\ref{subsec:mccle_kappa_tv} we will be able to reduce the general case to this case by using a resampling argument.

\begin{lemma}
\label{le:cle_tv_convergence}
The statements of Propositions~\ref{prop:mccle_tv_convergence_int} and~\ref{prop:mccle_tv_convergence_bd} hold in the case $N=0$.
\end{lemma}

\begin{figure}[ht]
 \centering
 \includegraphics[width=0.3\textwidth]{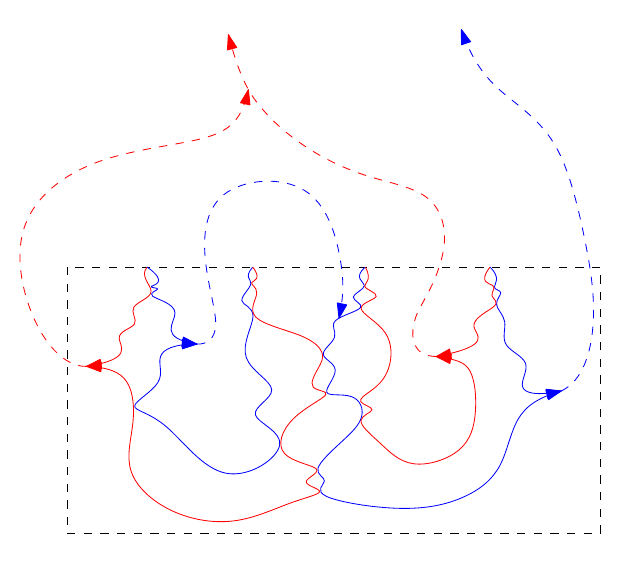}\includegraphics[width=0.3\textwidth]{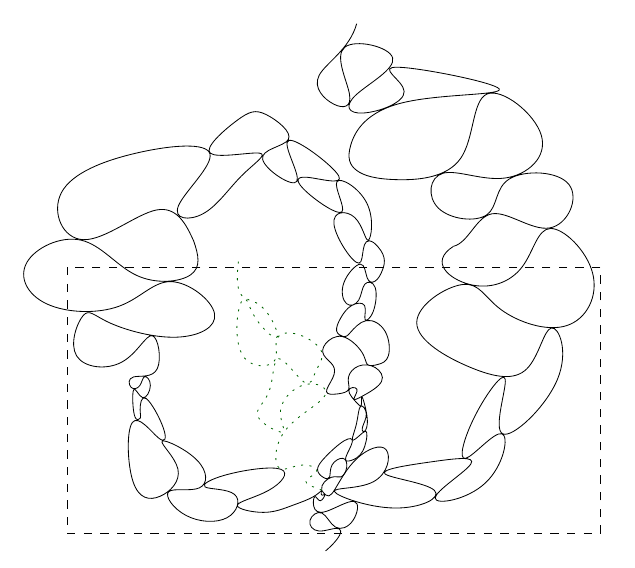}\\
 \includegraphics[width=0.3\textwidth]{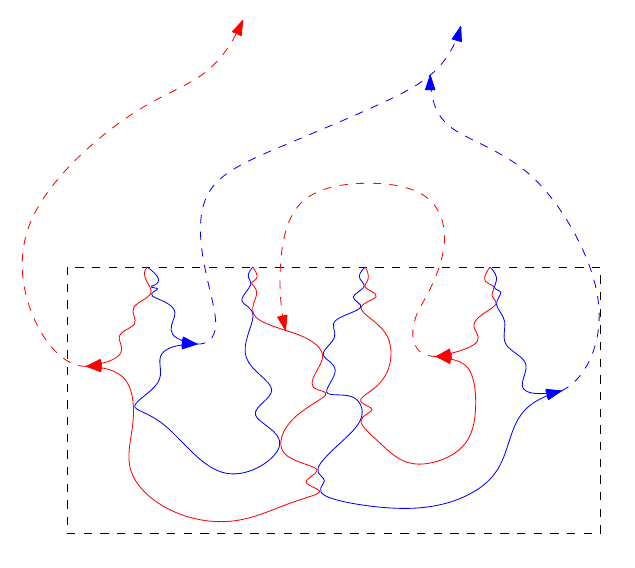}\includegraphics[width=0.3\textwidth]{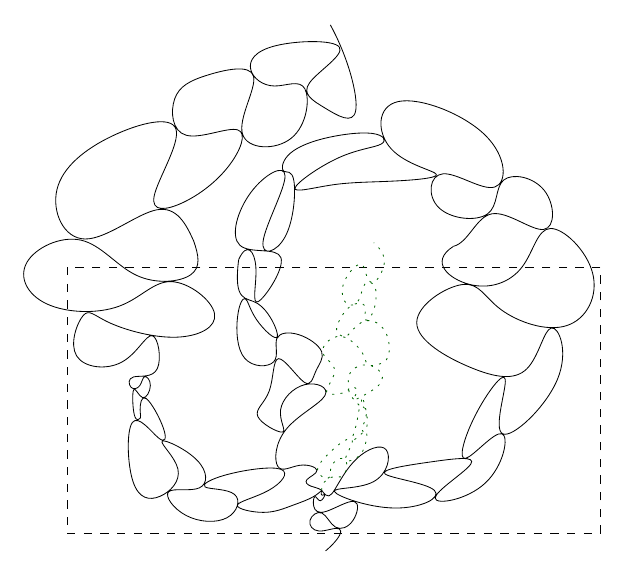}
 \caption{Illustration of why the order of the strands of the space-filling \slek{} is relevant for determining the \clek{} strands. Shown are two different orderings of the identical strands, and the resulting \clek{} configurations. The red (resp.\ blue) flow lines have angle $\pi/2$ (resp.\ $-\pi/2$).}
 \label{fi:spf_ordering}
\end{figure}

\begin{lemma}\label{le:spf_order_cle}
 Let $\kappa \in (4,8)$. Let $\Gamma$ be a \clek{} in a simply connected domain $D$. Let $U_1,U_2,U_3$ be such that $(U_1,U_2) \in \domainpair{D}$, $(U_2,U_3) \in \domainpair{D}$. Let $\eta$ be the space-filling \slek{} associated with an exploration tree for~$\Gamma$. Let $\CH$ be the collection of (oriented) maximal segments of $\eta$ that intersect $U_2$ and are contained in $\ol{U_3}$. Then $\CH$ together with the ordering of the strands in $\CH$ determine $\Gamma_\inside^{*,U_2,U_1}$.
\end{lemma}

We remark that different orderings of the strands in $\CH$ may lead to different configurations of the \clek{} strands, see Figure~\ref{fi:spf_ordering} for an example.

\begin{proof}[Proof of Lemma~\ref{le:spf_order_cle}]
 The \clek{} loops inside the regions disconnected by a single strand in $\CH$ are clearly determined by the strand. Therefore we only need to consider the loops traced by the boundaries of the strands.
 
 Suppose $\gamma_1,\gamma_2 \in \CH$ are two intersecting strands and $\gamma_1$ is visited first by $\eta$. Let $z$ be an intersection point. Consider the first time $\gamma_2$ intersects $\gamma_1$ at $z$. Let $\CL$ be the loop that $\gamma_2$ is tracing before hitting $z$. Upon intersecting, $\gamma_2$ continues in the component toward the same direction as $\gamma_1$. On the other hand, $\CL$ branches toward the opposite direction of $\gamma_1$. See Figure~\ref{fi:spf_ordering}. That is, $\gamma_2$ begins tracing new loops after visiting $z$. Let $\gamma_3 \in \CH$ be the strand corresponding to when $\eta$ comes back to $z$ again (note that $\gamma_3$ necessarily intersects $U_2$ because $\gamma_1,\gamma_2$ do; it is possible that $\gamma_3 = \gamma_2$ if it does not exit $U_3$ in between). The continuation of $\CL$ is given by (a subset of) $\gamma_3$ starting from $z$.
 
 In summary, we see that
 \begin{itemize}
  \item a loop begins when a strand in $\CH$ hits a previous strand,
  \item a loop terminates when a strand in $\CH$ passes through the intersection of two previous strands (in the opposite direction),
  \item a loop continues at an intersection point of two strands when the or another strand passes through the intersection again (in the opposite direction).
 \end{itemize}
 This shows that all the \clek{} strands contained in $U_2$ are determined by the strands of the space-filling \slek{} in $\CH$ together with their orientations and ordering.
\end{proof}

\begin{proof}[Proof of Lemma~\ref{le:cle_tv_convergence} for $\kappa \in (4,8)$]
 The proof is exactly the same for the interior and for the boundary case, only the notation is slightly different. We write out the proof for the interior case (Proposition~\ref{prop:mccle_tv_convergence_int}).
 
 Let $\Gamma$ (resp.\ $\Gamma_n$) be coupled with a GFF $h$ (resp.\ $h_n$) as in Section~\ref{subsubsec:cle_gff}. Let $V \Subset U_2 \Subset U_3 \Subset D$. The law of the restriction $h_n\big|_{U_3}$ converges in total variation to the law of $h\big|_{U_3}$.
 
 Let $\CH$ be the collection of strands of the space-filling \slek{} that intersect $U_2$, extended until exiting~$U_3$. Then $\CH$ is determined by the values of the GFF in $U_3$. By Lemma~\ref{le:spf_order_cle}, the proposition follows if we can also show that the ordering of the strands in $\CH$ converges in total variation as $n\to\infty$. In order to determine the ordering of the strands, we need to extend the flow lines on the boundaries of the strands in $\CH$ until they hit $\partial D$ (resp.\ $\partial D_n$).
 
 Let $z \in D$. Let $\varphi\colon \D \to D$ (resp.\ $\varphi_n\colon \D \to D_n$) be the conformal map with $\varphi(0) = z$ and $\varphi'(0) > 0$ (resp.\ $\varphi_n(0) = z$ and $\varphi_n'(0) > 0$).
 
 Let $\varepsilon > 0$ be given. Let $h_\D$ be a GFF in $\D$ compatible with a \clek{}. Let $\delta > 0$ be small enough so that $\varphi_n(B(0,1-\delta)) \supseteq U_3$ for all sufficiently large $n$. Given $\delta$, let $\delta_1 > 0$ be small enough so that with probability at least $1-\varepsilon$ every flow line with angle $\pi/2$ or $-\pi/2$ starting in $B(0,1-\delta)$ and exiting $B(0,1-\delta_1)$ hits $\partial\D$ without entering $B(0,1-\delta)$ again.
 
 Let $U_4 \Subset D$ be such that $\varphi_n(B(0,1-\delta_1)) \subseteq U_4$ for all sufficiently large $n$. The law of the restriction $h_n\big|_{U_4}$ converges in total variation to the law of $h\big|_{U_4}$. By the choice of $\delta_1$, the total variation distance between the true ordering of the strands in $\CH$ and the ordering obtained by extending the flow lines just to $\partial U_4$ is at most $\varepsilon$. This by Lemma~\ref{le:spf_order_cle} implies the result.
\end{proof}

For the proof of Lemma~\ref{le:cle_tv_convergence} for $\kappa \in (8/3,4]$, it is convenient to use the loop soup construction of \clek{}. Let $\mu$ be the Brownian loop measure \cite{lw-loop-soup}, and let $\mu_D$ be its restriction to the set of loops contained in $D$. Let $\Lambda_D$ be the \emph{Brownian loop soup} of intensity $\mathbf{c}>0$ in $D$, i.e.\ the Poisson point process with intensity $\mathbf{c}\mu_D$. The clusters of $\Lambda_D$ are the equivalence classes under the relation that relates all pairs of intersecting loops. Let $\Gamma_D$ be the collection of outer boundaries of the outermost loop soup clusters. It is shown in \cite{sw2012cle} that when $\kappa \in (8/3,4]$ and $\mathbf{c} = 1-6\chi^2 = (3\kappa-8)(6-\kappa)/(2\kappa) \in (0,1]$, the law of $\Gamma_D$ is that of a non-nested \clek{} in $D$.

\begin{proof}[Proof of Lemma~\ref{le:cle_tv_convergence} for {$\kappa \in (8/3,4]$}]
The proof is exactly the same for the interior and for the boundary case, only the notation is slightly different. We write out the proof for the interior case (Proposition~\ref{prop:mccle_tv_convergence_int}).

We couple the \clek{} on $D$ and on $D_n$ with the same Brownian loop soup. Let $\Lambda$ (resp.\ $\Lambda_n$) denote the collection of Brownian loops in the loop soup contained in $D$ (resp.\ $D_n$), and let $\Gamma$ (resp.\ $\Gamma_n$) denote the induced non-nested \clek{}. (It is enough to consider the non-nested \clek{} since the nested \clek{} is obtained by sampling further \clek{} inside its loops.) We argue that for large $n$, with high probability the loop-soup clusters intersecting $V$ are the same for $\Lambda$ and $\Lambda_n$.

Let $z \in D$. Let $\varphi\colon \D \to D$ (resp.\ $\varphi_n\colon \D \to D_n$) be the conformal map with $\varphi(0) = z$ and $\varphi'(0) > 0$ (resp.\ $\varphi_n(0) = z$ and $\varphi_n'(0) > 0$).

Let $\varepsilon > 0$. Let $\Gamma_\D$ be a \clek{} in $\D$. Let $\delta > 0$ be small enough so that $\varphi_n(B(0,1-\delta)) \supseteq V$ for all sufficiently large $n$. Given $\delta$, let $\delta_1 > 0$ be small enough so that with probability at least $1-\varepsilon$ each loop of $\Gamma_\D$ whose interior intersects $B(0,1-\delta)$ is contained in $B(0,1-\delta_1)$.

Let $V_1 \Subset D$ be such that $\varphi_n(B(0,1-\delta_1)) \subseteq V_1$ for all sufficiently large $n$. Then, with probability at least $1-2\varepsilon$, each loop soup cluster of $\Lambda$ and $\Lambda_n$ whose interior intersects $V$ is contained in $V_1$. Since the loops of $\Lambda$ and $\Lambda_n$ contained in $V_1$ agree, each such loop soup cluster of $\Lambda$ is contained in a loop soup cluster of $\Lambda_1$ and vice versa. Consequently, these loop soup clusters agree.
\end{proof}

\subsection{The case of multichordal \clek{}}
\label{subsec:mccle_kappa_tv}

We now move to the case of general multichordal \clek{}. The main step for the proof in the case of marked domains consists of using  a resampling argument which allows us to deduce the general result from the case of  \clek{} proved in the previous subsection. One technical challenge for the interior variant (Proposition~\ref{prop:mccle_tv_convergence_int}) is that if we condition on the \clek{} configuration in $U \Subset D$, the remainder lives in an annular region for which we have not developed the theory. To overcome this, we will first prove the boundary variant (Proposition~\ref{prop:mccle_tv_convergence_bd}) and then deduce the interior variant by successively replacing parts of the boundary.

\begin{figure}[ht]
 \centering
 \includegraphics[width=0.6\textwidth]{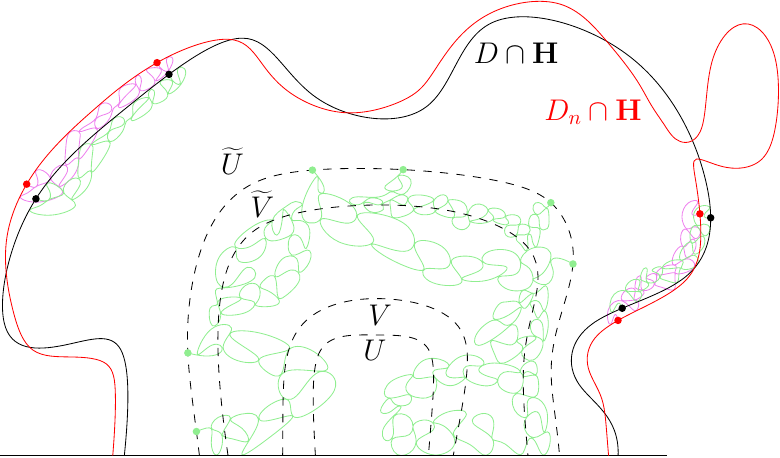}
 \caption{Illustration of the proof of Proposition~\ref{prop:mccle_tv_convergence_bd}.}
 \label{fi:mccle_tv_bd_pf}
\end{figure}

\begin{proof}[Proof of Proposition~\ref{prop:mccle_tv_convergence_bd}]
 Let us write $\p = \mcclelaw{D\cap\h;\ul{x};\beta}$, $\p_n = \mcclelaw{D_n\cap\h;\ul{x}_n;\beta}$. Let $\varphi\colon \D \to D$ (resp.\ $\varphi_n\colon \D \to D_n$) be the conformal map with $\varphi(0) = 0$ and $\varphi'(0) > 0$ (resp.\ $\varphi_n(0) = 0$ and $\varphi_n'(0) > 0$). Note that $\varphi$ (resp.\ $\varphi_n$) maps $\D\cap\h$ conformally to $D\cap\h$ (resp.\ $D_n\cap\h$).
 
 Let $\varepsilon > 0$ be given. Let $A$ be any event for $\Gamma_\inside^{*,V,U}$. Our goal is to show that
 \[ \sup_{A \in \sigma(\Gamma_\inside^{*,V,U})} \abs{\p[A]-\p_n[A]} < \varepsilon \]
 for sufficiently large $n$.
 
 Let $(\wt{U},\wt{V}) \in \domainpair{D\cap\h}$ be such that $\wt{V} \cap V = \varnothing$ and $\dist(D\cap\h\setminus \wt{U}, \partial D) > 0$. Let $\wt{U}_n = D_n \cap \h \setminus (D \setminus \wt{U})$ and $\wt{V}_n = D_n \cap \h \setminus (D \setminus \wt{V})$. See Figure~\ref{fi:mccle_tv_bd_pf}. Note that by the convergence $\mathrm{d}^{\eldomain{2N},0}((D;\ul{x}),(D_n;\ul{x}_n)) \to 0$ we have $\partial \wt{U} \cap D \subseteq D_n$ for all sufficiently large $n$.
 
 Note that $\Gamma_\inside^{*,V,U}$ (resp.\ $(\Gamma_n)_\inside^{*,V,U}$) is determined by $\Gamma_\outside^{*,\wt{V},\wt{U}}$ (resp.\ $(\Gamma_n)_\outside^{*,\wt{V}_n,\wt{U}_n}$).
 
 For $\delta_1 > 0$, let $G$ be the event that the marked points of $\Gamma_\outside^{*,\wt{V},\wt{U}}$ have distance at least $\delta_1$ from each other. Theorem~\ref{th:continuity_mcle} implies that we can find $\delta_1 > 0$ such that $\sup_n \p_n[G^c] < \varepsilon$.
 
 For $\delta > 0$, let $E$ (resp.\ $E_n$) be the event that all chords emanating from the marked points on $\partial D$ (resp.\ $\partial D_n$) are entirely contained in $\varphi(A(0,1-\delta,1) \cap \h)$ (resp.\ $\varphi_n(A(0,1-\delta,1) \cap \h)$). 
 Fix $M \in \N$ large. For $k=1,\ldots,M$ we let $F_k$ be the event that $\p[ E \mid \Gamma_\outside^{*,\wt{V},\wt{U}} ] \in (\frac{k-1}{M},\frac{k}{M}]$.
 
 By Theorem~\ref{thm:cle_partially_explored}, the conditional probability $\p[ E \mid \Gamma_\outside^{*,\wt{V},\wt{U}} ]$ (resp.\ $\p_n[ E_n \mid (\Gamma_n)_\outside^{*,\wt{V}_n,\wt{U}_n} ]$) is a.s.\ a function $f$ (resp.\ $f_n$) of the marked domain $(\wt{V}^{*,\wt{U}};\ul{x}^*;\beta^*)$ corresponding to $\Gamma_\outside^{*,\wt{V},\wt{U}}$ (resp.\ $(\Gamma_n)_\outside^{*,\wt{V}_n,\wt{U}_n}$). By Theorem~\ref{th:continuity_mcle} we have that $f_n \to f$ pointwise for each $\delta > 0$ (except possibly for countably many). By a compactness argument, the convergence rate is uniform for all marked domains $(\wt{V}^{*,\wt{U}};\ul{x}^*;\beta^*)$ where the marked points on $\partial \wt{U}$ have distance at least $\delta_1$ from each other. Consequently, for each $k=1,\ldots,M$
 \[
  \limsup_{n\to\infty} \sup_{A \in \sigma(\Gamma_\outside^{*,\wt{V},\wt{U}})} \abs*{\frac{\p_n[A \cap F_k \cap G \cap E_n]}{\p_n[A \cap F_k \cap G]} - \frac{k}{M}} \le \frac{1}{M} .
 \]
 Hence, for $n$ sufficiently large, we have for all $A \in \sigma(\Gamma_\outside^{*,\wt{V},\wt{U}})$,
 \begin{multline*}
  \abs{\p[A \cap F_k \cap G] - \p_n[A \cap F_k \cap G]} \\
  \le \frac{M}{k} \abs{\p[A \cap F_k \cap G \cap E] - \p_n[A \cap F_k \cap G \cap E_n]} + \frac{1}{k}\p[A \cap F_k \cap G] + \frac{1+\varepsilon}{k}\p_n[A \cap F_k \cap G] .
 \end{multline*}
 We can alternatively first sample the chords, and then (on the event $E$ resp.\ $E_n$) sample the \clek{} in the complementary components. Therefore Theorem~\ref{th:continuity_mcle} and Lemma~\ref{le:cle_tv_convergence} imply that
 \[
  \lim_{n\to\infty} \sup_{A \in \sigma(\Gamma_\outside^{*,\wt{V},\wt{U}})} \abs{\p[A \cap F_k \cap G \cap E] - \p_n[A \cap F_k \cap G \cap E_n]} = 0 .
 \]
 Combining everything, we see that
 \[ \begin{split}
  &\limsup_{n\to\infty} \sup_{A \in \sigma(\Gamma_\outside^{*,\wt{V},\wt{U}})} \sum_{k \ge M^{1/2}} \abs{\p[A \cap F_k] - \p_n[A \cap F_k]} \\
  &\quad\le \varepsilon + \limsup_{n\to\infty} \sup_{A \in \sigma(\Gamma_\outside^{*,\wt{V},\wt{U}})} \sum_{k \ge M^{1/2}} \abs{\p[A \cap F_k \cap G] - \p_n[A \cap F_k \cap G]} \\
  &\quad\le \varepsilon + \frac{2}{M^{1/2}} + \limsup_{n\to\infty} \sup_{A \in \sigma(\Gamma_\outside^{*,\wt{V},\wt{U}})} \sum_{k \ge M^{1/2}} \frac{M}{k} \abs{\p[A \cap F_k \cap G \cap E] - \p_n[A \cap F_k \cap G \cap E_n]} \\
  &\quad\le \varepsilon + \frac{2}{M^{1/2}} .
 \end{split} \]
 This bounds the error for all $k \gg 1$. To deal with the small $k$, we apply again Theorem~\ref{th:continuity_mcle} to see that
 \[
  \lim_{M\to\infty} \sup_n \p_n\left[\bigcup_{k \le M^{1/2}} F_k\right] = 0 .
 \]
 Therefore, if $M$ is chosen sufficiently large, we conclude that
 \[\begin{split}
  &\limsup_{n\to\infty} \sup_{A \in \sigma(\Gamma_\outside^{*,\wt{V},\wt{U}})} \abs{\p[A]-\p_n[A]} \\
  &\quad \le \sup_n \p_n\left[\bigcup_{k \le M^{1/2}} F_k\right] + \limsup_{n\to\infty} \sup_{A \in \sigma(\Gamma_\outside^{*,\wt{V},\wt{U}})} \sum_{k \ge M^{1/2}} \abs{\p[A \cap F_k] - \p_n[A \cap F_k]} \\
  &\quad \le 2\varepsilon + \frac{2}{M^{1/2}} .
 \end{split}\]
 This completes the proof.
\end{proof}

\begin{figure}[ht]
 \centering
 \includegraphics[width=0.6\textwidth]{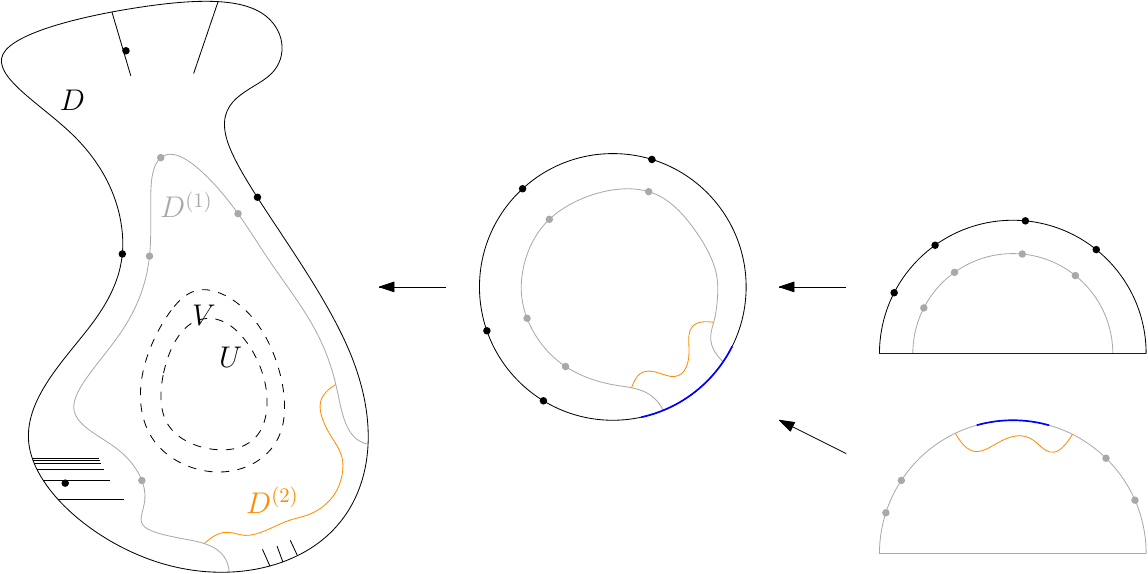}
 \caption{Illustration of Step~1 of the proof of Proposition~\ref{prop:mccle_tv_convergence_int}.}
 \label{fi:mccle_tv_int_pf1}
\end{figure}

\begin{figure}[ht]
 \centering
 \includegraphics[width=0.6\textwidth]{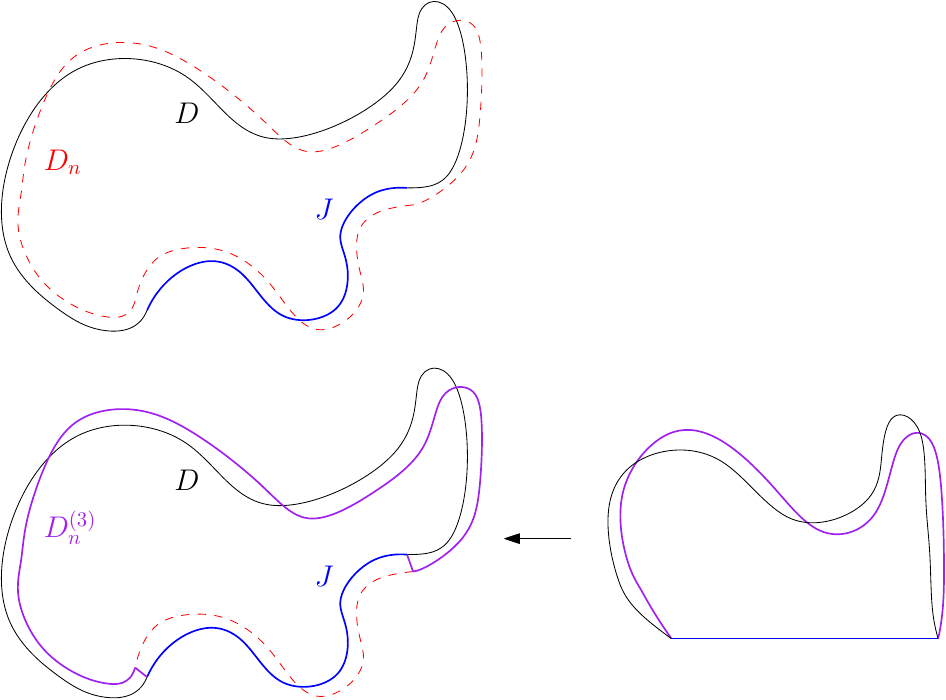}
 \caption{Illustration of Step~2 of the proof of Proposition~\ref{prop:mccle_tv_convergence_int}.}
 \label{fi:mccle_tv_int_pf2}
\end{figure}

\begin{proof}[Proof of Proposition~\ref{prop:mccle_tv_convergence_int}]
 We deduce this result from the boundary variant by subsequently replacing pieces of the domain boundaries; see Figures~\ref{fi:mccle_tv_int_pf1} and~\ref{fi:mccle_tv_int_pf2} for an illustration.
 
 \textbf{Step 1.} 
 We first approximate the domains $D$ resp.\ $D_n$ from the interior by Jordan domains. Let $\varphi$ (resp.\ $\varphi_n$) be the corresponding conformal maps from $\D$ to $D$ (resp.\ $D_n$). Let $I \subseteq \partial\D$ be a boundary arc disjoint from $\varphi^{-1}(\ul{x})$. Consider a conformal map $f\colon \D\cap\h \to \D$ with $f([-1,1]) = I$. Let $\delta > 0$, and let $D^{(1)} = \varphi(f(B(0,1-\delta)\cap\h))$ and $\ul{x}^{(1)} = \varphi \circ f((1-\delta)(\varphi \circ f)^{-1}(\ul{x}))$. Let $D_n^{(1)}$, $\ul{x}_n^{(1)}$ be defined analogously with $\varphi_n$ in place of $\varphi$. Let $\Gamma^{(1)}$ (resp.\ $\Gamma_n^{(1)}$) have law $\mcclelaw{D^{(1)};\ul{x}^{(1)};\beta}$ (resp.\ $\mcclelaw{D_n^{(1)};\ul{x}_n^{(1)};\beta}$). For $\delta > 0$ small enough, by Proposition~\ref{prop:mccle_tv_convergence_bd} (applied to the image of the \clek{} under $(\varphi \circ f)^{-1}$ resp.\ $(\varphi_n \circ f)^{-1}$) and the local compactness of the Carathéodory topology, the total variation between $(\Gamma_n)_\inside^{*,V,U}$ and $(\Gamma_n^{(1)})_\inside^{*,V,U}$ is at most $\varepsilon$ uniformly in $n$, and the total variation between $\Gamma_\inside^{*,V,U}$ and $(\Gamma^{(1)})_\inside^{*,V,U}$ is at most $\varepsilon$.
 
 Repeating this procedure once more, we find $(D^{(2)};\ul{x}^{(2)};\beta) \in \eldomain{2N}$ and $(D_n^{(2)};\ul{x}_n^{(2)};\beta) \in \eldomain{2N}$ so that
 \begin{itemize}
  \item $D^{(2)} \subseteq D$, $D_n^{(2)} \subseteq D_n$ are Jordan domains, and $V \Subset D^{(2)}$,
  \item $\partial D_n^{(2)} \to \partial D^{(2)}$ in the uniform topology of curves, and $\ul{x}_n^{(2)} \to \ul{x}^{(2)}$,
  \item if $\Gamma^{(2)}$, $\Gamma_n^{(2)}$ are the multichordal \clek{} in these marked domains, the total variation between $(\Gamma_n)_\inside^{*,V,U}$ and $(\Gamma_n^{(2)})_\inside^{*,V,U}$ is at most $\varepsilon$ uniformly in $n$, and the total variation between $\Gamma_\inside^{*,V,U}$ and $(\Gamma^{(2)})_\inside^{*,V,U}$ is at most $\varepsilon$.
 \end{itemize}
 
 Therefore in order to prove Proposition~\ref{prop:mccle_tv_convergence_int}, we only need to prove it for Jordan domains whose boundaries converge uniformly.
 
 \textbf{Step 2.} 
 Suppose that $D$, $D_n$ are Jordan domains, $\partial D_n \to \partial D$ in the uniform topology, and $\ul{x}_n \to \ul{x}$. Suppose $(U,V)\in\domainpair{D}$ and $V \Subset D$.
 
 Let $J \subseteq \partial D$ be a boundary arc disjoint from $\ul{x}$. We construct a sequence of marked Jordan domains $(D_n^{(3)};\ul{x}^{(3)};\beta) \in \eldomain{2N}$ so that $\partial D_n^{(3)}$ contains $J$, agrees with $\partial D_n$ on a large part of its boundary, and $\mathrm{d}_\infty(\partial D_n^{(3)}, \partial D_n) \to 0$. Since $D$, $D_n$ are Jordan domains, we can choose $D_n^{(3)}$ so that there is a common conformal map that takes $D$, $D_n^{(3)}$ to the setup of Proposition~\ref{prop:mccle_tv_convergence_bd} where $J$ is mapped to a real interval. See Figure~\ref{fi:mccle_tv_int_pf2}. Hence, $(\Gamma_n^{(3)})_\inside^{*,V,U}$ converges to $\Gamma_\inside^{*,V,U}$ in total variation. Finally, applying Proposition~\ref{prop:mccle_tv_convergence_bd} another time, together with the local compactness of the Carathéodory topology, shows that the total variation between $(\Gamma_n^{(3)})_\inside^{*,V,U}$ and $(\Gamma_n)_\inside^{*,V,U}$ converges to $0$ as $n \to \infty$. This concludes the proof.
\end{proof}

\begin{proof}[Proof of Proposition~\ref{pr:mccle_gasket_tv_convergence}]
 Again, we first show the boundary version. Let $(\wt{U},\wt{V}) \in \domainpair{D\cap\h}$ be as in the proof of Proposition~\ref{prop:mccle_tv_convergence_bd}. Note that $\Upsilon_{\Gamma;\CI} \cap U$ is determined by $\Gamma_\outside^{*,\wt{V},\wt{U}}$ and the interior linking pattern of $\Gamma_\inside^{*,\wt{V},\wt{U}}$. By Theorem~\ref{th:continuity_mcle}, the conditional law of the interior linking pattern given $\Gamma_\outside^{*,\wt{V},\wt{U}}$ converges as $n\to\infty$. Together with Proposition~\ref{prop:mccle_tv_convergence_bd}, this implies the result.
 
 The interior version can be deduced from the boundary version by the same argument as in the proof of Proposition~\ref{prop:mccle_tv_convergence_int}.
\end{proof}

\bibliographystyle{alpha}

\begin{thebibliography}{KMSW19}

\bibitem[AHSY23]{ahsy-welding-msle}
Morris Ang, Nina Holden, Xin Sun, and Pu~Yu.
\newblock Conformal welding of quantum disks and multiple {SLE}: the non-simple
  case.
\newblock {\em ArXiv e-prints}, 2023.

\bibitem[AM22]{am2022sle8}
Valeria {Ambrosio} and Jason {Miller}.
\newblock {A continuous proof of the existence of the SLE$_8$ curve}.
\newblock {\em arXiv e-prints}, page arXiv:2203.13805, March 2022.

\bibitem[AMY25]{amy2025tightness}
Valeria {Ambrosio}, Jason {Miller}, and Yizheng {Yuan}.
\newblock {Tightness of approximations to metrics on non-simple conformal loop
  ensemble gaskets}.
\newblock {\em arXiv e-prints}, page arXiv:2507.15589, July 2025.

\bibitem[Bef08]{b2008dimension}
Vincent Beffara.
\newblock The dimension of the {SLE} curves.
\newblock {\em Ann. Probab.}, 36(4):1421--1452, 2008.

\bibitem[BH19]{bh2019ising}
St\'{e}phane Benoist and Cl\'{e}ment Hongler.
\newblock The scaling limit of critical {I}sing interfaces is
  {$\mathrm{CLE}_3$}.
\newblock {\em Ann. Probab.}, 47(4):2049--2086, 2019.

\bibitem[BPW21]{bpw-msle-uniqueness}
Vincent Beffara, Eveliina Peltola, and Hao Wu.
\newblock On the uniqueness of global multiple {SLE}s.
\newblock {\em Ann. Probab.}, 49(1):400--434, 2021.

\bibitem[CN06]{cn2006cle}
Federico Camia and Charles~M. Newman.
\newblock Two-dimensional critical percolation: the full scaling limit.
\newblock {\em Comm. Math. Phys.}, 268(1):1--38, 2006.

\bibitem[DMS21]{dms2021mating}
Bertrand Duplantier, Jason Miller, and Scott Sheffield.
\newblock Liouville quantum gravity as a mating of trees.
\newblock {\em Ast\'{e}risque}, (427):viii+257, 2021.

\bibitem[Dub07]{dub-commutation}
Julien Dub\'edat.
\newblock Commutation relations for {S}chramm-{L}oewner evolutions.
\newblock {\em Comm. Pure Appl. Math.}, 60(12):1792--1847, 2007.

\bibitem[Dub09]{dub2009gff}
Julien Dub\'{e}dat.
\newblock S{LE} and the free field: partition functions and couplings.
\newblock {\em J. Amer. Math. Soc.}, 22(4):995--1054, 2009.

\bibitem[FLPW24]{flpw-msle}
Yu~Feng, Mingchang Liu, Eveliina Peltola, and Hao Wu.
\newblock Multiple {SLE}s for $\kappa\in(0,8)$: {C}oulomb gas integrals and
  pure partition functions.
\newblock {\em ArXiv e-prints}, 2024.

\bibitem[FPW24]{fpw-connection-prob}
Yu~Feng, Eveliina Peltola, and Hao Wu.
\newblock Connection probabilities of multiple {FK}-{I}sing interfaces.
\newblock {\em Probab. Theory Related Fields}, 189(1-2):281--367, 2024.

\bibitem[GM21a]{gm2021saw}
Ewain Gwynne and Jason Miller.
\newblock Convergence of the self-avoiding walk on random quadrangulations to
  {$\rm SLE_{8/3}$} on {$\sqrt{8/3}$}-{L}iouville quantum gravity.
\newblock {\em Ann. Sci. \'{E}c. Norm. Sup\'{e}r. (4)}, 54(2):305--405, 2021.

\bibitem[GM21b]{gm2021percolation}
Ewain Gwynne and Jason Miller.
\newblock Percolation on uniform quadrangulations and {$\rm SLE_6$} on
  {$\sqrt{8/3}$}-{L}iouville quantum gravity.
\newblock {\em Ast\'{e}risque}, (429):vii+242, 2021.

\bibitem[GMQ21]{gmq2021sphere}
Ewain Gwynne, Jason Miller, and Wei Qian.
\newblock Conformal invariance of {${\rm CLE}_\kappa$} on the {R}iemann sphere
  for {$\kappa\in(4, 8)$}.
\newblock {\em Int. Math. Res. Not. IMRN}, (23):17971--18036, 2021.

\bibitem[KL07]{kl-msle}
Michael~J. Kozdron and Gregory~F. Lawler.
\newblock The configurational measure on mutually avoiding {SLE} paths.
\newblock In {\em Universality and renormalization}, volume~50 of {\em Fields
  Inst. Commun.}, pages 199--224. Amer. Math. Soc., Providence, RI, 2007.

\bibitem[KMS22]{kms2022removability}
Konstantinos {Kavvadias}, Jason {Miller}, and Lukas {Schoug}.
\newblock {Conformal removability of SLE$_4$}.
\newblock {\em arXiv e-prints}, page arXiv:2209.10532, September 2022.

\bibitem[KMSW19]{kmsw2019bipolar}
Richard Kenyon, Jason Miller, Scott Sheffield, and David~B. Wilson.
\newblock Bipolar orientations on planar maps and {${\rm SLE}_{12}$}.
\newblock {\em Ann. Probab.}, 47(3):1240--1269, 2019.

\bibitem[KS19]{ks2019fkising}
Antti Kemppainen and Stanislav Smirnov.
\newblock Conformal invariance of boundary touching loops of {FK} {I}sing
  model.
\newblock {\em Comm. Math. Phys.}, 369(1):49--98, 2019.

\bibitem[LPW25]{lpw-ust}
Mingchang Liu, Eveliina Peltola, and Hao Wu.
\newblock Uniform spanning tree in topological polygons, partition functions
  for {SLE}(8), and correlations in c=--2 logarithmic {CFT}.
\newblock {\em Ann. Probab.}, 53(1):23--78, 2025.

\bibitem[LSW03]{lsw2003restriction}
Gregory Lawler, Oded Schramm, and Wendelin Werner.
\newblock Conformal restriction: the chordal case.
\newblock {\em J. Amer. Math. Soc.}, 16(4):917--955, 2003.

\bibitem[LSW04]{lsw2004lerw}
Gregory~F. Lawler, Oded Schramm, and Wendelin Werner.
\newblock Conformal invariance of planar loop-erased random walks and uniform
  spanning trees.
\newblock {\em Ann. Probab.}, 32(1B):939--995, 2004.

\bibitem[LSW17]{lsw2017schnyder}
Yiting {Li}, Xin {Sun}, and Samuel~S. {Watson}.
\newblock {Schnyder woods, SLE(16), and Liouville quantum gravity}.
\newblock {\em arXiv e-prints}, page arXiv:1705.03573, May 2017.

\bibitem[LW04]{lw-loop-soup}
Gregory~F. Lawler and Wendelin Werner.
\newblock The {B}rownian loop soup.
\newblock {\em Probab. Theory Related Fields}, 128(4):565--588, 2004.

\bibitem[MQ20]{mq2020notsle}
Jason Miller and Wei Qian.
\newblock The geodesics in {L}iouville quantum gravity are not
  {S}chramm-{L}oewner evolutions.
\newblock {\em Probab. Theory Related Fields}, 177(3-4):677--709, 2020.

\bibitem[MS16a]{ms2016ig1}
Jason Miller and Scott Sheffield.
\newblock Imaginary geometry {I}: interacting {SLE}s.
\newblock {\em Probab. Theory Related Fields}, 164(3-4):553--705, 2016.

\bibitem[MS16b]{ms2016ig2}
Jason Miller and Scott Sheffield.
\newblock Imaginary geometry {II}: reversibility of {${\rm
  SLE}_\kappa(\rho_1;\rho_2)$} for {$\kappa\in(0,4)$}.
\newblock {\em Ann. Probab.}, 44(3):1647--1722, 2016.

\bibitem[MS16c]{ms2016ig3}
Jason Miller and Scott Sheffield.
\newblock Imaginary geometry {III}: reversibility of {$\rm SLE_\kappa$} for
  {$\kappa\in(4,8)$}.
\newblock {\em Ann. of Math. (2)}, 184(2):455--486, 2016.

\bibitem[MS17]{ms2017ig4}
Jason Miller and Scott Sheffield.
\newblock Imaginary geometry {IV}: interior rays, whole-plane reversibility,
  and space-filling trees.
\newblock {\em Probab. Theory Related Fields}, 169(3-4):729--869, 2017.

\bibitem[MS19]{ms2019lightcone}
Jason Miller and Scott Sheffield.
\newblock Gaussian free field light cones and {${\rm SLE}_\kappa(\rho)$}.
\newblock {\em Ann. Probab.}, 47(6):3606--3648, 2019.

\bibitem[MSW14]{msw2014dimension}
Jason Miller, Nike Sun, and David~B. Wilson.
\newblock The {H}ausdorff dimension of the {CLE} gasket.
\newblock {\em Ann. Probab.}, 42(4):1644--1665, 2014.

\bibitem[MSW17]{msw2017cleperc}
Jason Miller, Scott Sheffield, and Wendelin Werner.
\newblock C{LE} percolations.
\newblock {\em Forum Math. Pi}, 5:e4, 102, 2017.

\bibitem[MSW20]{msw2020nonsimple}
Jason Miller, Scott Sheffield, and Wendelin Werner.
\newblock Non-simple {SLE} curves are not determined by their range.
\newblock {\em J. Eur. Math. Soc. (JEMS)}, 22(3):669--716, 2020.

\bibitem[MW17]{mw2017intersections}
Jason Miller and Hao Wu.
\newblock Intersections of {SLE} paths: the double and cut point dimension of
  {SLE}.
\newblock {\em Probab. Theory Related Fields}, 167(1-2):45--105, 2017.

\bibitem[MW18]{mw2018connection}
Jason Miller and Wendelin Werner.
\newblock Connection probabilities for conformal loop ensembles.
\newblock {\em Comm. Math. Phys.}, 362(2):415--453, 2018.

\bibitem[PW19]{pw-msle-simple}
Eveliina Peltola and Hao Wu.
\newblock Global and local multiple {SLE}s for {$\kappa \leq 4$} and connection
  probabilities for level lines of {GFF}.
\newblock {\em Comm. Math. Phys.}, 366(2):469--536, 2019.

\bibitem[RS05]{rs2005basic}
Steffen Rohde and Oded Schramm.
\newblock Basic properties of {SLE}.
\newblock {\em Ann. of Math. (2)}, 161(2):883--924, 2005.

\bibitem[RY99]{ry-book}
Daniel Revuz and Marc Yor.
\newblock {\em Continuous martingales and {B}rownian motion}, volume 293 of
  {\em Grundlehren der mathematischen Wissenschaften [Fundamental Principles of
  Mathematical Sciences]}.
\newblock Springer-Verlag, Berlin, third edition, 1999.

\bibitem[Sch00]{s2000sle}
Oded Schramm.
\newblock Scaling limits of loop-erased random walks and uniform spanning
  trees.
\newblock {\em Israel J. Math.}, 118:221--288, 2000.

\bibitem[She07]{s2007gff}
Scott Sheffield.
\newblock Gaussian free fields for mathematicians.
\newblock {\em Probab. Theory Related Fields}, 139(3-4):521--541, 2007.

\bibitem[She09]{s2009cle}
Scott Sheffield.
\newblock Exploration trees and conformal loop ensembles.
\newblock {\em Duke Math. J.}, 147(1):79--129, 2009.

\bibitem[She16a]{s2016zipper}
Scott Sheffield.
\newblock Conformal weldings of random surfaces: {SLE} and the quantum gravity
  zipper.
\newblock {\em Ann. Probab.}, 44(5):3474--3545, 2016.

\bibitem[She16b]{s2016hc}
Scott Sheffield.
\newblock Quantum gravity and inventory accumulation.
\newblock {\em Ann. Probab.}, 44(6):3804--3848, 2016.

\bibitem[Smi01]{s2001percolation}
Stanislav Smirnov.
\newblock Critical percolation in the plane: conformal invariance, {C}ardy's
  formula, scaling limits.
\newblock {\em C. R. Acad. Sci. Paris S\'{e}r. I Math.}, 333(3):239--244, 2001.

\bibitem[Smi10]{s2010ising}
Stanislav Smirnov.
\newblock Conformal invariance in random cluster models. {I}. {H}olomorphic
  fermions in the {I}sing model.
\newblock {\em Ann. of Math. (2)}, 172(2):1435--1467, 2010.

\bibitem[SS09]{ss2009dgff}
Oded Schramm and Scott Sheffield.
\newblock Contour lines of the two-dimensional discrete {G}aussian free field.
\newblock {\em Acta Math.}, 202(1):21--137, 2009.

\bibitem[SS13]{ss2013continuumcontour}
Oded Schramm and Scott Sheffield.
\newblock A contour line of the continuum {G}aussian free field.
\newblock {\em Probab. Theory Related Fields}, 157(1-2):47--80, 2013.

\bibitem[SSW09]{ssw2009radii}
Oded Schramm, Scott Sheffield, and David~B. Wilson.
\newblock Conformal radii for conformal loop ensembles.
\newblock {\em Comm. Math. Phys.}, 288(1):43--53, 2009.

\bibitem[SW05]{sw2005coordinate}
Oded Schramm and David~B. Wilson.
\newblock S{LE} coordinate changes.
\newblock {\em New York J. Math.}, 11:659--669, 2005.

\bibitem[SW12]{sw2012cle}
Scott Sheffield and Wendelin Werner.
\newblock Conformal loop ensembles: the {M}arkovian characterization and the
  loop-soup construction.
\newblock {\em Ann. of Math. (2)}, 176(3):1827--1917, 2012.

\bibitem[Var01]{varadhan-probability}
S.~R.~S. Varadhan.
\newblock {\em Probability theory}, volume~7 of {\em Courant Lecture Notes in
  Mathematics}.
\newblock New York University, Courant Institute of Mathematical Sciences, New
  York; American Mathematical Society, Providence, RI, 2001.

\bibitem[Wu20]{wu-hypergeometric-sle}
Hao Wu.
\newblock Hypergeometric {SLE}: conformal {M}arkov characterization and
  applications.
\newblock {\em Comm. Math. Phys.}, 374(2):433--484, 2020.

\bibitem[WW13]{ww-cle-explorations}
Wendelin Werner and Hao Wu.
\newblock On conformally invariant {CLE} explorations.
\newblock {\em Comm. Math. Phys.}, 320(3):637--661, 2013.

\bibitem[Zha24]{zhan-msle}
Dapeng Zhan.
\newblock Existence and uniqueness of nonsimple multiple {SLE}.
\newblock {\em J. Stat. Phys.}, 191(8):Paper No. 101, 15, 2024.

\end{thebibliography}

\end{document}